\documentclass[12pt,leqno]{amsart}

\usepackage{verbatim}

\usepackage{amsfonts}

\usepackage{amsthm}

\usepackage{amssymb}

\usepackage{amsmath}

\usepackage{mathabx}

\usepackage{enumerate}

\usepackage[all]{xy}

\usepackage{graphicx}

\usepackage[pagebackref,colorlinks,linkcolor=red,citecolor=blue,urlcolor=blue,hypertexnames=true]{hyperref}

\usepackage[pagebackref,hypertexnames=true]{hyperref}

\newcommand{\N}{\mathbb{N}}

\newcommand{\R}{\mathbb{R}}

\newcommand{\Q}{\mathbb{Q}}

\newcommand{\C}{\mathbb{C}}

\newcommand{\supp}{\text{supp}}

\newcommand{\pr}{\text{Prob}}

\newcommand{\e}{{}_Y\mathcal{E}_A}

\newcommand{\eg}{\mathcal{E}_{A\rtimes_r\Gamma}}

\newcommand{\Manoa}{M\=anoa}

\newcommand{\Hawaii}{Hawai\kern.05em`\kern.05em\relax i}

\numberwithin{equation}{section}

\theoremstyle{plain}
\newtheorem{theorem}{Theorem}[section]
\newtheorem{lemma}[theorem]{Lemma}
\newtheorem{corollary}[theorem]{Corollary}
\newtheorem{proposition}[theorem]{Proposition}

\newtheorem{definition-theorem}[theorem]{Definition / Theorem}

\newtheorem*{conjecture*}{Conjecture}
\newtheorem*{theorem*}{Theorem}
\newtheorem*{proposition*}{Proposition}

\newtheorem*{proposition*1}{Proposition~\ref{basecase}}
\newtheorem*{proposition*2}{Proposition~\ref{indstep}}

\theoremstyle{definition}
\newtheorem{definition}[theorem]{Definition}

\theoremstyle{remark}
\newtheorem{remark}[theorem]{Remark}
\newtheorem{question}[theorem]{Question}

\newtheorem*{example*}{Example}  
\newtheorem*{remark*}{Remark}

\title[Dynamical complexity and $K$-theory]{Dynamical complexity and controlled operator $K$-theory}

\author{Erik Guentner}
\address{University of \Hawaii\ at \Manoa.
}
\email{erik@math.hawaii.edu}
\author{Rufus Willett}
\address{University of \Hawaii\ at \Manoa.
}
\email{rufus@math.hawaii.edu}
\author{Guoliang Yu}
\address{Texas A\&M University 
 and Shanghai Center for Mathematical Sciences.}
 \email{guoliangyu@math.tamu.edu}

\begin{document}

\maketitle

\begin{abstract}
In this paper, we introduce a property of topological dynamical systems that we call finite dynamical complexity.  For systems with this property,  one can in principle compute the $K$-theory of the associated crossed product $C^*$-algebra by splitting it up into simpler pieces and using the methods of controlled $K$-theory.  The main part of the paper illustrates this idea by giving a new proof of the Baum-Connes conjecture for actions with finite dynamical complexity.

We have tried to keep the paper as self-contained as possible: we hope the main part will be accessible to someone with the equivalent of a first course in operator $K$-theory.  In particular, we do not assume prior knowledge of controlled $K$-theory, and use a new and concrete model for the Baum-Connes conjecture with coefficients that requires no bivariant $K$-theory to set up. 
\end{abstract}

\tableofcontents

\section{Introduction}\label{intro sec}

Throughout this paper, the symbol `$\Gamma\lefttorightarrow X$' will mean that $\Gamma$ is a countable discrete group, $X$ is a compact Hausdorff space, and $\Gamma$ acts on $X$ by homeomorphisms.  We will abbreviate this information by saying that `$\Gamma\lefttorightarrow X$ is an action'.   

Our work here is based around a new property for actions, which we call \emph{finite dynamical complexity}.  This is partly inspired by the geometric notion of \emph{finite decomposition complexity}, introduced by the first and third authors together with Tessera \cite{Guentner:2009tg}, and by the notion of \emph{dynamic asymptotic dimension}, which was introduced by the current authors in earlier work \cite{Guentner:2014aa}.  

The precise definition of finite dynamical complexity requires groupoid language to state; rather than get into details here, we just give an idea and refer the reader to Definition \ref{gpd fdc} (see also Definition \ref{real gpd fdc}) for the precise version.   Roughly, then, we say an action $\Gamma\lefttorightarrow X$ \emph{decomposes} over some collection $\mathcal{C}$ of `dynamical systems' (more precisely, \'{e}tale groupoids) if it can be `locally cut into two pieces', each of which is in $\mathcal{C}$.  The action $\Gamma\lefttorightarrow X$ has \emph{finite dynamical complexity} if it is contained in the smallest class $\mathcal{C}$ that is: closed under decompositions; and contains all dynamical systems that are `essentially finite' (more precisely, have compact closure inside the ambient \'{e}tale groupoid).  

This definition allows the $K$-theory groups $K_*(C(X)\rtimes_r\Gamma)$ to be computed, at least in principle: the idea is that one can often compute the $K$-theory of essentially finite pieces using classical (`commutative') techniques from algebraic topology and the theory of type $I$ $C^*$-algebras, then use generalized (`controlled' \cite{Oyono-Oyono:2011fk}) Mayer-Vietoris arguments to reassemble this into the $K$-theory of the whole crossed product $C(X)\rtimes_r\Gamma$.  Strikingly, the $C^*$-algebras $C(X)\rtimes_r\Gamma$ to which these methods apply are often simple; this one has no hope of applying classical Mayer-Vietoris techniques, as these require the presence of non-trivial ideals.  This strategy works particularly well when one is trying to show vanishing of certain $K$-theory groups.

To illustrate this strategy for computing $K$-theory, the main part of this paper applies the idea above to the Baum-Connes conjecture for an action $\Gamma\lefttorightarrow X$ with finite dynamical complexity.  This conjecture (a special case of the Baum-Connes conjecture for $\Gamma$ with coefficients \cite{Baum:1994pr}) posits that a particular \emph{assembly map}
\begin{equation}\label{bc map}
\mu:KK^\text{top}_*(\Gamma,C(X))\to K_*(C(X)\rtimes_r\Gamma)
\end{equation}
is an isomorphism; here the domain is a topologically defined group associated to the action, and the codomain is the operator $K$-theory of the reduced crossed product $C^*$-algebra $C(X)\rtimes_r\Gamma$, an analytically defined object.  The existence of such an isomorphism relating two quite different aspects of the action has important consequences for both: for example, it has consequences for Novikov-type conjectures associated to $\Gamma$, and implies the existence of various tools to better understand the $K$-theory of the crossed product.

The main part of the paper proves the following result, which is inspired in part by the third author's work \cite{Yu:1998wj} on the coarse Baum-Connes conjecture for spaces with finite asymptotic dimension, the first and third authors' work with Tessera on the bounded Borel conjecture for spaces with finite decomposition complexity \cite{Guentner:2009tg}, and the work of all three authors on dynamic asymptotic dimension \cite{Guentner:2014aa}.

\begin{theorem}\label{act the}
Let $\Gamma\lefttorightarrow X$ be an action with finite dynamical complexity, where $X$ is a second countable compact space.  Then the Baum-Connes conjecture holds for $\Gamma$ with coefficients in $C(X)$.
\end{theorem}

Our proof of Theorem \ref{act the} starts by replacing the problem of proving that $\mu$ as in line \eqref{bc map} above is an isomorphism with the problem of showing that the $K$-theory of a certain \emph{obstruction $C^*$-algebra} $A(\Gamma\lefttorightarrow X)$ vanishes.  For this obstruction $C^*$-algebra one can apply the strategy for computing $K$-theory outlined above, and show that it is indeed zero.

The hypotheses of Theorem \ref{act the} cover many interesting actions: we refer the reader to our companion paper \cite{Guentner:2014aa}, particularly the introduction, for a discussion of the case of finite dynamic asymptotic dimension.  We suspect that finite dynamic dimension implies finite dynamical complexity, but did not seriously pursue at this stage.\\

Relating the above to the literature, we should note that Theorem \ref{act the} is implied by earlier work: indeed, it follows from work of Tu \cite{Tu:1999bq} on the Baum-Connes conjecture for amenable groupoids and the fact (Theorem \ref{fdc vs fdc} below) that finite dynamical complexity of a groupoid implies amenability.  Some of the key tools in Tu's proof are the Dirac-dual-Dirac method of Kasparov \cite{Kasparov:1988dw}, the work of Higson and Kasparov on the Baum-Connes conjecture for a-T-menable groups \cite{Higson:2001eb}, and Le Gall's groupoid-equivariant bivariant $K$-theory \cite{Le-Gall:1999aa}.  As already hinted at above, our proof is quite different: it gives a direct way of understanding the group $K_*(C(X)\rtimes_r\Gamma)$ that uses much less machinery.

Our motivations for giving a new proof of Theorem \ref{act the} are fourfold.  First, we want to illustrate the controlled methods for computing $K$-theory as already mentioned above.  Second, we want to make the Baum-Connes theory more direct so that it might be adapted to computations of $K$-theory for much more general classes of $C^*$-algebras with an eye on the K\"{u}nneth theorem \cite{Oyono-Oyono:2016qd} and UCT problem\footnote{Since this paper was written, a program of this sort has been carried out  in \cite{Willett:2021te}.}.  Third, we want to make techniques from the Baum-Connes theory more algebraic, so as to highlight and strengthen interactions with the Farrell-Jones theory in algebraic topology \cite{Bartels:2012fk,Bartels:2012fu}.  Fourth, the proof is fairly self-contained: we have tried to make it accessible to a reader who has understood an introduction to $C^*$-algebra $K$-theory at the level of \cite{Rordam:2000mz} or \cite{Wegge-Olsen:1993kx}.  

On this fourth point, we hope that the paper can be read without prior  knowledge of Baum-Connes theory, groupoids, controlled $K$-theory, or even crossed product $C^*$-algebras.  This makes the proof more elementary than most existing proofs of special cases of the Baum-Connes conjecture.  In order to do this, we introduce a direct geometric / combinatorial reformulation of the Baum-Connes conjecture; we show that it agrees with the traditional one using Kasparov's $KK$-theory \cite{Baum:1994pr} in an appendix.  Using these elementary methods also has the advantage that Theorem \ref{act the} remains true (correctly interpreted) if one drops the second countability assumption on $X$.

To conclude this introduction, we should note that this paper only just starts the study of finite dynamical complexity and its relation to other properties.  We ask several open questions in \ref{fdc vs amen q} through \ref{fdc c* q} below: some of these might be difficult, but we suspect some are quite accessible.

\subsubsection*{Outline of the paper}

Section \ref{ass sec} builds a concrete model for the Baum-Connes assembly map for an action $\Gamma\lefttorightarrow X$ based on the localization algebras used by the third author to give a model for the coarse Baum-Connes assembly map \cite{Yu:1997kb}.  Section \ref{gpd sec} introduces some language from groupoid theory that will be useful in carrying out various decompositions, and which is crucial for the definition of finite dynamical complexity given at the end of that section.  Section \ref{conk sec} gives a self-contained description of the controlled $K$-theory groups we will need for the proof, following work of the third author \cite{Yu:1998wj}, and of Oyono-Oyono in collaboration with the third author \cite{Oyono-Oyono:2011fk}.    Section \ref{strat sec} lays out the strategy for proving Theorem \ref{act the}, which is based roughly on the proof of the coarse Baum-Connes conjecture for spaces with finite asymptotic dimension of the third author \cite{Yu:1998wj}, and the work of the first and third authors with Tessera \cite{Guentner:2009tg} on the stable Borel conjecture; in particular, it reduces the proof to two technical propositions.  These technical propositions are established in Sections \ref{base sec} and \ref{ind sec}.  There are two appendices, which require a bit more background of the reader.  Appendix \ref{fdc sec} relates our finite dynamical complexity to finite decomposition complexity in the sense of \cite{Guentner:2009tg}, and to topological amenability \cite{Anantharaman-Delaroche:2000mw} as well as asking some questions; this requires some background in the general theory of \'{e}tale groupoids.  Appendix \ref{bc sec} identifies our model for the Baum-Connes assembly map with one of the standard models using $KK$-theory; as such, it requires some background in equivariant $KK$-theory.  The appendices are included to connect what we have done here to preexisting theory, and are certainly not needed to understand the rest of the paper.

\subsubsection*{Acknowledgments}

The authors would like to thank the University of \Hawaii~at \Manoa, the Shanghai Center for Mathematical Sciences, and Texas A\&M University for their hospitality during some of the work on this project.  We would also like to thank Yeong Chyuan Chung, Hao Guo, Yuhei Suzuki, and Santiago Vega for some helpful comments on earlier versions.  We are particularly grateful to Suzuki and Vega who (independently) spotted a fairly serious mistake in an earlier version of this manuscript.  This mistake, and the fact that the current version seems more natural and general, were our motivations for switching the focus of the paper from dynamic asymptotic dimension (as was the case in earlier drafts) to finite dynamical complexity.

The first author was partially supported by the Simons Foundation ($\#245398$).  The second author was partially supported by the US NSF (DMS-$1401126$ and DMS-$1564281$).  The third author was partially supported by the US NSF (DMS-$1362772$ and DMS-$1564281$) and the NSFC (NSFC$11420101001$).  The authors gratefully acknowledge this support.

\section{Assembly maps}\label{ass sec}

Throughout this section, $\Gamma\lefttorightarrow X$ is an action in our usual sense: $\Gamma$ is a countable discrete group, $X$ is a compact Hausdorff topological space, and $\Gamma$ acts on $X$ by homeomorphisms.   Our goal in this section is to develop a concrete and elementary model for the Baum-Connes assembly map for $\Gamma$ with coefficients in $C(X)$.  The construction is modeled on the localization algebra approach to the coarse Baum-Connes conjecture of the third author \cite{Yu:1997kb}.

We will assume throughout that $\Gamma$ is equipped with a proper length function and the associated right invariant metric as in the next definition.

\begin{definition}\label{length fun}
A \emph{(proper) length function} on a group $\Gamma$ is a function $|\cdot|:\Gamma\to\N$ that satisfies the following conditions:
\begin{enumerate}[(i)]
\item $|g|=0$ if and only if $g=e$ (where $e$ is the identity element of $\Gamma$);
\item $|gh|\leq |g|+|h|$;
\item $|g^{-1}|=|g|$;
\item for any $r\geq 0$, $\{g\in \Gamma\mid |g|\leq r\}$ is finite.
\end{enumerate}
Associated to such a length function is a metric defined by $d(g,h):=|gh^{-1}|$.
\end{definition}
Note that the right action of the group on itself preserves the metric associated to a length function, and that if the length function is proper, then for any $r\geq 0$ there is a uniform bound on the cardinality of all $r$-balls.  Examples of length functions and associated metrics are provided by word metrics associated to finite generating sets, when such exist.  Length functions always exist on a countable group, whether or not it is finitely generated, and the metrics they define are unique up to `coarse' (large-scale) equivalence: see for example \cite[Proposition 2.3.3]{Willett:2009rt}.  Our use of a length function will only depend on the coarse equivalence class, and therefore fixing one makes no real difference.

\begin{definition}\label{rips}
Let $s\geq 0$. The \emph{Rips complex of $\Gamma$ at scale $s$}, denoted $P_s(\Gamma)$, is the simplicial complex with vertex set $\Gamma$, and where a finite subset $E$ of $\Gamma$ spans a simplex if and only if 
\begin{equation}\label{ripscon}
d(g,h)\leq s \text{ for all } g,h\in E.
\end{equation}
Points $z\in P_s(\Gamma)$ can be written as formal linear combinations 
$$
z=\sum_{g\in \Gamma}t_gg,
$$
where each $t_g$ is in $[0,1]$ and $\sum_{g\in \Gamma}t_g=1$.  We equip the space $P_s(\Gamma)$ with the \emph{$\ell^1$-metric}
$$
d\Big( \sum_{g\in \Gamma}t_gg,\sum_{g\in \Gamma}s_gg\Big)=\sum_{g\in \Gamma}|t_g-s_g|.
$$
The \emph{barycentric coordinates} on $P_s(\Gamma)$ are the continuous functions $(t_g:P_s(\Gamma)\to [0,1])_{g\in \Gamma}$ indexed by $g\in \Gamma$ that are uniquely determined by the condition 
$$
z=\sum_{g\in \Gamma}t_g(z)g
$$
for all $z\in P_s(\Gamma)$.
\end{definition}

Using the fact that balls of radius $s$ in $\Gamma$ are (uniformly) finite, it is straightforward to check that $P_s(\Gamma)$ is finite dimensional and locally compact.  Note also that the right translation action of $\Gamma$ on itself extends to a right action of $\Gamma$ on $P_s(\Gamma)$ by (isometric) simplicial automorphisms.

We now want to build Hilbert spaces and $C^*$-algebras connected to both the large scale geometry of $\Gamma$ (called `Roe algebras') and the topological structure of $P_s(\Gamma)$ (called `localization algebras').  

\begin{definition}\label{zf def}
For each $s\geq 0$, define 
$$
Z_s:=\Bigg\{\sum_{g\in \Gamma} t_gg\in P_s(\Gamma)~\Big|~t_g\in \Q \text{ for all } g\in \Gamma\Bigg\}.
$$
\end{definition}
\noindent Note that $Z_s$ is $\Gamma$-invariant, so the $\Gamma$-action on $P_s(\Gamma)$ induces a (right) action on each $Z_s$.  

Let $\ell^2(Z_s)$ denote the Hilbert space of square-summable functions on $Z_s$.  Let $\ell^2(X)$ denote the Hilbert space of square-summable functions on $X$.  Fix also a separable infinite dimensional Hilbert space $H$, and define
$$
H_s:=\ell^2(Z_s)\otimes \ell^2(X)\otimes H\otimes \ell^2(\Gamma).
$$
Equip $H_s$ with the unitary $\Gamma$ action defined for $g\in \Gamma$ by
$$
u_g:\delta_z\otimes \delta_x\otimes \eta\otimes \delta_h \mapsto \delta_{zg^{-1}}\otimes \delta_{gx}\otimes \eta\otimes \delta_{gh},
$$
where $z\in Z_s$, $x\in X$, $\eta\in H$ and $h\in \Gamma$.  When convenient, we will use the canonical identification 
\begin{equation}\label{hf2}
H_s=\ell^2(Z_s\times X,H\otimes \ell^2(\Gamma))
\end{equation}
of $H_s$ with the Hilbert space of square-summable functions from $Z_s\times X$ to $H\otimes \ell^2(\Gamma)$.

Note that if $s_0\leq s$ then $P_{s_0}(\Gamma)$ identifies equivariantly and isometrically with a subcomplex of $P_{s}(\Gamma)$, and moreover $Z_{s_0}\subseteq Z_{s}$.  Hence there are canonical equivariant isometric inclusions
\begin{equation}\label{inc}
H_{s_0}\subseteq H_{s}
\end{equation}
which we will use many times below.

Write now $\mathcal{K}_\Gamma$ for the compact operators on $H\otimes \ell^2(\Gamma)$ equipped with the $\Gamma$ action by $*$-automorphisms that is induced by the tensor product of the trivial action on $H$ and the (left) regular representation on $\ell^2(\Gamma)$.   Equip the $C^*$-algebra $C(X,\mathcal{K}_\Gamma)=C(X)\otimes \mathcal{K}_\Gamma$ of continuous functions from $X$ to $\mathcal{K}_\Gamma$ with the diagonal action of $\Gamma$ induced by the actions on $C(X)$ and $\mathcal{K}_\Gamma$.  Note that the natural faithful representation of $C(X,\mathcal{K}_\Gamma)$ on $\ell^2(X)\otimes H\otimes \ell^2(\Gamma)$ is covariant for the unitary representation defined by tensoring the natural permutation action on $\ell^2(X)$, the trivial representation on $H$, and the regular representation on $\ell^2(\Gamma)$.

\begin{definition}\label{roe}
Let $T$ be a bounded operator on $H_s$.  We may think of $T$ as a $Z_s\times Z_s$-indexed matrix $T=(T_{y,z})$, where each entry is a bounded operator on $\ell^2(X)\otimes H\otimes \ell^2(\Gamma)$.  We will be interested in the following properties of such $T$.
\begin{enumerate}[(i)]
\item $T$ is \emph{$\Gamma$-invariant} if $u_gTu_g^*=T$ for all $g\in \Gamma$.
\item The \emph{Rips-propagation} of $T$ is the number 
$$
r=\sup\{d_{P_s(\Gamma)}(y,z)~|~T_{y,z}\neq 0\}.
$$
\item The \emph{$\Gamma$-propagation} of $T$ is the supremum (possibly infinite) of the set
$$
\{d_\Gamma(g,h)\mid T_{y,z}\neq 0 \text{ for some }y,z\in Z_s \text{ with } t_g(y)\neq 0,~ t_h(z)\neq 0\},
$$
where $t_g,t_h:P_s(\Gamma)\to [0,1]$ are the barycentric coordinates associated to $g$ and $h$ as in Definition \ref{rips} above.
\item $T$ is \emph{$X$-locally compact} if for all $y,z\in Z_s$, the operator $T_{y,z}$ is in the $C^*$-subalgebra $C(X,\mathcal{K}_\Gamma)$ of the bounded operators on $\ell^2(X)\otimes H\otimes \ell^2(\Gamma)$, and moreover if for any compact subset $K$ of $P_s(\Gamma)$, the set 
$$
\{(y,z)\in K\times K\mid T_{y,z}\neq 0\}
$$
is finite.
\end{enumerate}
\end{definition}

\begin{remark}
(This remark relates the definition above to earlier ones in the literature, and can be safely ignored by readers who do not know the earlier material).   Traditionally in this area (see for example \cite{Yu:1997kb}) one defines a suitable length metric on each Rips complex $P_s(\Gamma)$, and uses only the propagation defined relative to this metric.  We have `decoupled' the notion of propagation into the $\Gamma$-propagation (relevant only for large-scale structure) and the Rips-propagation (relevant only for small scale structure).  The reason for doing this is that in the traditional approach the metric depends on the Rips parameter $s$, and it is convenient for us to have metrics that do \emph{not} vary in this way.
\end{remark}

\begin{definition}\label{roe alg}
Let $\C[\Gamma\lefttorightarrow X;s]$ denote the collection of all $\Gamma$-invariant, $X$-locally compact operators on $H_s$, with finite $\Gamma$-propagation.  It is straightforward to check that $\C[\Gamma\lefttorightarrow X;s]$ is a $*$-algebra of bounded operators.

Let $C^*(\Gamma\lefttorightarrow X;s)$ denote the closure of $\C[\Gamma\lefttorightarrow X;s]$ with respect to the operator norm.  The $C^*$-algebra $C^*(\Gamma\lefttorightarrow X;s)$ is called the \emph{Roe algebra} of $\Gamma\lefttorightarrow X$ at scale $s$.
\end{definition}

We will always consider $\C[\Gamma\lefttorightarrow X;s]$ and $C^*(\Gamma\lefttorightarrow X;s)$ as concretely represented on $H_s$, and equipped with the corresponding operator norm.   Elements of $C^*(\Gamma\lefttorightarrow X;s)$ can be thought of as matrices $(T_{y,z})_{y,z\in Z_s}$ with entries continuous equivariant functions $T_{y,z}:X\to \mathcal{K}_\Gamma$ in a way that is compatible with the $*$-algebra structure; we will frequently use this description below.

\begin{remark}\label{cp alg rem}
The Roe algebras $C^*(\Gamma\lefttorightarrow X;s)$ are all isomorphic to the stabilization of the reduced crossed product $C^*$-algebra $C(X)\rtimes_r\Gamma$.  We do not need this remark in the main body of the paper, but include it now as it may help orient some readers.  See Appendix \ref{bc sec} for a proof.
\end{remark}

Note that the Rips-propagation is not relevant to the definition of the Roe algebras; it is, however, used in a crucial way in the next definition.

\begin{definition}\label{local}
Let $\C_L[\Gamma\lefttorightarrow X;s]$ denote the $*$-algebra of all bounded, uniformly continuous functions 
$$
a:[0,\infty)\to \C[\Gamma\lefttorightarrow X;s]
$$
such that the $\Gamma$-propagation of $a(t)$ is uniformly finite as $t$ varies, and so that the Rips-propagation of $a(t)$ tends to zero as $t$ tends to infinity.

Let $C^*_L(\Gamma\lefttorightarrow X;s)$ denote the completion of $\C_L[\Gamma\lefttorightarrow X;s]$ for the norm 
$$
\|a\|:=\sup_{t\in [0,\infty)}\|a(t)\|_{C^*(\Gamma\lefttorightarrow X;s)}.
$$
The $C^*$-algebra $C^*_L(\Gamma\lefttorightarrow X;s)$ is called the \emph{localization algebra} of $\Gamma\lefttorightarrow X$ at scale $s$.  
\end{definition}

Note that an element $a$ of $C^*_L(\Gamma\lefttorightarrow X;s)$ comes from a unique bounded, uniformly continuous functions 
$$
a:[0,\infty)\to C^*(\Gamma\lefttorightarrow X;s)
$$
(satisfying some additional properties).  We will think of $\C_L[\Gamma\lefttorightarrow X;s]$ and $C^*_L(\Gamma\lefttorightarrow X;s)$ as concretely represented on the Hilbert space $L^2[0,\infty)\otimes H_s$ in the obvious way.  

Finally in this section, we come to the definition of the assembly map.   First, let 
\begin{equation}\label{s ass}
\epsilon_0:K_*(C^*_L(\Gamma\lefttorightarrow X;s))\to K_*(C^*(\Gamma\lefttorightarrow X;s))
\end{equation}
denote the evaluation-at-zero $*$-homomorphism $a\mapsto a(0)$.   Say that $s_0\leq s$.  Then the isometric equivariant inclusion $H_{s_0}\subseteq  H_s$ from line \eqref{inc} above induces isometric inclusions 
$$
C^*(\Gamma\lefttorightarrow X;s_0)\subseteq C^*(\Gamma\lefttorightarrow X;s)
$$ 
and 
$$
\quad C^*_L(\Gamma\lefttorightarrow X;s_0)\subseteq C^*_L(\Gamma\lefttorightarrow X;s)
$$ 
of $C^*$-algebras, and thus to directed systems $(C^*(\Gamma\lefttorightarrow X;s)\big)_{s\geq 0}$ and $\big(C^*_L(\Gamma\lefttorightarrow X;s)\big)_{s\geq 0}$ of $C^*$-algebra inclusions.  Moreover the evaluation-at-zero maps from line \eqref{s ass} are clearly compatible with these inclusions, whence we may make the following definition.

\begin{definition}\label{ass}
The \emph{assembly map} for $\Gamma\lefttorightarrow X$ is the direct limit 
$$
\epsilon_0:\lim_{s\to\infty}K_*(C^*_L(\Gamma\lefttorightarrow X;s))\to \lim_{s\to\infty}K_*(C^*(\Gamma\lefttorightarrow X;s))
$$ 
of the evaluation-at-zero maps from line \eqref{s ass} above.
\end{definition}

\begin{remark}\label{bc id rem}
This map identifies naturally with the Baum-Connes assembly map for $\Gamma$ with coefficients in $C(X)$, whence the name.  Analogously to Remark \ref{cp alg rem} above, we do not need this fact in the main body of the paper, but include it now in case it is helpful for some readers.  See  Appendix \ref{bc sec} for a proof.
\end{remark}

Our main goal in this paper is to prove the following theorem.

\begin{theorem}\label{main}
Let $\Gamma \lefttorightarrow X$ be an action with finite dynamical complexity.  Then the assembly map is an isomorphism.
\end{theorem}

\begin{remark}\label{sc rem}
Thanks to the results of Appendix \ref{bc sec}, this is the same result as Theorem \ref{act the} from the introduction, although without the assumption that $X$ is second countable. The only reason for including second countability of $X$ in the statement of Theorem \ref{act the} is to avoid technical complications that arise in the traditional statement of the Baum-Connes conjecture with coefficients in a non-separable $C^*$-algebra.  Assuming separability would make no difference for the proof of Theorem \ref{main}, however, so we omit the assumption here.
\end{remark}

In order to prove this theorem, it is convenient to shift attention to an `obstruction group'.

\begin{definition}\label{obstruction}
Let $C^*_{L,0}(\Gamma\lefttorightarrow X;s)$ denote the $C^*$-subalgebra of $C^*_L(\Gamma\lefttorightarrow X;s)$ consisting of functions $a$ such that $a(0)=0$.  The $C^*$-algebra $C^*_{L,0}(\Gamma\lefttorightarrow X;s)$ is called the \emph{obstruction algebra} of $\Gamma\lefttorightarrow X$ at scale $s$. 
\end{definition}

There is clearly a directed system $\big(C^*_{L,0}(\Gamma\lefttorightarrow X;s)\big)_{s\geq 0}$ of obstruction algebras.  The following straightforward lemma explains the terminology `obstruction algebra': the $K$-theory of these algebras obstructs isomorphism of the assembly map.

\begin{lemma}\label{oblem}
The assembly map of Definition \ref{ass} is an isomorphism if and only if 
$$
\lim_{s\to\infty}K_*(C^*_{L,0}(\Gamma\lefttorightarrow X;s))=0.
$$
\end{lemma}

\begin{proof}
The short exact sequence 
$$
\xymatrix{ 0 \ar[r] & C^*_{L,0}(\Gamma\lefttorightarrow X;s) \ar[r] & C^*_{L}(\Gamma\lefttorightarrow X;s) \ar[r] & 
C^*(\Gamma\lefttorightarrow X;s) \ar[r] & 0}
$$
gives rise to six term exact sequence in $K$-theory.  The lemma follows from this, continuity of $K$-theory under direct limits, and the fact that a direct limit of an exact sequence of abelian groups is exact.
\end{proof}

Thus in order to prove Theorem \ref{main}, it suffices to prove that the group in the statement of Lemma \ref{oblem} vanishes whenever $\Gamma\lefttorightarrow X$ has finite dynamical complexity.  The proof of this occupies the next five sections. We spend the next two sections developing machinery: in Section \ref{gpd sec} we introduce some convenient language from groupoid theory, use this to define useful subalgebras of the obstruction algebras, and introduce finite dynamical complexity; and in Section \ref{conk sec} we introduce controlled $K$-theory, which gives us extra flexibility when performing $K$-theoretic computations. 

Having built this machinery, Section \ref{strat sec} sketches out the strategy of the proof of Theorem \ref{main}: the basic idea is to first use a homotopy invariance result to show that the $K$-theory of the obstruction algebra associated to an `essentially finite' dynamical subsystem of $\Gamma\lefttorightarrow X$ vanishes; and then to use a Mayer-Vietoris type argument to show that the class of dynamical subsystems of $\Gamma\lefttorightarrow X$ for which the $K$-theory of the obstruction algebra vanishes is closed under decomposability.  The proofs of the homotopy invariance result, and of the Mayer-Vietoris type argument are somewhat technical, and are carried out in Sections \ref{base sec} and \ref{ind sec} respectively.

\section{Groupoids and decompositions}\label{gpd sec}

Our goal in this section is to show how `subgroupoids' of the action $\Gamma\lefttorightarrow X$ give rise to $C^*$-subalgebras of the Roe algebras and localization algebras of Section \ref{ass sec}.  We try to keep the exposition self-contained: in particular, we do not assume that the reader has any background in the theory of locally compact groupoids or their $C^*$-algebras.

Throughout this section $\Gamma\lefttorightarrow X$ is an action in our usual sense: $\Gamma$ is a countable discrete group and $X$ is a compact space equipped with an action of a $\Gamma$ by homeomorphisms.  We also fix a (proper) length function on $\Gamma$ and associated right-invariant metric as in Definition \ref{length fun}.  

\begin{definition}\label{transgpd}
The \emph{transformation groupoid} associated to $\Gamma\lefttorightarrow X$, denoted $\Gamma\ltimes X$, is defined as follows.  As a set, $\Gamma\ltimes X$ is equal to
$$
\{(gx,g,x)\in X\times \Gamma\times X\mid g\in \Gamma,x\in X\}.
$$
The set $\Gamma\ltimes X$ is equipped with the topology such that the (bijective) projection $\Gamma\ltimes X\to \Gamma\times X$ onto the second and third factors is a homeomorphism.

The topological space $\Gamma\ltimes X$ is equipped with the following additional structure. 
\begin{enumerate}[(i)]
\item A pair $\big((hy,h,y),(gx,g,x)\big)$ of elements of $\Gamma\ltimes X$ is \emph{composable} if $y=gx$.  If the pair is composable, their \emph{product} is defined by 
$$
(hgx,h,gx)(gx,g,x):=(hgx,hg,x).
$$
\item The \emph{inverse} of an element $(gx,g,x)$ of $\Gamma\ltimes X$ is defined by
$$
(gx,g,x)^{-1}:=(x,g^{-1},gx).
$$
\item The \emph{units} are the elements of the open and closed subspace 
$$
(\Gamma\ltimes X)^{(0)}:=\{(x,e,x)\in \Gamma\ltimes X\mid x\in X\}
$$
of $\Gamma\ltimes X$.
\end{enumerate}
\end{definition}

We can now discuss the algebra of supports of elements in the Roe algebra.  For this, recall from Definitions \ref{roe} and \ref{roe alg} that we can think of an operator $T$ in $C^*(\Gamma\lefttorightarrow X;s)$ as a matrix $(T_{y,z})_{y,z\in Z_s}$ indexed by $Z_s$ with entries continuous functions $T_{y,z}:X\to \mathcal{K}_\Gamma$.  

\begin{definition}\label{t supp}
Let $s\geq 0$, and let $P_s(\Gamma)$ be the associated Rips complex with barycentric coordinates $t_g:P_s(\Gamma)\to [0,1]$ as in  Definition \ref{rips}.  Define the \emph{support} of  $z\in P_s(\Gamma)$ to be the finite subset 
$$
\text{supp}(z):=\{g\in \Gamma\mid t_g(z)\neq 0\}
$$
of $\Gamma$.  Define the \emph{support} of $T\in C^*(\Gamma\lefttorightarrow X;s)$ to be the subset
\begin{align*}
\text{supp}(T):=\left\{\begin{array}{l|l}(gx,gh^{-1},hx) &   \text{ there are }  y,z\in Z_s \text{ with } T_{y,z}(x)\neq 0 \\ \quad\in \Gamma\ltimes X& \text{ and } g\in \text{supp}(y),~h\in \text{supp}(z)\end{array}\right\}
\end{align*}
of $\Gamma\ltimes X$. 
\end{definition}

Note that for $T\in \C[\Gamma\lefttorightarrow X;s]$, the $\Gamma$-propagation of $T$ as in Definition \ref{roe} is equal to the largest value of $|k|$ such that $(x,k,y)$ appears in $\text{supp}(T)$ for some $x,y\in X$.  

Supports of operators behave well with respect to composition and adjoints; this is the content of the next lemma.  To state it, note that the groupoid operations on $\Gamma\ltimes X$ extend to subsets in natural ways: if $A,B\subseteq \Gamma\ltimes X$ then we define
$$
A^{-1}:=\{a^{-1}\mid a\in A\}
$$
and 
$$
AB:=\{ab\mid a\in A,b\in B\text{ and } (a,b) \text{ composable}\}.
$$

\begin{lemma}\label{supp alg}
Let $S,T\in C^*(\Gamma\lefttorightarrow X;s)$.  Then 
$$
\text{supp}(S^*)=\text{supp}(S)^{-1} \quad \text{and}\quad \text{supp}(ST)\subseteq \text{supp}(S)\text{supp}(T).
$$
\end{lemma}

\begin{proof}
As the adjoint of $S$ has matrix entries $(S^*)_{y,z}=S_{z,y}^*$, the statement about adjoints is clear.  To see the statement about multiplication, say $T,S\in C^*(\Gamma\lefttorightarrow X;s)$, and $(gx,gh^{-1},hx)$ is a point in the support of $TS$.  Then there are $y,z\in Z_s$ such that $g\in \text{supp}(y)$, $h\in \text{supp}(z)$, and $(TS)_{y,z}(x)\neq 0$.  Hence there is $w\in Z_s$ with $T_{y,w}(x)\neq 0$, and $S_{w,z}(x)\neq 0$.  Say $k$ is any point in $\text{supp}(w)$, so we must have that $(gx,gk^{-1},kx)$ is in the support of $T$ and $(kx,kh^{-1},hx)$ is in the support of $S$.  As 
$$
(gx,gh^{-1},hx)=(gx,gk^{-1},kx) (kx,kh^{-1},hx),
$$
this shows that the support of $TS$ is contained in the product of the supports of $T$ and $S$.
\end{proof}

The lemma implies that subspaces of $\Gamma\ltimes X$ that are closed under the groupoid operations will give rise to $*$-subalgebras of $\C[\Gamma\lefttorightarrow X;s]$.  The relevant algebraic  notion is that of a subgroupoid as in the next definition.

\begin{definition}\label{subgpd}
Let $\Gamma\ltimes X$ be the transformation groupoid associated to the action $\Gamma\lefttorightarrow X$.  A \emph{subgroupoid} of $\Gamma\ltimes X$ is a subset $G$ of $\Gamma\ltimes X$ closed under the operations in the following sense.
\begin{enumerate}[(i)]
\item If $(hgx,h,gx)$ and $(gx,g,x)$ are in $G$, then so is the composition $(hgx,hg,x)$.
\item If $(gx,g,x)$ is in $G$, then so is its inverse $(x,g^{-1},gx)$.
\item If $(gx,g,x)$ is in $G$, then so are the units $(x,e,x)$ and $(gx,e,gx)$.
\end{enumerate}
A subgroupoid is equipped with the subspace topology inherited from $\Gamma\ltimes X$. 
\end{definition}

The following lemma is now almost clear.

\begin{lemma}\label{gsubalgebra}
Let $G$ be an open subgroupoid of $\Gamma\ltimes X$.  Define $\C[G;s]$ to be the subspace of $\C[\Gamma\lefttorightarrow X;s]$ consisting of all operators $T$ with support contained in a compact subset of $G$.  Then $\C[G;s]$ is a $*$-subalgebra of $\C[\Gamma\lefttorightarrow X;s]$.
\end{lemma}

\begin{proof}
Lemma \ref{supp alg} gives most of this: the only remaining point to check is that a product of two relatively compact subsets\footnote{Recall that a subset of a topological space is \emph{relatively compact} if its closure is compact.} of $G$ is relatively compact, which we leave it to the reader (or see for example \cite[Lemma 5.2]{Guentner:2014aa} for a more general statement and proof).
\end{proof}

Using this lemma, the following definitions make sense.

\begin{definition}\label{smallob}
Let $G$ be an open subgroupoid of $\Gamma\ltimes X$.  Let $\C_{L}[G;s]$ denote the $*$-subalgebra of $\C_L[\Gamma\lefttorightarrow X;s]$ (see Definition \ref{local} above) consisting of all $a$ such that $\bigcup_t \text{supp}(a(t))$ has compact closure inside $G$.  Let $\C_{L,0}[G;s]$ denote the ideal of $\C_{L}[G;s]$ consisting of functions such that $a(0)=0$.  

Let $C^*(G;s)$, $C^*_L(G;s)$, and $C^*_{L,0}(G;s)$ denote the closures of $\C[G;s]$, $\C_{L}[G;s]$, and $\C_{L,0}[G;s]$ inside $C^*(\Gamma\lefttorightarrow X;s)$, $C^*_L(\Gamma\lefttorightarrow X;s)$, and $C^*_{L,0}(\Gamma\lefttorightarrow X;s)$ respectively.
\end{definition}
\noindent Note that operators of finite $\Gamma$-propagation always have support contained in some compact subset of $\Gamma\ltimes X$.  Hence if $G=\Gamma\ltimes X$ then $C^*(G;s)$ is just $C^*(\Gamma\lefttorightarrow X;s)$, and similarly for the localization and obstruction algebras.

\begin{remark}\label{gpd alg rem}
(This remark may be safely ignored by readers who do not have any background in groupoids and the associated $C^*$-algebras.)  Analogously to Remark \ref{cp alg rem}, for any open subgroupoid $G$ of $\Gamma\ltimes X$, the $C^*$-algebra $C^*(G;s)$ is Morita equivalent to the reduced groupoid $C^*$-algebra $C^*_r(G)$; this makes sense, as an open subgroupoid of $\Gamma\ltimes X$ is \'{e}tale so has a canonical Haar system given by counting measures.   We only include this remark as it might help to orient some readers; we will not use it in any way, or prove it.  
\end{remark}

Our next goal in this section is to construct filtrations on these $C^*$-algebras in the sense of the definition below, and discuss how they interact with the subalgebras coming from groupoids above.  

\begin{definition}\label{filt}
A \emph{filtration} on a $C^*$-algebra $A$ is a collection of self-adjoint subspaces $(A_r)_{r\geq 0}$ of $A$ indexed by the non-negative real numbers that satisfies the following properties:
\begin{enumerate}[(i)]
\item if $r_1\leq r_2$, then $A_{r_1}\subseteq A_{r_2}$;
\item for all $r_1,r_2$, we have $A_{r_1}\cdot A_{r_2}\subseteq A_{r_1+r_2}$;
\item the union $\bigcup_{r\geq 0}A_r$ is dense in $A$.
\end{enumerate}
\end{definition}

As the case of the obstruction $C^*$-algebras will be particularly important for us, we introduce some shorthand notation for that case. 

\begin{definition}\label{main ob}
For an open subgroupoid $G$ of $\Gamma\ltimes X$ and $s\geq 0$, define $A^s(G)$ to be the $C^*$-algebra $C^*_{L,0}(G;s)$.  For each $r\in [0,\infty)$, define 
$$
A^s(G)_r:=\{a\in \C_{L,0}[G;s] \mid a(t) \text{ has $\Gamma$-propagation at most $r$ for all $t$}\},
$$
and define $A^s(G)_\infty:=\C_{L,0}[G;s]$.  In the special case that $G=\Gamma\ltimes X$, we omit it from the notation and just write $A^s$ and $A^s_r$.
\end{definition}
\noindent When convenient, we will consider all these $C^*$-algebras as faithfully represented on the Hilbert space
$$
L^2([0,\infty),H_s)=L^2[0,\infty)\otimes \ell^2(Z_s)\otimes \ell^2(X)\otimes H\otimes \ell^2(\Gamma).
$$

\begin{lemma}\label{filt lem}
For any open subgroupoid $G$ of $\Gamma\ltimes X$ and any $s\geq 0$, the subsets $(A^s(G)_r)_{r\geq 0}$ define a filtration of $A^s(G)$ in the sense of Definition \ref{filt}.
\end{lemma}

\begin{proof} 
An element $a\in A^s(G)$ is in $A^s(G)_r$ if and only if it is in the dense $*$-subalgebra $\C_{L,0}[G;s]$ and if whenever $(gx,g,x)$ is in $\text{supp}(a(t))$ for some $t$, we have $|g|\leq r$.  The filtration properties follows directly from this, the facts that $|gh|\leq |g||h|$ and $|g^{-1}|=|g|$, and Lemma \ref{supp alg}.
\end{proof}

Our next goal in this section is to discuss what happens when we take products of elements from $A^s_r$ and $A^s(G)$ for some $G$ and $r$.  Note first that analogously to the case of subgroups, one may build a subgroupoid generated by some $S\subseteq \Gamma\ltimes X$ by iteratively closing under taking compositions, inverses, and units in the sense of parts (i)-(iii) of Definition \ref{subgpd} above.  From this, it is straightforward to check that if $S$ is an open subset of $\Gamma\ltimes X$, then the subgroupoid it generates is also open: see \cite[Lemma 5.2]{Guentner:2014aa} for a proof.  

\begin{definition}\label{extgpd}
Let $r\geq 0$ and $G$ be an open subgroupoid of $\Gamma\ltimes X$, and $H$ be an open subgroupoid of $G$.  The \emph{expansion} of $H$ by $r$ relative to $G$, denoted $H^{+r}$, is the open subgroupoid of $\Gamma\ltimes X$ generated by 
$$
H\cup \{(gx,g,x)\in G\mid |g|\leq r,~x\in H^{(0)}\}.
$$
\end{definition}
Note that the expansion $H^{+r}$ depends on the ambient groupoid $G$; we do not include $G$ in the notation, however, to avoid clutter, and as which groupoid we are working inside should be clear from context.

We now have two basic lemmas.

\begin{lemma}\label{filt ext}
Let $G$ be an open subgroupoid of $\Gamma\ltimes X$, and $H$ an open subgroupoid of $G$.  Let $r,s\geq 0$.  Then 
$$
A^s(H)\cdot A^s_r(G)\cup A^s_r(G)\cdot A^s(H)\subseteq A^s(H^{+r}).
$$
\end{lemma}

\begin{proof}
Immediate from Lemma \ref{supp alg}.
\end{proof}

\begin{lemma}\label{twice}
For any $r_1,r_2\geq 0$ and open subgroupoid $H$ of an open subgroupoid $G$ of $\Gamma\ltimes X$, we have that $(H^{+r_1})^{+r_2}\subseteq H^{+(r_1+r_2)}$.
\end{lemma}

\begin{proof}
Clearly $H^{+r_1}\subseteq H^{+(r_1+r_2)}$, so it suffices to show that if $x$ is in the unit space of $H^{+r_1}$ and $|g|\leq r_2$ is such that $(gx,g,x)$ is in $G$, then $(gx,g,x)$ is in $H^{+(r_1+r_2)}$.  Indeed, as $x$ is in the unit space of $H^{+r_1}$ there is $(hx,h,x)\in G$ with $h\in \Gamma$, $|h|\leq r_1$, and $hx\in H^{(0)}$.  Hence $(gx,gh^{-1},hx)$ and $(hx,h,x)$ are in $H^{+(r_1+r_2)}$ and we have 
$$
(gx,g,x)=(gx,gh^{-1},hx)(hx,h,x)\in H^{+(r_1+r_2)}
$$
as required.
\end{proof}

Finally, we conclude this section with the definition of finite dynamical complexity and a basic lemma about the property.

\begin{definition}\label{gpd fdc}
Let $\Gamma\lefttorightarrow X$ be an action, let $G$ be an open subgroupoid of $\Gamma\ltimes X$, and let $\mathcal{C}$ be a set of open subgroupoids of $\Gamma\ltimes X$.  We say that $G$ is \emph{decomposable} over $\mathcal{C}$ if for all $r\geq 0$ there exists an open cover $G^{(0)}=U_0\cup U_1$ of the unit space of $G$ such that for each $i\in \{0,1\}$ the subgroupoid of $G$ generated by 
$$
\{(gx,g,x)\in G\mid x\in U_i,~|g|\leq r\}
$$
(i.e.\ the expansion $U_i^{+r}$ relative to $G$ of Definition \ref{extgpd}) is in $\mathcal{C}$. 

An open subgroupoid of $\Gamma\ltimes X$ (for example, $\Gamma\ltimes X$ itself) has \emph{finite dynamical complexity} if it is contained in the smallest set $\mathcal{D}$ of open subgroupoids of $\Gamma\ltimes X$ that: contains all relatively compact open subgroupoids; and is closed under decomposability\footnote{More precisely, we mean `upwards closed': if $G$ decomposes over $\mathcal{D}$, then $G$ is in $\mathcal{D}$.}.
\end{definition}

We will need a slight variation of this definition.

\begin{definition}\label{strong fdc}
Say that an open subgroupoid $G$ of $\Gamma\ltimes X$ is \emph{strongly decomposable} over a set $\mathcal{C}$ of open subgroupoids of $\Gamma\ltimes X$ if for all $r\geq 0$ there exists an open cover $G^{(0)}=U_0\cup U_1$ of the unit space of $G$ such that for each $i\in \{0,1\}$, if $G_i$ is the subgroupoid of $G$ generated by 
$$
\{(gx,g,x)\in G\mid x\in U_i,~|g|\leq r\},
$$
then $G_i^{+r}$ (with expansion taken relative to $G$) is in $\mathcal{C}$.  Let $\mathcal{D}_s$ be the smallest class of open subgroupoids of $G$ that contains the relatively compact open subgroupoids, and that is closed under strong decomposability. 
\end{definition}

The following lemma records two basic properties of finite dynamical complexity that we will need later.

\begin{lemma}\label{fdc lem}
With notation as above:
\begin{enumerate}[(i)]
\item if $G$ is an open subgroupoid of $\Gamma\ltimes X$ in the class $\mathcal{D}$ (respectively $\mathcal{D}_s$), then all open subgroupoids of $G$ are in $\mathcal{D}$ (respectively $\mathcal{D}_s$); 
\item we have $\mathcal{D}=\mathcal{D}_s$.
\end{enumerate}
\end{lemma}

\begin{proof}
For part (i), we just look at the case of $\mathcal{D}$; the case of $\mathcal{D}_s$ is similar.  Let $\mathcal{D}'$ be the set of all open subgroupoids of $\Gamma\ltimes X$, all of whose open subgroupoids are in $\mathcal{D}$.  Clearly $\mathcal{D}'\subseteq \mathcal{D}$, and $\mathcal{D}'$ contains all open subgroupoids with compact closure.  To complete the proof of (i), it suffices to show that $\mathcal{D}'$ is closed under decomposability.  Say then $G$ is an open subgroupoid of $\Gamma\ltimes X$ that decomposes over $\mathcal{D}'$.   Say $H$ is an open subgroupoid of $G$ and $r\geq 0$.  Let $\{U_0,U_1\}$ be an open cover of $G^{(0)}$ such that for each $i\in \{0,1\}$, the subgroupoid $G_i$ of $G$ generated by 
$$
\{(gx,g,x)\in G\mid x\in U_i,~|g|\leq r\}
$$
is in $\mathcal{D}'$.  Let $V_i=U_i\cap H^{(0)}$.  Then $\{V_0,V_1\}$ is an open cover of $H^{(0)}$ such that the (open) subgroupoid $H_i$ of $H$ generated by
$$
\{(gx,g,x)\in H\mid x\in V_i,~|g|\leq r\}
$$
is contained in $G_i$.  As each $G_i$ is in $\mathcal{D}'$, this implies that each $H_i$ is in $\mathcal{D}$; in other words, $H$ decomposes over $\mathcal{D}$, and is thus in $\mathcal{D}$ as required.

For part (ii), it clearly suffices to prove that $\mathcal{D}$ is closed under strong decomposability, and that $\mathcal{D}_s$ is closed under decomposability.  For the former, say that $G$ strongly decomposes over $\mathcal{D}$.  Then for any $r\geq 0$, there is an open cover $\{U_0,U_1\}$ of $G^{(0)}$ such that if $G_i$ is generated by 
$$
\{(gx,g,x)\in G\mid x\in U_i,~|g|\leq r\},
$$
then $G_i^{+r}$ is in $\mathcal{D}$.  However, $G_i$ is an open subgroupoid of $G_i^{+r}$ whence is in $\mathcal{D}$ by part (i). Hence $G$ decomposes over $\mathcal{D}$, and thus is in $\mathcal{D}$ as required.

For the other case, say $G$ decomposes over $\mathcal{D}_s$ and let $r\geq 0$.  Then there is an open cover $G^{(0)}=U_0\cup U_1$ such that the subgroupoid $H_i$ of $G$ generated by 
$$
\{(gx,g,x)\in G \mid x\in U_i,~|g|\leq 2r\}
$$
is in $\mathcal{D}_s$.  We claim that if $G_i$ is the subgroupoid of $G$ generated by 
$$
\{(gx,g,x)\in G \mid x\in U_i,~|g|\leq r\},
$$
then $G_i^{+r}$ is an open subgroupoid of $H_i$; this will suffice to complete the proof by part (i).  Indeed, we have that $G_i^{+r}$ is generated by $G_i$ and 
\begin{equation}\label{exp gen}
\{(gx,g,x)\in G \mid x\in G_i^{(0)},~|g|\leq r\};
\end{equation}
as $G_i$ is clearly contained in $H_i$, it suffices to show that the latter set if in $H_i$.  Let then $(gx,g,x)$ be in the set in line \eqref{exp gen}. As $x\in G_i^{(0)}$, we have $x=ky$ for some $y\in U_i$, and $k\in \Gamma$ with $|k|\leq r$.  Hence we may rewrite 
$$
(gx,g,x)=(gky,g,ky)=(gky,gk,y)(y,k^{-1},ky);
$$
as $y\in U_i$ and $|gk|,|k^{-1}|\leq 2r$, the product on the right is in $H_i$ and we are done.
\end{proof}

\section{Controlled $K$-theory}\label{conk sec}

In this section we will introduce the main general tool needed for the proof of Theorem \ref{main}: controlled $K$-theory. 

Our treatment in this section is based on the detailed development given by Oyono-Oyono and the third author \cite{Oyono-Oyono:2011fk}.  Our controlled $K$-theory groups are, however, both more general in some ways, and more specific in others, than those of \cite{Oyono-Oyono:2011fk}. We thus try to keep our exposition self-contained; in particular, we assume no background beyond the basics of $C^*$-algebra $K$-theory as covered for example in \cite{Rordam:2000mz} or \cite{Wegge-Olsen:1993kx}.  Throughout, we provide references to \cite{Oyono-Oyono:2011fk} for comparison purposes, and the reader is encouraged to look there for a broader picture of the theory. 

We now define the controlled $K$-groups that we need.  Compare  \cite[Section 1.2]{Oyono-Oyono:2011fk} for the following definition.

\begin{definition}\label{qp}
Let $A$ be a $C^*$-algebra.  A \emph{quasi-projection} in $A$ is an element $p$ of $A$  such that $p=p^*$ and $\|p^2-p\|<1/8$.  If $S$ is a self-adjoint subspace of $A$, write $M_n(S)$ for the matrices in $M_n(A)$ with all entries coming from $S$, and $P^{1/8}_n(S)$ for the collection of quasi-projections in $M_n(S)$.  

Let $\chi=\chi_{(1/2,\infty]}$ be the characteristic function of $(1/2,\infty]$.  Then $\chi$ is continuous on the spectrum of any quasi-projection, and thus there is a well-defined map
$$
\kappa:P_n^{1/8}(S)\to P_n(A),\quad p\mapsto \chi(p),
$$
where $P_n(A)$ denotes the projections in $M_n(A)$.
\end{definition}

\begin{remark}\label{1/8 rem}
The choice of `$1/8$' in the above is not important: any number no larger than $1/4$ would do just as well.  In some arguments in controlled $K$-theory, it is useful to allow the bound on the `projection error' $\|p^2-p\|$ to change; for this reason in \cite[Section 1.2]{Oyono-Oyono:2011fk}, what we have called a quasi-projection would be called a \emph{$(1/8)$-projection}.   We do not need this extra flexibility, so it is more convenient to just fix an absolute error bound throughout.  Similar remarks apply to quasi-unitaries as introduced in Definition \ref{qu} below.
\end{remark}

If $A$ is a $C^*$-algebra we denote its unitization by $\widetilde{A}$.  For a natural number $n$, let $1_n$ denote the unit of $M_n(\widetilde{A})$.  For the following definition, compare \cite[Section 1, particularly Definition 1.12]{Oyono-Oyono:2011fk}.  

\begin{definition}\label{k0}
Let $A$ be a non-unital $C^*$-algebra, and let $S$ be a self-adjoint subspace of $A$.  Let $\widetilde{S}$ denote the subspace $S+\C1$ of $\widetilde{A}$.  

Using the inclusions
$$
P_n^{1/8}(\widetilde{S})\owns p\mapsto \begin{pmatrix} p & 0 \\ 0 & 0 \end{pmatrix}\in P_{n+1}^{1/8}(\widetilde{S})
$$
we may define the union
$$
P_\infty^{1/8}(\widetilde{S}):=\bigcup_{n=1}^\infty P_n^{1/8}(\widetilde{S}).
$$

Let $C([0,1],\widetilde{S})$ denote the self-adjoint subspace of the $C^*$-algebra $C([0,1],\widetilde{A})$ consisting of functions with values in $\widetilde{S}$.  Define an equivalence relation on $P^{1/8}_\infty(\widetilde{S})\times \N$ by $(p,m)\sim (q,n)$ if there exists a positive integer $k$ and an element $h$ of $P^{1/8}_\infty(C([0,1],\widetilde{S}))$ such that 
$$
h(0)=\begin{pmatrix} p & 0 \\ 0 & 1_{n+k}\end{pmatrix}~~\text{ and } ~~h(1)=\begin{pmatrix} q & 0 \\ 0 & 1_{m+k}\end{pmatrix}.
$$
For $(p,m)\in P_\infty^{1/8}(\widetilde{S})\times \N$, denote by $[p,m]$ its equivalence class under $\sim$.

Let now $\rho:M_n(\widetilde{S})\to M_n(\C)$ be the restriction to $M_n(\widetilde{S})$ of the map induced on matrices by the canonical unital $*$-homomorphism $\rho:\widetilde{A}\to \C$ with kernel $A$.  Finally, define
$$
K_0^{1/8}(S):=\{[p,m]\in \big(P_\infty^{1/8}(\widetilde{S})\times \N \big)\,/\sim~\mid~ \text{rank}(\kappa(\rho(p)))=m\}.
$$
The set $K_0^{1/8}(S)$ is equipped with an operation defined by
$$
[p,m]+[q,n]:=\Big[\begin{pmatrix} p & 0 \\ 0 & q \end{pmatrix}\,,\,m+n\Big].
$$
\end{definition}

\noindent Using standard arguments in $K$-theory, one sees that $K_0^{1/8}(S)$ is an abelian group with unit $[0,0]$: compare \cite[Lemmas 1.14 and 1.15]{Oyono-Oyono:2011fk}.\\

We now look at controlled $K_1$ groups.  Compare \cite[Section 1.2]{Oyono-Oyono:2011fk} for the following definition.

\begin{definition}\label{qu}
Let $A$ be a unital $C^*$-algebra.  A \emph{quasi-unitary} in $A$ is an element $u$ of $A$  such that $\|1-uu^*\|<1/8$ and $\|1-u^*u\|<1/8$.    If $S$ is a self-adjoint subspace of $A$ containing the unit, write $U^{1/8}_n(S)$ for the collection of quasi-unitaries in $M_n(S)$.  

Note that as $\|1-u^*u\|<1/8<1$, $u^*u$ is invertible whence there is a well-defined map
$$
\kappa:U_n^{1/8}(S)\to U_n(A),~~~u\mapsto u(u^*u)^{-1/2},
$$
where $U_n(A)$ denotes the unitaries in $M_n(A)$.
\end{definition}

For the following definition, compare \cite[Section 1, particularly Definition 1.12]{Oyono-Oyono:2011fk}.

\begin{definition}\label{k1}
Let $A$ be a non-unital $C^*$-algebra, and let $S$ be a self-adjoint subspace of $A$.  Let $\widetilde{A}$ denote the unitization of $A$ and let $\widetilde{S}$ be the subspace $S+\C1$ of $\widetilde{A}$.  

Using the inclusions
$$
U_n^{1/8}(\widetilde{S})\owns u\mapsto \begin{pmatrix} u & 0 \\ 0 & 1 \end{pmatrix}\in U_{n+1}^{1/8}(\widetilde{S})
$$
we may define the union
$$
U_\infty^{1/8}(\widetilde{S}):=\bigcup_{n=1}^\infty U_n^{1/8}(\widetilde{S}).
$$
Define an equivalence relation on $U^{1/8}_\infty(\widetilde{S})$ by $u\sim v$ if there exists an element $h$ of $U^{1/8}_\infty(C([0,1],\widetilde{S}))$ such that $h(0)=u$ and $h(1)=v$.  For $u\in U_\infty^{1/8}(\widetilde{S})$, denote by $[u]$ its equivalence class under $\sim$.

Finally define
$$
K_1^{1/8}(S):=U_\infty^{1/8}(\widetilde{S})~/\sim.
$$
The set $K_1^{1/8}(S)$ is equipped with the operation defined by
$$
[u]+[v]:=\begin{bmatrix} u & 0 \\ 0 & v\end{bmatrix}.
$$
\end{definition}

Using standard arguments in $K$-theory, one sees that the operation on $K_1^{1/8}(S)$ makes it into an abelian monoid\footnote{The usual arguments showing $K_1(A)$ has inverses do not apply as they involve multiplying elements together, and so potentially go outside $\widetilde{S}$.} with unit $[1]$: compare \cite[Lemmas 1.14 and 1.16]{Oyono-Oyono:2011fk}.

\begin{definition}\label{k*}
Let $A$ be a $C^*$-algebra and $S$ a self-adjoint subspace of $A$.  Define
$$
K_*^{1/8}(S):=K_0^{1/8}(S)\oplus K_1^{1/8}(S).
$$
The (graded, unital, abelian) semigroup $K_*^{1/8}(S)$ is called the \emph{controlled $K$-theory semigroup of $S$}.
\end{definition}

Note that $K_*^{1/8}(S)$ depends on the embedding of $S$ inside the ambient $C^*$-algebra $A$; which embedding is being used will always be obvious, however, so we omit it from the notation.

\begin{remark}\label{subspaces}
If $S\subseteq T$ are nested self-adjoint subspaces of a (non-unital) $C^*$-algebra $A$, then we may consider elements of matrix algebras over $\widetilde{S}$ as elements of matrix algebras over $\widetilde{T}$.  This clearly gives rise to a map on controlled $K$-theory
\begin{equation}\label{subs incl}
K_*^{1/8}(S)\to K_*^{1/8}(T)
\end{equation}
induced by the inclusion.  To avoid clutter, we will not use specific notation for these inclusion maps; sometimes we refer to them as `subspace-inclusion' maps.  We also sometimes abuse terminology and say something like `let $x$ and $y$ be elements of $K_*^{1/8}(S)$ that are equal in $K_*^{1/8}(T)$'; more precisely, this means that $x$ and $y$ are elements of $K_*^{1/8}(S)$ that go to the same element of $K_*^{1/8}(T)$ under the map in line \eqref{subs incl} above.
\end{remark}

For the next definition and lemma, which compares controlled $K$-theory to the usual $K$-theory groups of a $C^*$-algebra, compare \cite[Remark 1.18 and surrounding discussion]{Oyono-Oyono:2011fk}.  Let $K_*(A):=K_0(A)\oplus K_1(A)$ denote the usual (topological) $K$-theory group of a $C^*$-algebra $A$.

\begin{definition}\label{compare}
Let $A$ be a non-unital $C^*$-algebra, and let $S$ be a self-adjoint subspace of $A$.  Let $\kappa$ be one of the maps from Definition \ref{qp} and \ref{qu} (it will be obvious from context which is meant).  Define maps
\begin{align*}
c_0 & :K_0^{1/8}(S)\to K_0(A),~~~[p,m]\mapsto [\kappa(p)]-[1_m] \\
c_1 & :K_1^{1/8}(S)\to K_1(A),~~~[u]\mapsto [\kappa(u)] \\ 
c :=c_0\oplus c_1& :K_*^{1/8}(S)\to K_*(A).
\end{align*}
We call $c_0$, $c_1$, and $c$ the \emph{comparison maps}.
\end{definition}

Standard arguments in $C^*$-algebra $K$-theory show that $c_0$, $c_1$ and $c$ are well-defined unital semigroup homomorphisms. The map $c$ is `almost isomorphic' in the following sense.  

\begin{proposition}\label{comparelem}
Let $A$ be a $C^*$-algebra.  Let $(S_i)_{i\in I}$ be a collection of self-adjoint subspaces of $A$ such that the union $\bigcup_{i\in I} S_i$ is dense.  Then for any $x\in K_*(A)$ there exists $S_i$ and $y\in K_*^{1/8}(S_i)$ such that $c(y)=x$. 

Assume moreover that the collection $(S_i)_{i\in I}$ is directed for the partial order given by inclusion.  Then if $x,y\in K_*^{1/8}(S_i)$ are such that $c(x)=c(y)$ there exists $S_j$ containing $S_i$ such that $x$ and $y$ become equal in $K_*^{1/8}(S_j)$ in the sense of Remark \ref{subspaces}.

In particular, if $S$ is a dense self-adjoint subspace of $A$, then the comparison map $c:K_*^{1/8}(S)\to K_*(A)$ is an isomorphism.
\end{proposition} 

\begin{proof}
We just look at the case of $K_1$: the $K_0$ case is similar.  Let $[u]$ be a class in $K_1(A)$.  As $\bigcup_i S_i$ is dense in $A$, for any $\epsilon>0$ there is an $i\in I$ and $v$ in a matrix algebra over the unitization of $S_i$ such that $\|u-v\|<\epsilon$.  For $\epsilon$ suitably small, this implies that $v$ is a quasi-unitary, and moreover that $\|\kappa(v)-u\|<2$, whence $[\kappa(v)]=[u]$ in $K_1(A)$.  

The injectivity statement follows on applying a similar argument to homotopies.  Indeed, it suffices to show that any homotopy $h$ in $C([0,1],M_n(\widetilde{A}))$ can be approximated by a homotopy in $C([0,1],M_n(\widetilde{S_i}))$ for some $i$.  For this, note that density gives us elements $a_0,...,a_m$ with $a_k$ in some $M_n(\widetilde{S_{i_k}})$ such that the map $g:[0,1]\to A$ which sends $k/m$ to $a_k$ and linearly interpolates between these points is a good approximant to $h$.  The directedness assumption implies there is some $S_j$ that contains all of $S_{i_0},....,S_{i_m}$, so $g$ is in $C([0,1],M_n(\widetilde{S_j}))$.
\end{proof}

The most important examples of subspaces we will use (particularly in the context of Proposition \ref{comparelem}) come from filtrations as in Definition \ref{filt} (see Definition \ref{main ob} and Lemma \ref{filt lem} for examples).  It will be convenient to have some additional notation in the case when $A$ is filtered. 

\begin{definition}\label{conkfilt}
Let $A$ be a non-unital filtered $C^*$-algebra, and $S$ be a self-adjoint subspace of $A$.  For each $r\geq 0$ and $i\in \{0,1,*\}$, define
$$
K_i^{r,1/8}(S):=K_i^{1/8}(S\cap A_r).
$$
\end{definition} 
\noindent Note that in the case that $S=A$, our notation agrees with that of \cite[Section 1.3]{Oyono-Oyono:2011fk}.

\section{Strategy of proof of Theorem \ref{main}}\label{strat sec}

In this short section, we set out our strategy for the proof of Theorem \ref{main}.  Our goal is to reduce the proof to two technical propositions and show how finite dynamical complexity can be used as an input for these.  The two propositions will be proved in the next two sections.

Throughout this section, $\Gamma\lefttorightarrow X$ is an action as usual, and we use the shorthand notations of Definition \ref{main ob} for the obstruction $C^*$-algebras and their groupoid versions.

Here are the two technical propositions, which should be thought of as the base case and inductive step in the proof.  The names `homotopy invariance' and `Mayer-Vietoris' will be explained in later sections. 

\begin{proposition}[Homotopy invariance argument]\label{basecase}
Let $G$ be an open subgroupoid of $\Gamma\ltimes X$, and assume $s\geq 0$ is such that 
$$
G\subseteq \{(gx,g,x)\in \Gamma\ltimes X \mid |g|\leq s\}.
$$
Then $K_*(A^s(G))=0$.
\end{proposition}

\begin{proposition}[Mayer-Vietoris]\label{indstep}
Let $G$ be an open subgroupoid of $\Gamma\ltimes X$ that is in the class $\mathcal{D}$ of Definition \ref{gpd fdc}, and let $r_0,s_0\geq 0$.  Then there is $s\geq \max\{r_0,s_0\}$ depending on $r_0$, $s_0$ and $G$ such that the subspace-inclusion map \textup{(}cf.\ Remark \ref{subspaces}\textup{)}
$$
K_*^{r_0,1/8}(A^{s_0}(G)) \to K_*^{s,1/8}(A^s(G))
$$
is the zero map.
\end{proposition}

\begin{proof}[Proof of Theorem \ref{main}]
We need to show that for any $s_0\geq 0$ and any class $x\in K_*(A^{s_0})$ there is $s\geq s_0$ such that the subspace-inclusion map
$$
K_*(A^{s_0})\to K_*(A^{s})
$$
sends $x$ to zero.  Proposition \ref{comparelem} implies that there is $r_0\geq 0$ such that $x$ is in the image of the comparison map $c:K_*^{r_0,1/8}(A^{s_0})\to K_*(A^{s_0})$ from Definition \ref{compare}.   Proposition \ref{indstep} applied to $G=\Gamma\ltimes X$ implies that there is $s\geq \max\{r_0,s_0\}$ such that the subspace-inclusion map
$$
K_*^{r_0,1/8}(A^{s_0}) \to K_*^{s,1/8}(A^s)
$$
is zero.  Consider the diagram
$$
\xymatrix{  K_*(A^{s_0}) \ar[r]  & K_*(A^{s}) \\
K_*^{r_0,1/8}(A^{s_0}) \ar[r] \ar[u]^-c &  K_*^{s,1/8}(A^s) \ar[u]^-{c}~ ,} 
$$
where the two horizontal arrows are induced by subspace inclusions, and the two vertical arrows are comparison maps; it is clear from the definition of the comparison maps that it commutes.  We have that the element $x\in K_*(A^{s_0})$ is in the image of the left comparison map $c$, and that the bottom horizontal map is zero.  Hence the image of $x$ under the top horizontal map is zero as claimed.
\end{proof}

\section{Homotopy invariance}\label{base sec}

In this section, we prove Proposition \ref{basecase}, which we repeat below for the reader's convenience.

\begin{proposition*1}
Let $G$ be an open subgroupoid of $\Gamma\ltimes X$, and assume $s\geq 0$ is such that 
$$
G\subseteq \{(gx,g,x)\in \Gamma\ltimes X \mid |g|\leq s\}.
$$
Then $K_*(A^s(G))=0$.
\end{proposition*1}

The proof is based on a technique of the third author: see for example \cite[Lemma 4.8]{Yu:1998wj}.   The $K$-theory of the localization algebra is an (equivariant) generalized homology theory in an appropriate sense, and the point of the proof is to show that this homology theory is homotopy invariant.  Having done this, the condition on $s$ in the statement implies that if $\{(g_1x,g,x),...,(g_nx,g_n,x)\}$ are elements of $G$ for some $x\in X$, then $\{g_1,...,g_n\}$ spans a simplex in $P_s(\Gamma)$; hence the relevant space becomes contractible in an appropriate sense, and so the result follows from homotopy invariance. 

To try to make the proof more palatable, we will separate it into two parts.  The first is purely $K$-theoretic: it gives a sufficient condition for two $*$-homomorphisms to induce the same map on $K$-theory.  The second part uses the underlying dynamics to build an input to this $K$-theoretic machine, and completes the proof.  

\subsection*{$K$-theoretic part}

We start the $K$-theoretic part with three basic $K$-theory lemmas; there are well-known, but we include proofs where we could not find a good reference.  Recall first that if $A$ is a $C^*$-algebra, represented faithfully and non-degenerately on a Hilbert space $H$, then the \emph{multiplier algebra} of $A$ is 
$$
M(A):=\{b\in \mathcal{B}(H) \mid ba,~ab\in A \text{ for all } a\in A\};
$$
it is a $C^*$-algebra.  The \emph{strict topology} on $M(A)$ is the topology generated by the seminorms
$$
a\mapsto \|ba\|,\quad a\mapsto \|ab\|
$$
as $b$ ranges over $A$.  The multiplier algebra and strict topology do not depend on the choice of $H$ up to canonical isomorphism: see for example \cite[Section 3.12]{Pedersen:1979zr}.

\begin{lemma}\label{isom con}
Let $\alpha:A\to C$ be a $*$-homomorphism of $C^*$-algebras, and $v$ a partial isometry in the multiplier algebra of $C$ such that $\alpha(a)v^*v=\alpha(a)$ for all $a\in A$.  Then the map
$$
a\mapsto v\alpha(a)v^*
$$
is a $*$-homomorphism from $A$ to $C$, and induces the same map as $\alpha$ on the level of $K$-theory.
\end{lemma}

\begin{proof}
See for example \cite[Lemma 2 in Section 3]{Higson:1993th}.
\end{proof}

\begin{lemma}\label{double}
Let $I$ be an ideal in a unital $C^*$-algebra $C$, and define the \emph{double} of $C$ along $I$ to be the $C^*$-algebra
$$
D:=\{(c_1,c_2)\in C\oplus C\mid c_1-c_2\in I\}.
$$
Assume that $C$ has trivial $K$-theory.  Then the inclusion $\iota:I\to D$ defined by $c\mapsto (c,0)$ induces an isomorphism on $K$-theory, and the inclusion $\kappa:I\to D$ defined by $c\mapsto (c,c)$ induces the zero map on $K$-theory. 
\end{lemma}

\begin{proof}
Note that $\iota(I)$ is an ideal in $D$, and $D/\iota(I)$ is isomorphic to $C$ via the second coordinate projection.  Hence $\iota$ is an isomorphism by the six-term exact sequence.  On the other hand, $\kappa$ factors through the inclusion $C\to D$ defined by the same formula $c\mapsto (c,c)$, and thus $\kappa_*=0$ on $K$-theory as $K_*(C)=0$.
\end{proof}

\begin{lemma}\label{orth}
Say $\alpha,\beta:C\to D$ are $*$-homomorphisms between $C^*$-algebras with orthogonal images (this means that $\alpha(c_1)\beta(c_2)=0$ for all $c_1,c_2\in C$).  

Then $\alpha+\beta$ is a $*$-homomorphism from $C$ to $D$, and as maps on $K$-theory $\alpha_*+\beta_*=(\alpha+\beta)_*$.
\end{lemma}

\begin{proof}
Orthogonality of the images of $\alpha$ and $\beta$ directly implies that $\alpha+\beta$ is a $*$-homomorphism.  For $t\in [0,\pi/2]$, consider the map $\gamma_t:C\to M_2(D)$ defined by
$$
c\mapsto \begin{pmatrix} \alpha(c) & 0 \\ 0 & 0 \end{pmatrix}+\begin{pmatrix} \cos(t) & \sin(t) \\ -\sin(t) & \cos(t) \end{pmatrix}\begin{pmatrix} 0 & 0 \\ 0 & \beta(c) \end{pmatrix}\begin{pmatrix} \cos(t) & -\sin(t) \\ \sin(t) & \cos(t) \end{pmatrix}.
$$
One directly checks that this is a homotopy of $*$-homomorphisms.  Moreover, identifying $K_*(D)$ with $K_*(M_2(D))$ in the canonical way via the top-left-corner inclusion $D\to M_2(D)$, it is straightforward to check that $(\gamma_0)_*=\alpha_*+\beta_*$ and $(\gamma_1)_*=(\alpha+\beta)_*$.
\end{proof}

Before getting to the main result, we need one more preliminary.  First, a definition.

\begin{definition}\label{stab str}
Let $A$ be a $C^*$-algebra.  A \emph{stability structure} for $A$ consists of a sequence of isometries $(u_n)_{n=0}^\infty$ in $M(A)$ and a topology $\tau$ on $M(A)$ such that multiplication is continuous on norm-bounded sets\footnote{For example, the strict topology could be used here, but we will need something a little different.} with the properties that:
\begin{enumerate}[(i)]
\item $u_n^*u_m=0$ for $n\neq m$;
\item for all $a\in A$, 
$$
\sum_{n=0}^\infty u_nau_n^*
$$
$\tau$-converges to an element of $M(A)$;
\item there is an isometry $v\in M(A)$ such that 
$$
v\Big(\sum_{n=0}^\infty u_nau_n^*\Big)v^*=\sum_{n=1}^\infty u_nau_n^*
$$
\end{enumerate}
\end{definition}
\noindent Note that a stable $C^*$-algebra has a stability structure in a natural way: indeed, if $H$ is a separable infinite-dimensional Hilbert space, fix isometries $u_{n}:H\to H$ with mutually orthogonal ranges such that $\sum_n u_n$ converges strongly to the identity; then $u_n$ naturally defines a multiplier of $A\otimes \mathcal{K}(H)$, which acts by $u_n(a\otimes k)=a\otimes u_nk$ on elementary tensors (and similarly for multiplication on the right); then the isometries $(u_n)$ together with the strict topology $\tau$ define a stability structure.  This is the motivation for the terminology.

\begin{lemma}\label{mult sum}
Let $A$ have a stability structure as in Definition \ref{stab str}.  Then $K_*(M(A))=0$.
\end{lemma}

\begin{proof}
Using continuity of multiplication on bounded sets for the $\tau$ topology, 
$$
\mu:M(A)\to M(A),\quad a\mapsto \sum_{n=0}^\infty u_nau_n^*
$$
is a $*$-homomorphism.  Let $v$ be the isometry appearing in Definition \ref{stab str}.  Using Lemma \ref{isom con}, $\mu$ induces the same map on $K$-theory as the map $a\mapsto v\mu(a)v^*$, i.e.\ as 
$$
\mu^{+1}(a):=\sum_{n=1}^\infty u_nau_n^*.
$$ 
Let $\mu^0(a)=u_0au_0^*$.  Then the images of $\mu^0$ and $\mu^{+1}$ are orthogonal, and clearly $\mu^0+\mu^{+1}=\mu$, whence by Lemma \ref{orth}
$$
\mu_*=\mu^0_*+\mu^{+1}_*=\mu^0_*+\mu_*
$$
as maps on $K$-theory.  Hence $\mu^0_*=0$.  However, by Lemma \ref{isom con} again, $\mu^0_*$ is the identity map on $K_*(M(A))$.
\end{proof}

\begin{definition}\label{st equiv}
Let $A$ be a $C^*$-algebra which is faithfully and non-degenerately represented on a Hilbert space $H$, and which is equipped with a stability structure $(u_n)_{n=1}^\infty$ and $\tau$ as in Lemma \ref{mult sum}.  

Let $v_0$ and $v_\infty$ be isometries on $H$ that conjugate $A$ to itself.   We say that $v_0$ and $v_\infty$ are \emph{stably equivalent} if there are in addition isometries $(v_n)_{n=1}^\infty$ on $H$ that conjugate $A$ into itself and that satisfy:
\begin{enumerate}[(i)]
\item \label{new con} for any $a\in A$, the sum
$$
\sum_{n=0}^\infty u_nv_nav_n^*u_n^*
$$
$\tau$-converges to an element of $M(A)$;
\item \label{new con 2} for all $0\leq n \leq \infty$, $v_{n+1}v_{n}^*$ is in $M(A)$ (where $\infty+1=\infty$) and the sums 
$$
\sum_{n=0}^\infty u_nv_{n+1} v_n^*u_n^*\quad \text{and}\quad \sum_{n=0}^\infty u_nv_\infty v_\infty^*u_n^*
$$
$\tau$-converge to elements of $M(A)$;
\item \label{diffs0} for any $a$ in $A$, the difference
$$
\sum_{n=0}^\infty u_nv_nav_n^*u_n^*-\sum_{n=0}^\infty u_nv_\infty av_\infty^*u_n^*
$$
of elements of $M(A)$ is in $A$;
\item \label{diffs} for all $a\in A$ 
$$
\sum_{n=0}^\infty u_n a (v_\infty v_\infty^* - v_{n+1}v_n^*)u_n^*\quad \text{and}\quad \sum_{n=0}^\infty u_n (v_\infty v_\infty^* - v_{n+1}v_n^*) a u_n^*
$$
are in $A$.
\end{enumerate}
\end{definition}
For readers who know the terminology, compare condition \eqref{diffs0} above to the definition of a quasi-morphism in the sense of Cuntz \cite[Section 17.6]{Blackadar:1998yq}.

Here finally is the main $K$-theoretic ingredient we need.

\begin{proposition}\label{al e s}
Let $A$ be $C^*$-algebra faithfully represented on a Hilbert space $H$ equipped with a stability structure $(u_n)$ and $\tau$.  Let $v_0$ and $v_\infty$ be stably equivalent isometries for this representation.  Then the homomorphisms 
$$
\phi_0,\phi_\infty:A\to A
$$
induced by conjugation by $v_0$ and $v_\infty$ induce the same map on $K$-theory.
\end{proposition}

\begin{proof}
Let
$$
D:=\{(a,b)\in M(A)\oplus M(A)\mid a-b\in A\}
$$
be the double of $M(A)$ along $A$ as in Lemma \ref{double}.  Note that Lemma \ref{mult sum} implies that $K_*(M(A))=0$, so we may apply the conclusions of Lemma \ref{double} to $D$.  Let 
$$
C=\Big\{(c,d)\in D\mid d=\sum_{n=0}^\infty u_nv_\infty a v_\infty^*u_n^* \text{ for some } a\in A\Big\},
$$
which is a $C^*$-subalgebra of $D$.  Define also
$$
w_1:=\sum_{n=0}^\infty u_n v_{n+1}v_n^*u_n^*,\quad w_2:=\sum_{n=0}^\infty u_nv_\infty v_\infty^* u_n^* 
$$
(which are elements of $M(A)$ by condition \eqref{new con 2} in Definition \ref{st equiv}) and set $w:=(w_1,w_2)\in M(A)\oplus M(A)$.  We claim that $w$ is actually in the multiplier algebra of $C$.

Indeed, if $(c,d)$ is in $C$, then $dw_2=w_2d=d$, so it suffices to show that $cw_1-d$ and $w_1c-d$ are in $A$; we focus on $w_1c-d$, the other case being similar.  We have
$$
w_1c-d=w_1(c-d)+(w_1d-d),
$$
whence as $c-d\in A$ and $w_1\in M(A)$, it suffices to show that $w_1d-d$ is in $A$.  There exists $a\in A$ with
\begin{align*}
w_1d-d &=\sum_{n=0}^\infty u_n (v_{n+1}v_n^*v_\infty a v_\infty^*-v_\infty av_\infty^*)u_n^* \\
&=\sum_{n=0}^\infty u_n(v_{n+1}v_n^*-v_\infty v_\infty^*)v_\infty av_\infty^*u_n^*,
\end{align*}
and this is in $A$ by condition \eqref{diffs} of Definition \ref{st equiv}, completing the proof of the claim.

Now, provisionally define $*$-homomorphisms 
$$
\alpha,\beta:A\to C
$$
by the formulas
$$
\alpha(a):=\Bigg( \sum_{n=0}^\infty u_nv_nav_n^*u_n^*~,~\sum_{n=0}^\infty u_nv_\infty av_\infty^*u_n^*\Bigg)
$$
and
$$
\beta(a):=\Bigg( \sum_{n=0}^\infty u_nv_{n+1}av_{n+1}^*u_n^*~,~\sum_{n=0}^\infty u_nv_\infty av_\infty^*u_n^*\Bigg).
$$ 
It is clear from our assumptions that $\alpha:A\to C$ is a homomorphism.  That $\beta$ is a homomorphism and has image in $C$ follows as $w$ is in the multiplier algebra of $C$, and as $w\alpha(a)w^*=\beta(a)$ for all $a\in A$.  Moreover, a direct computation gives that $\alpha(a)w^*w=\alpha(a)$, whence $\alpha$ and $\beta$ induce the same map $K_*(A)\to K_*(C)$ by Lemma \ref{isom con}.  Post-composing with the map $K_*(C)\to K_*(D)$ induced by the inclusion of $C$ into $D$, it follows that $\alpha$ and $\beta$ induce the same map $K_*(A)\to K_*(D)$.  

Let now $v\in M(A)$ be the isometry with the property in Definition \ref{stab str}.  Then $(v,v)$ is a multiplier of $D$; conjugating by $(v,v)$ and applying Lemma \ref{isom con} shows that $\beta$ induces the same map $K_*(A)\to K_*(D)$ as the $*$-homomorphism $\gamma:A\to D$ defined by 
$$
\gamma (a):= \Bigg( \sum_{n=1}^\infty u_nv_{n}av_{n}^*u_n^*~,~\sum_{n=1}^\infty u_nv_\infty av_\infty^*u_n^*\Bigg).
$$
On the other hand, the $*$-homomorphism $\delta:A\to D$ defined by 
$$
\delta:a\mapsto (u_0v_\infty a v_\infty^* u_0^*~,~u_0v_\infty a v_\infty^* u_0^*)
$$
induces the zero map on $K$-theory (by Lemma \ref{double}) and has orthogonal image to $\gamma$.  Hence from Lemma \ref{orth} the sum $\epsilon:=\gamma+\delta$ is a well-defined $*$-homomorphism that induces the same map on $K$-theory as $\beta$.  

Compiling our discussion so far, we have
\begin{equation}\label{same homo}
\alpha_*=\beta_*=\gamma_*=\gamma_*+\delta_*=\epsilon_*
\end{equation}
as maps $K_*(A)\to K_*(D)$.  Let $\psi_0,\psi_\infty:A\to D$ be the $*$-homomorphisms defined by
$$
\psi_0:a\mapsto (u_0v_0av_0^*u_0^*,0),\quad \text{ and } \quad \psi_\infty: a\mapsto (u_0v_\infty av_\infty^* u_0^*,0),
$$
and define $\zeta:A\to D$ by
$$
\zeta(a):= \Bigg( \sum_{n=1}^\infty u_nv_{n}av_{n}^*u_n^*~,~\sum_{n=0}^\infty u_nv_\infty av_\infty^*u_n^*\Bigg).
$$
Note that $\zeta$ has orthogonal image to $\psi_0$ and $\psi_\infty$, and that 
$$
\psi_0+\zeta=\alpha\quad  \text{ and } \quad \psi_\infty+\zeta=\epsilon;
$$
hence from Lemma \ref{orth} and line \eqref{same homo},
$$
(\psi_0)_*+\zeta_*=\alpha_*=\epsilon_*=(\psi_\infty)_*+\zeta_*.
$$
Cancelling $\zeta_*$ thus gives that $\psi_0$ and $\psi_\infty$ induce the same maps on $K$-theory.  

Finally, note that if $\iota:A\to D$ is the map of Lemma \ref{double}, then 
$$
\psi_i(a)=u_0\iota(\phi_i(a))u_0^*
$$
for all $a\in A$ and $i\in \{0,\infty\}$.  This implies the desired result as Lemmas \ref{isom con} and \ref{double} show that conjugation of $D$ by $(u_0,u_0)$ and $\iota:A\to D$ both induce isomorphisms on $K$-theory.
\end{proof}

\subsection*{Dynamical part}

We now provide the dynamical input for Proposition \ref{al e s} needed to complete the proof of Proposition \ref{basecase}.  Recall that we want to show that $K_*(A^s(G))$ is zero whenever the open subgroupoid $G$ of $\Gamma\ltimes X$ and number $s\geq 0$ satisfy
\begin{equation}\label{f contain}
G\subseteq \{(gx,g,x)\in \Gamma\ltimes X\mid |g|\leq s\}.
\end{equation}
For the remainder of the section, fix $s$ and $G$ satisfying these hypotheses.

We will start by building a convenient representation of the $C^*$-algebra $A^s(G)$.  For $z\in P_s(\Gamma)$, recall from Definition \ref{t supp} that $\text{supp}(z)$ consists of those $g\in \Gamma$ spanning the minimal simplex containing $z$, and define
$$
P_s(G):=\{(z,x)\in P_s(\Gamma)\times X\mid (gx,g,x)\in G \text{ for all } g\in \text{supp}(z)\}.
$$
Recall from Definition \ref{zf def} that $Z_s$ is our fixed dense subset of $P_s(\Gamma)$. Define $Z_G:=(Z_s\times X)\cap P_s(G)$ and 
$$
H_{G}:=\ell^2(Z_{G},H\otimes \ell^2(\Gamma))=\ell^2(Z_G)\otimes H\otimes \ell^2(\Gamma)
$$
which is a subspace of $H_s=\ell^2(Z_s\times X,H\otimes \ell^2(\Gamma))$ as described in line \eqref{hf2} above.  We have the following lemma; the proof involves essentially the same computations as Lemma \ref{supp alg} above, and is thus omitted.

\begin{lemma}\label{g rep}
The faithful representation of $C^*(G;s)$ on $H_{s}$ restricts to a faithful representation on $H_{G}$. \qed
\end{lemma} 
For the remainder of this section, we will consider $C^*(G;s)$ as faithfully represented on $H_{G}$, and $A^s(G):=C^*_{L,0}(G;s)$ as faithfully represented on $L^2([0,\infty),H_{G})$ in the obvious way.

Now, the assumption in line \eqref{f contain} implies that if $(z,x)\in P_s(G)$ and $\text{supp}(z)=\{g_1,...,g_n\}$, then $\{e,g_1,...,g_n\}$ also spans a simplex $\Delta$ in $P_s(\Gamma)$ such that $\Delta\times \{x\}$ is contained in $P_s(G)$.  Hence the family of functions
$$
F_r:P_{s}(G)\to P_{s}(G),\quad (z,x)\mapsto ((1-r)z+re,x), \quad 0\leq r\leq 1
$$
defines a homotopy between the identity map on $P_s(G)$ and the obvious projection onto the subset $\{(z,x)\in P_s(G)~|~z=e\}$, which is just a copy of the unit space $G^{(0)}$.  

Let $\pi:Z_{G}\to Z_s$ denote the projection onto the first coordinate, and let $Z$ denote the image of $\pi$.  Note that $Z$ is countable, whence as $H$ is infinite dimensional we may find a family  $\{w_{z,0}\}_{z\in Z}$ of isometries on $H$ such that $\sum_{z\in Z} w_{z,0}w_{z,0}^*$ converges strongly to the identity.  For each $z\in Z$, let $w_z:H\otimes \ell^2(\Gamma)\to H\otimes \ell^2(\Gamma)$ be the isometry defined by $w_z:=w_{z,0}\otimes 1_{\ell^2(\Gamma)}$.

Now, for each $r\in \Q\cap [0,1]$ define 
\begin{align*}
w(r):\ell^2(Z_{G})\otimes H\otimes \ell^2(\Gamma) & \to \ell^2(Z_{G})\otimes H\otimes \ell^2(\Gamma) \\  \delta_{z,x}\otimes \eta & \mapsto \delta_{F_r(z,x)}\otimes w_z\eta
\end{align*}
which is a well-defined isometry by definition of $Z_s$ (Definition \ref{zf def} above) and of $Z_G$; note that the different $w(r)$ have mutually orthogonal ranges as $r$ ranges over $\Q\cap [0,1]$.  For each $t\in [0,\infty)$ and $n\in \N\cup\{\infty\}$ (we assume $\N$ includes zero), define an isometry
$$
v_n(t):\ell^2(Z_{G})\otimes H\otimes \ell^2(\Gamma)\to \ell^2(Z_{G})\otimes H\otimes \ell^2(\Gamma)
$$ 
by the following prescription.  First, we define for $m\in \N$
$$
v_n(m)=\left\{\begin{array}{ll} w(0) & m\leq n \\ w(\frac{1}{n}(m-n)) & m\in (n,2n)\cap \N \\ w(1) & m\geq 2n\end{array}\right..
$$
Schematically, we thus have 
$$
v_n(m)=  \underbrace{w(0),\,\cdots ,\,w(0)}_{m\in [0,n]},\, \underbrace{w({\scriptstyle \frac{1}{n}}),\,w({\scriptstyle \frac{2}{n}}),\,\cdots,\,w({\scriptstyle \frac{n-1}{n}})}_{m\in (n,2n)},\,\underbrace{w(1),\,w(1),\,\cdots}_{m\in [2n,\infty)}.
$$
Now we interpolate between these values by defining for $t=m+s$ with $s\in (0,1)$,
$$
v_n(t)=\cos(\frac{\pi}{2}s)v_n(m)+\sin(\frac{\pi}{2}s)v_n(m+1).
$$
It is not difficult to check that the map 
$$
[0,\infty)\to \mathcal{B}(\ell^2(Z_{G})\otimes H\otimes \ell^2(\Gamma)),\quad t\mapsto v_n(t)
$$
is norm continuous for each $n$ (and in fact, the family is equicontinuous as $n$ varies), and that the image consists entirely of isometries.  The following schematic may help to visualize the operators $v_n(t)$.\\
\setlength{\unitlength}{1cm}
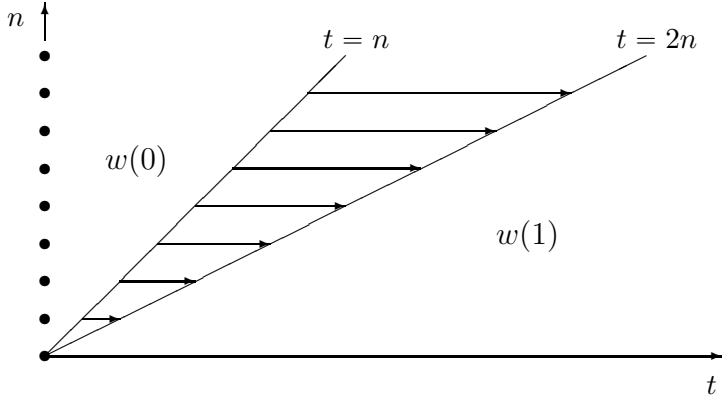
\begin{picture}(10,6)(-0.5,-1)
\put(0,0){\vector(1,0){9}}
\put(8.8,-0.5){\small{$t$}}
\put(0,0){\circle*{0.15}}
\put(0,0.5){\circle*{0.15}}
\put(0,1){\circle*{0.15}}
\put(0,1.5){\circle*{0.15}}
\put(0,2){\circle*{0.15}}
\put(0,2.5){\circle*{0.15}}
\put(0,3){\circle*{0.15}}
\put(0,3.5){\circle*{0.15}}
\put(0,4){\circle*{0.15}}
\put(0,4.2){\vector(0,1){0.5}}
\put(-0.5,4.4){\small{$n$}}
\put(0,0){\line(1,1){4}}
\put(3.7,4.1){\small{$t=n$}}
\put(0,0){\line(2,1){8}}
\put(7.6,4.1){\small{$t=2n$}}
\put(0.8,2.5){$w(0)$}
\put(6,1.5){$w(1)$}
\put(0.5,0.5){\vector(1,0){0.5}}
\put(1,1){\vector(1,0){1}}
\put(1.5,1.5){\vector(1,0){1.5}}
\put(2,2){\vector(1,0){2}}
\put(2.5,2.5){\vector(1,0){2.5}}
\put(3,3){\vector(1,0){3}}
\put(3.5,3.5){\vector(1,0){3.5}}
\end{picture}\\
Here $v_n(t)$ is constantly equal to $w(0)$ in the left triangular region, and constantly equal to $w(1)$ in the right triangular region.  Along each of the horizontal arrows in the intermediate region, $v_n(t)$ interpolates between $w(0)$ and $w(1)$, taking longer and longer to do so as $n$ increases.

We are now ready to construct isometries as demanded by the definition of stable equivalence.  Define an isometry
$$
v_n:L^2([0,\infty), H_G) \to L^2([0,\infty), H_G)
$$
for each $n$ by defining for each $\xi\in L^2([0,\infty), H_G)$ 
$$
(v_n\xi)(t):=v_n(t)(\xi(t)).
$$
On the other hand, choose a unitary isomorphism 
$$
H\cong \bigoplus_{n=0}^\infty H
$$
and use this to define isometries $u_{n,0}:H\to H$ with mutually orthogonal ranges such that the sum $\sum u_{n,0}u_{n,0}^*$ converges strongly to the identity operator.  Define $u_n$ to be the isometry on 
$$
L^2([0,\infty), H_G)=L^2[0,\infty)\otimes \ell^2(Z_G)\otimes H\otimes \ell^2(\Gamma) 
$$ 
induced by tensoring $u_{n,0}$ with the identity on the other factors.  We may think of elements of $M(A^s(G))$ as functions from $[0,\infty)$ to $M(C^*(G;s))$ (subject to various additional conditions, but those are not important here).  Thought of like this, let $\tau$ be the topology of pointwise strict convergence on $M(A^s(G))$, i.e.\ a net $(m_i)$ converges to $m$ if and only if $(m_i(t))$ converges to $m(t)$ strictly for all $t\in [0,\infty)$.  It is then not difficult to see that $(u_n)$ and $\tau$ together define a stability structure, where we take the isometry $v$ needed by the definition to be
$$
v:=\sum_{n=0}^\infty u_{n+1}u_n^*,
$$
noting that the sum $\tau$-converges to an element of $M(A^s(G))$.  

\begin{lemma}\label{con es}
With notation and stability structure as above, the isometries $v_0$ and $v_\infty$ are stably equivalent with respect to the $C^*$-algebra $A^s(G)$.
\end{lemma}

\begin{proof}
Let $a$ be an element of $\C_{L,0}[G;s]$, and let $T=a(t)$ for some fixed $t\in [0,\infty)$.  The matrix entries $(v_n(t)Tv_n(t)^*)_{y,z}(x)$ of $v_n(t)Tv_n(t)^*$ will then be a linear combination of at most four terms of the form
$$
w_z^*T_{F_{r_1}(y),F_{r_2}(z)}(x)w_y^*
$$
where $r_1$ and $r_2$ are in $\Q\cap [0,1]$, and $|r_1-r_2|<1/m$ whenever $t>2m$.  From this description, it is straightforward to check  that the Rips-propagation of $v_n(t)a(t)v_n(t)^*$ is at most the Rips-propagation of $a(t)$ plus $\min\{1,2/(t-1)\}$; and therefore in particular that $v_nav_n^*$ is in $\C_{L,0}[G;s]$.  Condition \eqref{new con} follows from this estimate on Rips propagation and equicontinuity of the sequence of maps $(t\mapsto v_n(t))$.    

To see that $v_{n+1}v_n^*$ is a multiplier of $A^s(G)$, note that the operators $S_t:=v_{n+1}(t)v_n(t)^*$ on $\ell^2(Z_G)\otimes H\otimes \ell^2(\Gamma)$ have matrix entries $(S_t)_{y,z}$ that act as constant functions $X\to \mathcal{B}(H\otimes \ell^2(\Gamma))$; 
that their Rips-propagation tends to zero as $t$ tends to infinity uniformly in $n$; and that they have $\Gamma$-propagation at most $s$ for all $t$.  Condition \eqref{new con 2} follows from this.

Finally, the fact that the operators
$$
\sum_{n=0}^\infty u_nv_nav_n^*u_n^*-\sum_{n=0}^\infty u_nv_\infty av_\infty^*u_n^*,
$$
as well as 
$$
\sum_{n=0}^\infty u_n a (v_\infty v_\infty^* - v_{n+1}v_n^*)u_n^*\quad \text{and}\quad \sum_{n=0}^\infty u_n (v_\infty v_\infty^* - v_{n+1}v_n^*) a u_n^*
$$
are in $A^s(G)$ for all $a\in A^s(G)$ follows from the above discussion and as for any fixed $t$ and all $n>t$, $v_n(t)=v_\infty(t)$.
\end{proof}

Let now
$$
\phi_0,\phi_\infty:K_*(A^{s}(G))\to K_*(A^{s}(G))
$$
be the maps induced on $K$-theory by conjugation by $v_0$ and by $v_\infty$.  Proposition \ref{al e s} and Lemma \ref{con es} imply that these are the same map.   The following two lemmas will now complete the proof of Proposition~\ref{basecase}.

\begin{lemma}\label{phiinfid}
The map 
$$
\phi_\infty:K_*(A^{s}(G))\to K_*(A^{s}(G))
$$
constructed above is the identity map.
\end{lemma}

\begin{proof}
The map $\phi_\infty$ is given by conjugation by the isometry $w(0)$ (constantly in the `localization variable' $t$), which is in the multiplier algebra of $A^s(G)$.  Hence it induces the identity on $K$-theory by Lemma \ref{isom con}.
\end{proof}

\begin{lemma}\label{phi0zero}
The map 
$$
\phi_0:K_*(A^{s}(G))\to K_*(A^{s}(G))
$$
constructed above is the zero map.
\end{lemma}

\begin{proof}
Let $G^{(0)}$ be the unit space of $G$, which is an open subgroupoid of $\Gamma\ltimes X$, and thus $A^s(G^{(0)})$ makes sense.  It is straightforward to check that $\phi_0$ fits into a commutative diagram
$$
\xymatrix{ K_*(A^{s}(G)) \ar[dr] \ar[rr]^{\phi_0}  & & K_*(A^{s}(G)) \\
& K_*(A^{s}(G^{(0)})) \ar[ur] & },
$$
whence it suffices to show that $K_*(A^{s}(G^{(0)}))=0$. 

Say now that $a$ is an element of $A^s(G^{(0)})$ and $t\in [0,\infty)$, $y,z\in P_s(\Gamma)$, and $x\in X$ are such that $a(t)_{y,z}(x)\neq 0$.  Then by the condition that the support of $a(t)$ is contained in $G^{(0)}$ (compare Definitions \ref{t supp} and \ref{smallob} above), we must have that for all $g\in \text{supp}(y)$ and all $h\in \text{supp}(z)$, $(gx,gh^{-1},hx)$ is in $G^{(0)}$.  This forces $gh^{-1}$ to be the identity element of $\Gamma$ and thus $g=h$.  As this happens for all $g\in \text{supp}(y)$ and $h\in \text{supp}(z)$, this forces $\text{supp}(y)$ and $\text{supp}(z)$ to both reduce to a single element of $\Gamma$, and moreover that these elements are necessarily the same.  In particular, $a(t)$ has zero Rips-propagation for all $t$.  

Now, let $u_n:L^2([0,\infty),H_G)\to L^2([0,\infty),H_G)$ be the isometries constructed above.  For each $n\in \N$ and element $a$ of $A^s(G^{(0)})$, define $a^{(n)}$ to be the function 
$$
a^{(n)}(t)=\left\{\begin{array}{ll} a(t-n) & t\geq n \\ 0 & t<n \end{array}\right.
$$
in $A^s(G^{(0)})$.  Define 
$$
\alpha:A^s(G^{(0)})\to  A^s(G^{(0)}),~~~a\mapsto \sum_{n=0}^\infty u_n a^{(n)}u_n^*.
$$
As every element of $A^s(G^{(0)})$ has zero Rips-propagation and satisfies $a(0)=0$, $\alpha$ is a well-defined $*$-homomorphism.  It thus induces a map on $K$-theory 
$$
\alpha_*:K_*(A^s(G^{(0)}))\to K_*(A^s(G^{(0)})).
$$
However, if $\iota:A^s(G^{(0)})\to A^s(G^{(0)})$ is the identity map, then a straightforward homotopy using uniform continuity of each element of $A^s(G^{(0)})$ shows that $\alpha_*+\iota_*=\alpha_*$ and we are done.
\end{proof}

\section{Mayer-Vietoris}\label{ind sec}

In this section, we prove Proposition \ref{indstep}, which we repeat below for the reader's convenience.

\begin{proposition*2}
Let $G$ be an open subgroupoid of $\Gamma\ltimes X$ that is in the class $\mathcal{D}$ of Definition \ref{gpd fdc}, and let $r_0,s_0\geq 0$.  Then there is $s\geq \max\{r_0,s_0\}$ (depending on $r_0$, $s_0$ and $G$) such that the subspace-inclusion map (cf.\ Remark \ref{subspaces})
$$
K_*^{r_0,1/8}(A^{s_0}(G)) \to K_*^{s,1/8}(A^s(G))
$$
is the zero map.
\end{proposition*2}

As in Section \ref{base sec}, we first build an abstract $K$-theoretic machine, and then use the dynamical assumptions to produce ingredients for that machine.

\subsection*{$K$-theoretic part}

We start with a technical lemma about when elements of controlled $K$-groups are zero: compare \cite[Section 1.6]{Oyono-Oyono:2011fk}.  The proof is in large part the same as that of \cite[Proposition 1.31]{Oyono-Oyono:2011fk}, but as our set up and precise statement are a little different, and to keep things self-contained, we give a complete proof here.

Before stating the lemma, we recall some notation from Section \ref{conk sec}, and introduce some more.  Let $A$ be a non-unital $C^*$-algebra and $S\subseteq A$ a self-adjoint subspace.  Let $\widetilde{A}$ be the unitization of $A$, and let $\widetilde{S}$ be the subspace of $\widetilde{A}$ spanned by $S$ and the unit.  Then 
$$
\kappa:P_n^{1/8}(\widetilde{S})\to P_n(\widetilde{A}),\quad p\mapsto \chi_{(1/2,\infty]}(p)
$$
is the map from quasi-projections in $M_n(\widetilde{S})$ to projections in $M_n(\widetilde{A})$ of Definition \ref{qp}.  Similarly, 
$$
\kappa:U_n^{1/8}(\widetilde{S})\to U_n(\widetilde{A}), \quad u\mapsto u(u^*u)^{-1/2}
$$
is the map from quasi-unitaries in $M_n(\widetilde{S})$ to unitaries in $M_n(\widetilde{A})$ of Definition \ref{qu}.  
For each $m\in \N$, define 
$$
\widetilde{S}^m:=\text{span}\{a_1\cdots a_m\in A\mid a_i\in \widetilde{S} \text{ for all } i\in \{1,...,m\} \}.
$$
We will also need some notation for standard matrices.  Given $m\in \N$, we will write `$1_m$' for the $m\times m$ identity matrix and `$0_m$' for the $m\times m$ zero matrix.  We will adopt the shorthand `$\text{diag}(\cdots)$' for a diagonal matrix with given entries: for example
$$
\text{diag}(a,b,0)=\begin{pmatrix} a & 0 & 0 \\ 0 & b & 0 \\ 0 & 0 & 0\end{pmatrix}.
$$

\begin{lemma}\label{con zero}
For any $\epsilon>0$ there are constants $L=L(\epsilon)\geq 0$ and $M=M(\epsilon)\in \N$ with the following properties.   Let $A$ be a non-unital $C^*$-algebra and $S\subseteq A$ a self-adjoint subspace.  
\begin{enumerate}[(i)]
\item Say $l\in \N$ and $(p,n)\in P^{1/8}_l(\widetilde{S})\times \N$ is such that the class $[p,n]$ is zero in $K_0^{1/8}(S)$.  Then there exist $k_1,k_2\in \N$ and a homotopy $h:[0,1]\to P_{l+k_1+k_2}(\widetilde{A})$ with
$$
h(0)=\text{diag}(0_{k_1}, \kappa(p), 1_{k_2}) \quad \text{and}\quad h(1)=\text{diag}(0_{k_1},0_{l-n}, 1_{n+k_2}),
$$ 
such that there is an $L$-Lipschitz map $h_\epsilon:[0,1]\to M_{l+k_1+k_2}(\widetilde{S}^M)$ with 
$$\sup_{t\in [0,1]}\|h_\epsilon(t)-h(t)\|<\epsilon.$$
\item Say $l\in \N$, and $u\in U^{1/8}_l(\widetilde{S})$ is such that the class $[u]$ is zero in $K_1^{1/8}(S)$.  Then there is $k\in \N$ and a homotopy $h:[0,1]\to U_{l+k}(\widetilde{A})$ with 
$$
h(0)=\text{diag}(\kappa(u), 1_k )  \quad \text{and}\quad h(1)=(1_l,1_k )
$$ 
such that there is an $L$-Lipschitz map $h_\epsilon:[0,1]\to M_{l+k}(\widetilde{S}^M)$ such that $$\sup_{t\in [0,1]}\|h_\epsilon(t)-h(t)\|<\epsilon.$$
\end{enumerate}
\end{lemma}

To summarize the idea, if an element of $K_*^{1/8}(S)$ is zero, then the fact that its image under the comparison map $c:K_*^{1/8}(S)\to K_*(A)$ is zero can be witnessed by a homotopy that is well-controlled, both with respect to how fast it goes, and with respect to the subspace of $M_\infty(\widetilde{A})$ it passes through.

\begin{proof}
We look first at the case of $K_0$.  Let $(p,n)\in P^{1/8}_l(\widetilde{S})\times \N$ satisfy the hypotheses of the lemma.  Unwrapping the definitions, this is equivalent to saying that there exist $j_1,j_2,j_3\in \N$ and an element $\{p_t\}_{t\in [0,1]}$ of $P_{l+j_1+n+j_2+j_3}^{1/8}(C([0,1],\widetilde{S}))$ such that 
$$
p_0=\text{diag}(p,0_{j_1},0_n,1_{j_2},0_{j_3}) \quad \text{and} \quad p_1=\text{diag}(0_l,0_{j_1},1_n,1_{j_2},0_{j_3}).
$$
As $[0,1]$ is compact, there are $0=t_0<\cdots <t_N=1$ such that
\begin{equation}\label{eps}
\|p_{t_i}-p_{t_{i-1}}\|<1/12 \quad \text{for all $i\in \{1,...,N\}$}.
\end{equation}
Set $m=j_1+j_2+j_3+n+l$.  We will first define a Lipschitz homotopy between 
$$
\text{diag} (p_0, 1_{mN},0_{mN})\quad  \text{and} \quad \text{diag}( 0_{mN},p_1,1_{mN})
$$
by concatenating the steps below.

\begin{enumerate}[(i)]
\item Perform a rotation homotopy between 
$$
\text{diag}(p_0,1_{mN},0_{mN}) \text{ and } \text{diag}(p_0,\underbrace{1_m,0_m,...,1_m,0_m}_{N\text{ copies of }(1_m,0_m)});
$$
\item Let 
$$
r(t)=\begin{pmatrix} \cos(\pi t/2) & -\sin(\pi t/2) \\ \sin(\pi t/2) & \cos(\pi t/2) \end{pmatrix}\in M_{2m}(\C),
$$
where each entry represents the corresponding scalar times $1_m$.  For $i\in \{1,...,N\}$, in the $i^{th}$ `block' $\text{diag}(1_m,0_m)$ appearing in the above apply the homotopy
$$
t\mapsto \begin{pmatrix} 1_m-p_{t_i} & 0 \\ 0 & 0 \end{pmatrix}+r(t)\begin{pmatrix} p_{t_i} & 0 \\ 0 & 0 \end{pmatrix}r(t)^*
$$
between $\text{diag}(1_m,0_m)$ and $\text{diag}(1-p_{t_i},p_{t_i})$ to get a homotopy between 
\begin{align*}
\text{diag}&(p_0,1_m,0_m,...,1_m,0_m) \\& \text{and}\quad  \text{diag}(p_0,1_m-p_{t_1},p_{t_1},1_m-p_{t_2},p_{t_2},...,1-p_{t_N},p_{t_N}).
\end{align*}
\item For each $i\in \{1,...,N\}$, use a straight line homotopy between $1_m-p_{t_i}$ and $1_m-p_{t_{i-1}}$ in each appropriate entry to build a homotopy between
\begin{align*}
\text{diag}&(p_0,1_m-p_{t_1},p_{t_1},1_m-p_{t_2},p_{t_2},...,1-p_{t_N},p_{t_N}) \\ & \text{and}\quad \text{diag}(p_0,1_m-p_{t_0},p_{t_1},1_m-p_{t_1},p_{t_2},...,1-p_{t_{N-1}},p_{t_N}).
\end{align*}
\item Using a similar homotopy to step (ii), and that $p_0=p_{t_0}$, build a homotopy between
\begin{align*}
\text{diag} (p_0,1_m-p_{t_0},p_{t_1},1_m-p_{t_1},p_{t_2}&,...,1-p_{t_{N-1}},p_{t_N}) \\ 
& \text{and}\quad  \text{diag}(0_m,1_m,....,0_m,1_m,p_{t_N}).
\end{align*}
\item Finally, recall that $p_{t_N}=p_1$ and use another rotation homotopy between
$$
\text{diag}(0_m,1_m,....,0_m,1_m,p_1) \quad \text{and} \quad \text{diag}(0_{mN},p_1,1_{mN})
$$
to complete the proof of the claim.
\end{enumerate}

Write $\{q_t\}_{t\in [0,1]}$ for the homotopy arrived at by concatenating the steps above; it is straightforward to check that this homotopy is Lipschitz for some universal Lipschitz constant.  Note also that all of the matrices from steps (i)-(v) above have entries from $\widetilde{S}$, so this homotopy has image in $M_{m(2N+1)}(\widetilde{S})$.

We claim that $\|q_t^2-q_t\|<5/24$ for all $t$.  Indeed, for all $t$ associated to steps (i), (ii), (iv), and (v) we clearly have that $\|q_t^2-q_t\|<1/8$, so the only thing to be checked is the straight line homotopy in step (iii).  For this, using the bound in line \eqref{eps} it suffices to show that if $p,q$ are two quasi-projections with $\|p-q\|<1/12$, then for any $t\in [0,1]$, 
$$
\|((1-t)p+tq)^2-((1-t)p+tq)\|<5/24.
$$
Computing,
\begin{align*}
(&(1-t)p+tq)^2-((1-t)p+tq) \\ & =(1-t)(p^2-p)-(1-t)tp(p-q)-(1-t)tq(q-p)+t(q^2-q).
\end{align*}
As $\|p\|$ and $\|q\|$ are both bounded above by $2$, this gives
\begin{align*}
\|(&(1-t)p+tq)^2-((1-t)p+tq)\| \\ & < (1-t)(1/8)+t(1-t)(1/6)+t(1-t)(1/6)+t(1/8) \\ &\leq 1/8+2/24=5/24
\end{align*}
as claimed.  

Hence the spectrum of every $q_t$ is bounded away from $1/2$, and thus defining $h(t):=\kappa(q_t)$ makes sense.  Fix a sequence of real-valued polynomials $(f_i)$ converging uniformly to $\chi_{(1/2,\infty]}$ on the spectrum of every $q_t$.  As $f_i$ is a polynomial and $t\mapsto q_t$ is Lipschitz and bounded, $t\mapsto f_i(q_t)$ is Lipschitz for each $i$, with some Lipschitz constant depending only on the fixed choice of $f_i$, and on the Lipschitz constant of $t\mapsto q_t$.   It moreover takes image in $M_{m(2N+1)}(\widetilde{S}^{M_i})$, where $M_i$ is the degree of $f_i$.  
It follows from all of this that we may take $h_\epsilon(t):=f_i(q_t)$ for some suitably large $i$, and this will have all the right properties.\\

We now turn to the case of $K_1$.  Let $[u]$ satisfy the hypotheses, so there exist $j\in \N$ and a homotopy $\{u_t\}_{t\in [0,1]}$ in $U_{l+j}^{1/8}(C([0,1],\widetilde{S}))$ such that 
$$
u_0=\text{diag}(u,1_j) \quad \text{and} \quad u_1=\text{diag}(1_l,1_j).
$$
Set $m=l+j$.  Let $0=t_0<\cdots <t_N=1$ be such that 
\begin{equation}\label{eps2}
\|u_{t_i}-u_{t_{i-1}}\|<1/32 \quad \text{for all $i\in \{1,...,N\}$}.
\end{equation}
We will define a Lipschitz homotopy between
$$
\text{diag}(u_0,1_{2mN}) \quad \text{and} \quad \text{diag}(u_1,1_{2mN}).
$$
by concatenating the homotopies below.
\begin{enumerate}[(i)]
\item Connect 
\begin{align*}
\text{diag}(u_0,1_{2mN})=\text{diag}(u_0&,\underbrace{1_m,...,1_m}_N,1_{mN}) \\ & \text{ and } \quad \text{diag}( u_0,u_{t_1}^*u_{t_1},....,u_{t_N}^*u_{t_N},1_{mN})
\end{align*}
by the straight line homotopy between the $i^\text{th}$ copy of $1_m$ and $u_{t_i}^*u_{t_i}$.
\item Use a rotation homotopy between 
$$
\text{diag}( 1_m,u_{t_1}^*,....,u_{t_N}^*,1_{mN}) \quad \text{and} \quad \text{diag}(u_{t_1}^*,....,u_{t_N}^*,1_m,1_{mN}),
$$
to produce a homotopy between 
\begin{align*}
\text{diag}&( u_0,u_{t_1}^*u_{t_1},....,u_{t_N}^*u_{t_N},1_{mN}) \\ &=\text{diag}( 1_m,u_{t_1}^*,....,u_{t_N}^*,1_{mN})\text{diag}( u_0,u_{t_1},....,u_{t_N},1_{mN})
\end{align*}
and 
\begin{align*}
\text{diag}&( u_{t_1}^*u_0,u_{t_2}^*u_{t_1},....,u_{t_{N}}^*u_{t_{N-1}},u_{t_N}1_{mN}) \\ &=\text{diag}(u_{t_1}^*,....,u_{t_N}^*,1_m,1_{mN})\text{diag}( u_0,u_{t_1},....,u_{t_N},1_{mN}).
\end{align*}
\item Use another straight line homotopy between each $u_{t_i}^*u_{t_{i-1}}$ and $1_m$ to build a homotopy between
\begin{align*}
\text{diag}( u_{t_1}^*u_0,u_{t_2}^*u_{t_1} ,....,&u_{t_{N}}^*u_{t_{N-1}},u_{t_N},1_{mN}) \\& \text{and} \quad \text{diag}(\underbrace{1_m,....,1_m}_N,u_{t_N},1_{mN})
\end{align*}
\item Finally, one more rotation homotopy connects 
$$
\text{diag}(\underbrace{1_m,....,1_m}_N,u_{t_N},1_{mN}) \quad \text{and} \quad \text{diag}(u_{1},1_{2mN}),
$$
where we used that $u_{t_N}=u_1$.
\end{enumerate}

Write $\{v_t\}_{t\in [0,1]}$ for the homotopy resulting from concatenating the above homotopies; it is straightforward to check that this homotopy is Lipschitz for some universal Lipschitz constant, and that each $v_t$ is a matrix in $M_{m(2N+1)}(\widetilde{S}^2)$.  

We claim that for each $t$,
\begin{equation}\label{v est}
\|v_t^*v_t-1\|<7/8 \quad \text{and} \quad \|v_tv_t^*-1\|<7/8.
\end{equation}
For $t$ associated to step (iv), this is immediate.  For $t$ associated to step (ii), this follows from the fact that if $v$ and $w$ are quasi-unitaries, then their product satisfies 
\begin{align*}
\|(vw)(vw)^*-1\| & \leq \|v(ww^*-1)v^*\|+\|vv^*-1\| \\ & <(1+1/8)(1/8)+(1/8) \\ & <7/8
\end{align*}
and similarly $\|(vw)^*(vw)-1\|<7/8$.  For $t$ associated to steps (i) and (iii), we first claim that if $\delta,\epsilon\in (0,1)$ and $\|u-v\|<\epsilon$, and $\|uu^*-1\|<\delta$ and $\|u^*u-1\|<\delta$, then 
\begin{equation}\label{4de}
\|vv^*-1\|<\delta+4\epsilon\quad \text{and}\quad \|v^*v-1\|<\delta+4\epsilon.
\end{equation}
Indeed, using that $\|u\|<\sqrt{1+\delta}$, 
\begin{align*}
\|vv^*-1\| & < \|vv^*-uu^*\|+\delta\leq \|v\|\|v^*-u^*\|+\|v-u\|\|u^*\|+\delta \\ &<(\sqrt{1+\delta}+\epsilon)\epsilon+(\sqrt{1+\delta})\epsilon+\delta<\delta+4\epsilon,
\end{align*}
and the other estimate is similar.  Now, looking at step (i), we have that all elements appearing in the homotopy are within $1/8$ of $\text{diag}(u_0,1_{2mN})$, so applying the estimate in line \eqref{4de} with $\delta=\epsilon=1/8$ establishes the estimate in line \eqref{v est}.  On the other hand, looking at step (iii), we first note that for any $i\in \{1,...,N\}$,
\begin{align*}
\|u_{t_i}^*u_{t_{i-1}}-1\|&\leq \|u_{t_i}^*\|\|u_{t_i}-u_{t_{i-1}}\|+\|u_{t_i}^*u_{t_i}-1\| \\ & \leq \sqrt{1+1/8}(1/32)+1/8 < 3/16.
\end{align*}
Hence every element in the homotopy in step (iii) is within $3/16$ of 
$$
\text{diag}(1_{mN},u_{t_N},1_{mN}).
$$
Applying the estimate in line \eqref{4de} with $\delta=1/8$ and $\epsilon=3/16$ then again gives the estimate in line \eqref{v est}, and we are done with the claim.

It follows in particular that $(v_t^*v_t)^{-1/2}$ makes sense for all $t$, and thus we may define $h(t):=\kappa(v_t)$.  Moreover, there is a sequence $(f_i)$ of real-valued polynomials that converges uniformly to the function $t\mapsto t^{-1/2}$ on the spectrum of each $v_t^*v_t$.  Analogously to the case of $K_0$, we may now take $h_\epsilon(t):=v_tf_i(v_t^*v_t)$ for some suitably large $i$ (depending on $\epsilon$); this has all the right properties.
\end{proof}

Our main $K$-theoretic goal is a sort of controlled Mayer-Vietoris sequence.  Let us first recall the relevant classical Mayer-Vietoris sequence in operator $K$-theory: see for example \cite[Section 3]{Higson:1993th} or \cite[Proposition 2.7.15]{Willett:2010ay} for more details.

\begin{proposition}\label{class mv}
Let $A$ be a $C^*$-algebra, and let $I$ and $J$ be ideals in $A$ such that $A=I+J$.  Then there is a functorial six-term exact sequence
$$
\xymatrix{ K_0(I\cap J)\ar[r] & K_0(I)\oplus K_0(J) \ar[r] & K_0(A) \ar[d]^-{\partial} \\ K_1(A) \ar[u]^-{\partial} & K_1(I)\oplus K_1(J) \ar[l] & K_1(I\cap J) \ar[l]}.
$$ 
The maps $$K_i(I\cap J)\to K_i(I)\oplus K_i(J)$$ are of the form $x\mapsto (x,-x)$ \textup{(}where we abuse notation by writing $x$ both for an element of $K_i(I\cap J)$, and its image in $K_i(I)$ and $K_i(J)$\textup{)}, and the maps $$K_i(I)\oplus K_i(J)\to K_i(A)$$ are of the form $(x,y)\mapsto x+y$ \textup{(}with a similar abuse of notation\textup{)}. \qed
\end{proposition}

The maps above labeled `$\partial$' can also be described explicitly (they are connected to the index and exponential maps of the usual six-term exact sequence), but we will not need this.

We will need some notation and an appropriate excisiveness condition for our controlled Mayer-Vietoris sequence.

\begin{definition}\label{stab subsp}
Let $\mathcal{K}=\mathcal{K}(\ell^2(\N))$, and $A$ be a $C^*$-algebra.  Let $A\otimes \mathcal{K}$ denote the spatial tensor product of $A$ and $\mathcal{K}$; using the canonical orthonormal basis on $\ell^2(\N)$, we think of elements of $A\otimes \mathcal{K}$ as $\N$-by-$\N$ matrices with entries from $A$.  For a subspace $S$ of $A$, let $S\otimes \mathcal{K}$ denote the subspace of $A\otimes \mathcal{K}$ consisting of matrices with all entries in $S$.

In particular, if $A$ is filtered as in Definition \ref{filt}, then we may define $(A\otimes \mathcal{K})_r:=A_r\otimes \mathcal{K}$.  It is straightforward to check that this definition induces a filtration on $A\otimes \mathcal{K}$.
\end{definition}

\begin{definition}\label{ind filt}
Let $A$ be a filtered $C^*$-algebra, and $I$ be a $C^*$-ideal in $A$ equipped with its own filtration.  We say that $I$ is a \emph{filtered ideal}\footnote{This is a more general notion than the filtrations on ideals used by Oyono-Oyono and the third author in \cite[Subsection 3.1]{Oyono-Oyono:2011fk}.}of $A$ if for any $r\geq 0$, $I_r\subseteq A_r$, and if for any $r,s\geq 0$, $A_s\cdot I_r\cup I_r\cdot A_s\subseteq I_{s+r}$.  
\end{definition}

\begin{remark}
For our applications, the special case where $I=A$ as a $C^*$-algebra, but where $I$ and $A$ do not have the same filtration, turns out to be particularly interesting.
\end{remark}

\begin{definition}\label{uni ex}
Let $(I^\omega,J^\omega;A^\omega)_{\omega\in \Omega}$ be an indexed set of triples, where each $A^\omega$ is a filtered $C^*$-algebra, and each $I^\omega$ and $J^\omega$ is a filtered ideal in $A^\omega$.  Give each stabilization $A^\omega\otimes\mathcal{K}$, $I^\omega\otimes \mathcal{K}$ and $J^\omega\otimes \mathcal{K}$ the filtration from Definition \ref{stab subsp} (note that $I^\omega\otimes \mathcal{K}$ and $J^\omega\otimes \mathcal{K}$ are also filtered ideals in $A^\omega\otimes \mathcal{K}$ with these definitions). 

The collection $(I^\omega,J^\omega;A^\omega)_{\omega\in \Omega}$ of pairs of ideals and $C^*$-algebras containing them is \emph{uniformly excisive} if for any $r_0,m_0\geq 0$ and $\epsilon>0$, there are $r\geq r_0$, $m\geq 0$, and $\delta>0$ such that:
\begin{enumerate}[(i)]
\item \label{uni ex i} for any $\omega\in \Omega$ and any $a\in (A^\omega\otimes \mathcal{K})_{r_0}$ of norm at most $m_0$, there exist elements $b\in (I^\omega\otimes \mathcal{K})_r$ and $c\in (J^\omega\otimes \mathcal{K})_r$ of norm at most $m$ such that $\|a-(b+c)\|<\epsilon$;
\item \label{uni ex ii} for any $\omega\in \Omega$ and any $a\in I^\omega\otimes \mathcal{K}\cap J^\omega\otimes \mathcal{K}$ such that 
$$
d(a,(I^{\omega}\otimes \mathcal{K})_{r_0})<\delta \quad \text{and} \quad d(a,(J^\omega\otimes\mathcal{K})_{r_0})<\delta
$$ 
there exists $b\in I^\omega_r\otimes \mathcal{K}\cap J^\omega_r\otimes \mathcal{K}$ such that $\|a-b\|<\epsilon$.
\end{enumerate}
Note that condition \eqref{uni ex ii} implies that for any $\omega$, the family of subspaces $(I^\omega_r\otimes \mathcal{K}\cap J^\omega_r\otimes \mathcal{K})_{r\geq 0}$ defines a filtration of $I^\omega\otimes \mathcal{K}\cap J^\omega\otimes \mathcal{K}$; we equip each $I^\omega \otimes \mathcal{K} \cap J^\omega\otimes \mathcal{K}$ with this filtration.  
\end{definition}

Note that condition \eqref{uni ex i} above is a controlled analogue of the condition `$A=I+J$' from Proposition \ref{class mv}, while condition \eqref{uni ex ii} is a controlled analogue of the fact that if $a\in A$ is close to both $I$ and $J$, then there is an element of $I\cap J$ that is close to $a$ (this can be shown using approximate units).  

We are now ready for our controlled Mayer-Vietoris theorem.  See \cite{Oyono-Oyono:2016qd} for related `controlled Mayer-Vietoris sequences', approached in a somewhat different way.

\begin{proposition}\label{con mv}
Let $(I^\omega,J^\omega;A^\omega)_{\omega\in \Omega}$ be a uniformly excisive collection, where the algebras and ideals are all non-unital.  Then for any $r_0\geq 0$ there are $r_1,r_2\geq r_0$ with the following property.  For each $\omega$ and each $x\in K_*^{r_0,1/8}(A^\omega)$ there is an element
$$
\partial_c(x) \in K_*^{r_1,1/8}(I^\omega\cap J^\omega)
$$
such that if $\partial_c(x)=0$ then there exist 
$$
y\in K_*^{r_2,1/8}(I^\omega) \quad \text{and} \quad z\in K_*^{r_2,1/8}(J^\omega)
$$
such that 
$$
x=y+z \quad \text{in}\quad K_*^{r_2,1/8}(A^\omega)
$$ 
\textup{(}where as usual we abuse notation by omitting explicit notation for subspace-inclusion maps\textup{)}.

Moreover, the `boundary map' $\partial_c$ has the following naturality property.  Let $(K^\theta,L^\theta;B^\theta)_{\theta\in \Theta}$ be another uniformly excisive collection, where the algebras and ideals are non-unital.  Assume moreover that there is a map $\pi:\Theta\to \Omega$ and for each $\theta\in \Theta$ an inclusion $A^{\pi(\theta)}\subseteq B^\theta$ such that for each $r\geq 0$ we have that $A^{\pi(\theta)}_r\subseteq B^\theta_r$, $I^{\pi(\theta)}_r\subseteq  K^\theta_r$, $J^{\pi(\theta)}_r\subseteq L^\theta_r$.  Let $r_0$ be given, and let $r_1$ be as in the statement above for both uniformly excisive families\footnote{We may assume the same $r_1$ works for both families at once by combining them into a single family and applying the first part.}.  Then the diagram
$$
\xymatrix{ K_*^{r_0,1/8}(A^{\pi(\theta)})\ar[r]^-{\partial_c} \ar[d] & K_*^{r_1,1/8}(I^{\pi(\theta)}\cap J^{\pi(\theta)})  \ar[d] \\ 
K_*^{r_0,1/8}(B^{\theta})\ar[r]^-{\partial_c} & K_*^{r_1,1/8}(K^{\theta}\cap J^{\theta})~,} 
$$
where the vertical maps are subspace inclusions, commutes.
\end{proposition}

The subscript in `$\partial_c$' stands for `controlled': $\partial_c$ is a controlled analogue of the usual Mayer-Vietoris boundary map in $K$-theory.  It will be crucial for our applications that the numbers $r_0,r_1,r_2$ appearing in the above are all independent of the index $\omega$.

\begin{proof}
Let $\Lambda$ be a set equipped with a map $\pi:\Lambda\to \Omega$.  Let $\mathcal{K}$ denote the compact operators on $\ell^2(\N)$ and let $\prod_{\lambda\in \Lambda}A^{\pi(\lambda)}\otimes \mathcal{K}$ denote the $C^*$-algebra of bounded, $\Lambda$-indexed sequences where the $\lambda^\text{th}$ element comes from  $A^{\pi(\lambda)}\otimes \mathcal{K}$.  With notation as in Definition \ref{stab subsp}, define
$$
\mathcal{A}_\Lambda:=\Bigg\{(a_\lambda)\in \prod_{\lambda\in \Lambda} A^{\pi(\lambda)}\otimes \mathcal{K}~\Big|~\text{ there is } r\geq 0 \text{ with } a_\lambda\in A^{\pi(\lambda)}_r\otimes \mathcal{K} \text{ for all }\lambda\Bigg\},
$$
which is a $*$-subalgebra of $\prod_{\lambda\in \Lambda}A^{\pi(\lambda)}\otimes \mathcal{K}$, and let $A_\Lambda$ be its $C^*$-algebraic closure.  Define also $\mathcal{I}_\Lambda$ to be
$$
\Big\{(a_\lambda)\in \mathcal{A}_\Lambda~\Big|~\text{ there is } r\geq 0 \text{ with } a_\lambda\in I^{\pi(\lambda)}_r\otimes \mathcal{K} \text{ for all }\lambda\Big\}
$$
and similarly for $\mathcal{J}_\Lambda$.  The definition of a filtered ideal (see Definition \ref{ind filt} above) implies that $\mathcal{I}_\Lambda$ and $\mathcal{J}_\Lambda$ are $*$-ideals in $\mathcal{A}_\Lambda$, whence their closures $I_\Lambda$ and $J_\Lambda$ are $C^*$-ideals in $A_\Lambda$.  Moreover, the uniform excisiveness assumption implies that $A_\Lambda=I_\Lambda+J_\Lambda$.   Hence (see for example \cite[Section 3]{Higson:1993th}) there is a six-term exact Mayer-Vietoris sequence
\begin{equation}\label{big ol mv}
\xymatrix{ K_0(I_\Lambda \cap J_\Lambda) \ar[r] & K_0(I_\Lambda)\oplus K_0(J_{\Lambda})  \ar[r] & K_0(A_\Lambda)  \ar[d]^{\partial} \\ K_1(A_\Lambda) \ar[u]^{\partial} & K_1(I_\Lambda)\oplus K_1(J_{\Lambda}) \ar[l] & K_1(I_\Lambda \cap J_\Lambda)  .\ar[l]}
\end{equation}

From now on, we will focus on the case of $K_0$; the case of $K_1$ is essentially the same.  Define 
$$
\Lambda:=\{(\omega,x)~|~\omega\in\Omega,~x\in K_0^{r_0,1/8}(A^\omega)\}
$$
equipped with the map $\pi:\Lambda\to \Omega$ that sends an element to its first coordinate.  For each $\lambda=(\omega,x)\in \Lambda$ choose a pair 
$$
(p_\lambda,n_\lambda)\in P_{m_\lambda}^{1/8}(A^{\omega}_{r_0})\times \N
$$ 
for some $m_\lambda\in \N$ (see Definition \ref{qp} for notation) such that $x=[p_\lambda,n_\lambda]$.  Identifying $M_{m_\lambda}(A^\omega)$ with the $C^*$-subalgebra of $A^\omega\otimes \mathcal{K}$ consisting of $\N\times \N$ matrices that are only non-zero in the first $m_\lambda\times m_\lambda$ square in the top left corner, we get a well-defined element $\vec{p}:=(p_\lambda)_{\lambda\in \Lambda}$ of $A_\Lambda$.   With $\kappa$ as in Definition \ref{qp}, the formal difference
$$
\vec{x}:=[(\kappa(p_\lambda))_{\lambda\in \Lambda}]-[(1_{n_\lambda})_{\lambda\in \Lambda}]
$$
defines an element of $K_0(A_\Lambda)$, and so the Mayer-Vietoris sequence of line \eqref{big ol mv} gives an element $\partial(\vec{x})$ in $K_1(I_\Lambda\cap J_\Lambda)$.  We may represent $\partial(\vec{x})$ as a $\Lambda$-indexed collection $(u_\lambda)_{\lambda\in \Lambda}$, where each $u_\lambda$ is a unitary in a matrix algebra over the unitization of $I^{\pi(\lambda)}\cap J^{\pi(\lambda)}$.  On the other hand, by definition of $I_\Lambda$ and $J_\Lambda$ and by the uniform excisiveness condition there is $r_1\geq 0$ (which we may assume is at least $r_0$) and a $\Lambda$-tuple $(v_\lambda)_{\lambda\in \Lambda}$ with each $v_\lambda$ in some matrix algebra over the space spanned by $I^{\pi(\lambda)}_{r_1}\cap J^{\pi(\lambda)}_{r_1}$ and the unit, and so that $\|u_\lambda-v_\lambda\|<1/20$.  From this estimate, one checks that each $v_\lambda$ is a quasi-unitary as in Definition \ref{qu} and thus defines a class $[v_\lambda]\in K_1^{r_1,1/8}(I^{\pi(\lambda)}\cap J^{\pi(\lambda)})$.  For each $\lambda=(\omega,x)\in J$, define 
$$
\partial_c(x):=[v_\lambda]\in K_1^{r_1,1/8}(I^\omega\cap J^\omega).
$$

We now look at what happens when $\partial_c(x)=0$.  Let $\Lambda'\subseteq \Lambda$ be the subset of all $(\omega,x)\in \Lambda$ such that $\partial_c(x)=0$ in $K_1^{r_1,1/8}(I^{\pi(\lambda)}\cap J^{\pi(\lambda)})$.  Define a new element $\vec{x}\,'\in K_0(A_\Lambda)$ by setting the $\lambda^{\text{th}}$ component equal to $[\kappa(p_\lambda)]-[1_{n_\lambda}]$ if $\lambda\in \Lambda'$, and equal to $0$ otherwise.

For each $\lambda=(\omega,x)\in \Lambda'$ let $v_\lambda$ be a quasi-unitary such that $\partial_c(x)=[v_\lambda]$.  The fact that $[v_\lambda]$ is zero in $K_1^{r_1,1/8}(I^{\pi(\lambda)}\cap J^{\pi(\lambda)})$ for each $\lambda\in \Lambda'$ and Lemma \ref{con zero} together give a homotopy between a stabilized version of the sequence $(\kappa(v_\lambda))_{\lambda\in \Lambda'}$ and zero in $K_1(I_\Lambda\cap J_\Lambda)$.  With $\partial$ the standard boundary map as in diagram \eqref{big ol mv} we thus have that $\partial(\vec{x}\,')=0$ in $K_1(I_\Lambda\cap J_\Lambda)$.    Hence by exactness of the Mayer-Vietoris sequence there are elements $\vec{y}\in K_0(I_\Lambda)$ and $\vec{z}\in K_0(J_\Lambda)$ such that $\vec{x}\,'=\vec{y}+\vec{z}$ in $K_0(A_\Lambda)$ (as usual, we have suppressed notation for subspace-inclusion maps).  Suitably approximating representatives of $\vec{y}$ and $\vec{z}$ in each component and applying the injectivity part of Proposition \ref{comparelem} gives the desired conclusion.

The naturality property in the second paragraph of the statement follows directly from the corresponding naturality property for the classical Mayer-Vietoris sequence: we leave the details to the reader.
\end{proof}

\subsection*{Dynamical part}

We now use the dynamical assumptions to produce ingredients for the $K$-theoretic machine just built, and thus complete the proof of Proposition \ref{indstep}. 

First, we define the algebras we will be using.  These are subalgebras of our usual obstruction $C^*$-algebras $A^s$ from Definition \ref{main ob}, but we will also need to allow ourselves to change the filtrations involved in order to `relax control' in some sense.  This will give us two different uniformly excisive families in the sense of Definition \ref{uni ex}: our first task in this section will be to define these families and establish that they are indeed uniformly excisive.

\begin{definition}\label{uni algs}
Fix an open subgroupoid $G$ of $\Gamma\ltimes X$ and a constant $s_0\geq 1$.  Let $\Omega$ be the set of all pairs $\omega=(G_0,G_1)$ where $G_0$ and $G_1$ are open subgroupoids of $G$ such that $G^{(0)}=G_0^{(0)}\cup G_1^{(0)}$.   Throughout this definition, we work relative to the subgroupoid $G$ when defining expansions as in Definition \ref{extgpd}.

For our first family, let $\omega=(G_0,G_1)\in \Omega$ and $r\geq 0$, and define  
$$
B^\omega_r:=A^{s_0}(G_0^{+r})_r+A^{s_0}(G_1^{+r})_r+A^{s_0}(G_0^{+r}\cap G_1^{+r})_r
$$
and define also 
$$
B^\omega:=\overline{\bigcup_{r\geq 0}B^\omega_r}
$$
to be the closure of the union of the family $\{B^\omega_r\}_{r\geq 0}$ in the norm topology of $A^{s_0}$.  

\begin{remark}
We note that the filtration $(B_r^\omega)_{r\geq 0}$ depends on the pair $\omega=(G_0,G_1)$, and not only on the ambient groupoid $G$.  In the main proof, we will choose an appropriate pair $(G_0,G_1)$ based on a given scale $r$ for the original filtration on $A^s$ that was introduced in Definition \ref{main ob}.  Thus the filtration on $A^{s}$ will be adapted depending on the scale at which we are working, and the decomposition of $G$ that is appropriate to that scale.
\end{remark}

Define subspaces of $B^\omega_r$ by
$$
I^\omega_r:=A^{s_0}(G_0^{+r})_r+A^{s_0}(G_0^{+r}\cap G_1^{+r})_r 
$$
and 
$$
J^\omega_r:=A^{s_0}(G_1^{+r})_r+A^{s_0}(G_0^{+r}\cap G_1^{+r})_r
$$
and define 
$$
I^\omega:=\overline{\bigcup_{r\geq 0}I^\omega_r}, \quad J^\omega:=\overline{\bigcup_{r\geq 0}J^\omega_r}.
$$

We now come to our second family.  Let $\Omega\times[s_0,\infty)$ be the set of all triples $(\omega,s)=(G_0,G_1,s)$, where $G_0$ and $G_1$ are open subgroupoids of $G$ and $s\geq s_0$.  For $(G_0,G_1,s)\in \Omega\times [s_0,\infty)$ and $r\geq 0$, define a subspace $B^{\omega,s}_r$ of $A^s$ by
$$
B^{\omega,s}_r:=A^{s_0}(G_0^{+r})_r+A^{s_0}(G_1^{+r})_r+A^s(G_0^{+r}\cap G_1^{+r})_{sr}
$$
and define also
$$
B^{\omega,s}:=\overline{\bigcup_{r\geq 0}B^{\omega,s}_r}
$$
to be the closure of the union of the family $\{B^{\omega,s}_r\}_{r\geq 0}$ in the norm topology of $A^{s}$.  Define subspaces of $B^{\omega,s}_r$ by
$$
I^{\omega,s}_r:=A^{s_0}(G_0^{+r})_r+A^s(G_0^{+r}\cap G_1^{+r})_{sr} 
$$
and
$$
J^{\omega,s}_r:=A^{s_0}(G_1^{+r})_r+A^s(G_0^{+r}\cap G_1^{+r})_{sr}
$$
and finally define 
$$
I^{\omega,s}:=\overline{\bigcup_{r\geq 0}I^{\omega,s}_r}, \quad J^\omega:=\overline{\bigcup_{r\geq 0}J^{\omega,s}_r}.
$$
\end{definition}

\begin{remark}\label{fdc comp}
(This remark may be safely ignored by readers who do not know the earlier work).  Comparing our work in this paper to \cite{Guentner:2009tg}, the second filtration above plays an analogous role to the \emph{relative Rips complex} of \cite[Appendix A]{Guentner:2009tg}.
\end{remark}
 
Our aim is to show that the definitions above give us two uniformly excisive families in the sense of Definition \ref{uni ex}.  This requires some preliminaries on `partition of unity' type constructions.  

The next definition and lemma will be given in slightly more generality than we need as this does not complicate the proof, and maybe makes the statements slightly cleaner.

\begin{definition}\label{pou ops}
Let $K$ be a compact subset of $X$, let $\{U_0,...,U_n\}$ be a finite collection of open subsets of $X$ that cover $K$, and let $\{\phi_0,...,\phi_n\}$ be a subordinate partition of unity on $K$: precisely, each $\phi_i$ is a continuous function $X\to [0,1]$ with support contained in $U_i$, and for each $x\in K$ we have $\phi_0(x)+\cdots +\phi_n(x)=1$.  

Let $s\geq 0$ and recall the definitions of the Rips complex $P_{s}(\Gamma)$, barycentric coordinates $t_g:P_{s}(\Gamma)\to [0,1]$, and Hilbert space $H_{s}:=\ell^2(Z_{s}\times X,H\otimes \ell^2(\Gamma))$ from Section \ref{ass sec}.  For $i\in \{0,...,n\}$, let $M_i$ be the multiplication operator on $H_{s}$ associated to the function
$$
Z_{s}\times X\to [0,1],\quad (z,x)\mapsto \sum_{g\in \Gamma} t_g(z)\phi_i(gx).
$$
\end{definition}

For the next lemma, recall the notion of the support of an operator in $C^*(\Gamma\lefttorightarrow X;s)$ from Definition \ref{t supp} above.

\begin{lemma}\label{pou lem}
With notation as in Definition \ref{pou ops}, the operators $M_i$ have the following properties.
\begin{enumerate}[(i)]
\item \label{norm mi} Each $M_i$ has norm at most one.
\item \label{pou mi} If $T\in C^*(\Gamma\lefttorightarrow X;s)$ satisfies 
$$
\{x\in X\mid (gx,g,x)\in \text{supp}(T) \text{ for some } g\in \Gamma\}\subseteq K,
$$
then $T=T(M_0+\cdots +M_n)$.
\item \label{supp mi} For any $T\in C^*(\Gamma\lefttorightarrow X;s)$ and $i\in \{0,...,n\}$, 
$$
\text{supp}(TM_i)\subseteq \Bigg\{(gx,g,x)\in \Gamma\ltimes X~\Big|~ x\in \bigcup_{|h|\leq s}h\cdot U_i\Bigg\}\cap \text{supp}(T).
$$ 
\end{enumerate}
\end{lemma}

\begin{proof}
Part \eqref{norm mi} follows as each $M_i$ is a multiplication operator associated to a function with range contained in $[0,1]$.  For use in the remainder of the proof, note that for $i\in \{0,...,n\}$, any $T\in  C^*(\Gamma\lefttorightarrow X;s)$  and any $y,z\in Z_s$ and $x\in X$,
\begin{equation}\label{m coeffs}
(TM_i)_{y,z}(x)=T_{y,z}(x)\cdot \sum_{h\in \Gamma}t_h(z)\phi_i(hx).
\end{equation}
Hence 
\begin{equation}\label{sum mcs}
T(M_0+\cdots + M_n)_{y,z}(x)=T_{y,z}(x)\cdot \sum_{h\in \Gamma}t_h(z)(\phi_0(hx)+\cdots +\phi_n(hx)).
\end{equation}
Assume now $T$ satisfies the support condition in part \eqref{pou mi}.  If $T_{y,z}(x)=0$, then clearly the above is zero.  Otherwise, if $T_{y,z}(x)\neq 0$, then $(gx,gh^{-1},hx)\in \text{supp}(T)$ for all $g\in \text{supp}(y)$ and $h\in \text{supp}(z)$.  In particular, $hx\in K$ for all $h\in \text{supp}(z)$, and so 
$$
\sum_{h\in \Gamma}t_h(z)(\phi_0(hx)+\cdots+\phi_n(hx))=\sum_{h\in \Gamma}t_h(z)=1,
$$
using that $\{\phi_0,...,\phi_n\}$ is a partition of unity on $K$; combined with line \eqref{sum mcs}, this gives part \eqref{pou mi}.

For part \eqref{supp mi}, say $(gx,gk^{-1},kx)\in \text{supp}(TM_i)$ for some $T\in C^*(\Gamma\lefttorightarrow X;s)$.  Hence there are $y,z\in Z_s$ with $g\in \text{supp}(y)$, $k\in \text{supp}(z)$ and $(TM_i)_{y,z}(x)\neq 0$.  From line \eqref{m coeffs}, this implies that $T_{y,z}(x)\neq 0$, whence $(gx,gk^{-1},kx)\in \text{supp}(T)$.  On the other hand, we must also have 
$$
\sum_{h\in \Gamma}t_h(z)\phi_i(hx)\neq 0,
$$
whence there is $h\in \text{supp}(z)$ with $\phi_i(hx)\neq 0$, and thus $hx$ is in $U_i$.  As $h$ and $k$ are both in $\text{supp}(z)$ and $z$ is in $Z_s$, this forces $|kh^{-1}|\leq s$.  On the other hand, $kx=(kh^{-1})hx$ is in $(kh^{-1})U_i$, which completes the proof.
\end{proof}

We are now ready to prove that the families from Definition \ref{uni algs} are uniformly excisive.

\begin{lemma}\label{gpd exc}
Fix an open subgroupoid $G$ of $\Gamma\ltimes X$, and $s_0\geq 1$.  Let $\Omega$ be the set of all pairs $(G_0,G_1)$ of open subgroupoids of $G$ as in Definition \ref{uni algs}.   Then with notation as in that definition, the collections 
$$
(I^\omega,J^\omega;B^\omega)_{\omega\in \Omega}\quad \text{and}\quad (I^{\omega,s},J^{\omega,s};B^{\omega,s})_{(\omega,s)\in \Omega\times [s_0,\infty)}
$$
are uniformly excisive.
\end{lemma}

\begin{proof} 
Each subspace $B^\omega_r$ of $A^{s_0}$ is self-adjoint as it is a sum of subspaces that are themselves self-adjoint by Lemma \ref{supp alg}.  It is clear that $B^\omega_{r_0}\subseteq B^\omega_r$ for $r_0\leq r$, and the union $\bigcup_{r\geq 0}B^\omega_r$ is dense in $B^\omega$ by definition.  These observations apply similarly for $B^{\omega,s}_r$, and for $I^\omega_r$, $J^\omega_r$, $I^{\omega,s}_r$, and $J^{\omega,s}_r$.  To complete the proof that $B^\omega$ is filtered and $I^\omega$, $J^\omega$ are filtered ideals (and similarly for the $s$-decorated versions) we must look at products.  Let $r_1,r_2\geq 0$.  First, note that the inclusions
$$
A^{s_0}(G_i^{+r_1})_{r_1}\cdot A^{s_0}(G_i^{+r_2})_{r_2}\subseteq A^{s_0}(G_i^{+(r_1+r_2)})_{r_1+r_2}, \quad \text{where } i\in \{0,1\},
$$
$$
A^{s_0}(G_0^{+r_1}\cap G_1^{+r_1})_{r_1}\cdot A^{s_0}(G_0^{+r_2}\cap G_1^{+r_2})_{r_2}\subseteq A^s(G_0^{+(r_1+ r_2)}\cap G_1^{+(r_1+r_2)})_{r_1+r_2},
$$
and 
$$
A^s(G_0^{+r_1}\cap G_1^{+r_1})_{r_1s}\cdot A^s(G_0^{+r_2}\cap G_1^{+r_2})_{r_2s}\subseteq A^s(G_0^{+(r_1+ r_2)}\cap G_1^{+(r_1+r_2)})_{(r_1+r_2)s}
$$
follow directly from Lemmas \ref{supp alg} and \ref{filt ext}.  Similarly, 
$$
A^{s_0}(G_0^{+r_1})_{r_1}\cdot A^{s_0}(G_1^{+r_2})_{r_2}\subseteq A^{s_0}(G_0^{+(r_1+ r_2)})_{r_1+r_2}\cap A^{s_0}(G_1^{+(r_1+ r_2)})_{r_1+r_2}
$$
and finally, using that $A^{s_0}(G)_r\subseteq A^s(G)_r$ for any open $G$ and $r\geq 0$ and that $s\geq s_0\geq 1$,
\begin{align*}
A^{s_0}(G_i^{+r_1})_{r_1}\cdot A^s(G_0^{+r_2}\cap G_1^{+r_2})_{r_2s} &\subseteq  A^s((G_0^{+r_2}\cap G_1^{+r_2})^{+r_1})_{r_2s+r_1} \\
& \subseteq A^s((G_0^{+r_2}\cap G_1^{+r_2})^{+r_1})_{(r_2+r_1)s}\\
 & \subseteq A^s((G_0^{+r_2})^{+r_1}\cap (G_1^{+r_2})^{+r_1})_{(r_1+r_2)s} \\ &\subseteq A^s(G_0^{+(r_1+r_2)}\cap G_1^{+(r_1+r_2)})_{(r_1+r_2)s}
\end{align*}
for $i\in \{0,1\}$, where the last step uses Lemma \ref{twice}.  Combining the last four displayed lines completes the check that both $B^\omega$ and $B^{\omega,s}$ are filtered.  Moreover, they show that $I^\omega$ and $J^\omega$ are filtered $C^*$-ideals in $B^\omega$, and similarly for the $s$-decorated versions.\\

We now have to check that the collection $(I^\omega,J^\omega;B^\omega)_{\omega\in \Omega}$ is uniformly excisive as in Definition \ref{uni ex}, and similarly for the $s$-decorated versions.  For notational simplicity, we will ignore the copy of the compact operators appearing in Definition \ref{uni ex}: the reader can check this makes no real difference to the proof.

Look first at part \eqref{uni ex i} of the definition.  Let $U_i$ be the unit space of $G_i^{+r_0}$ for $i\in \{0,1\}$.  Say $a$ is an element of $B^\omega_{r_0}$ whence 
$$
K:=\overline{\{x\in X\mid (gx,g,x)\in \text{supp}(a(t)) \text{ for some } t\in [0,\infty),g\in \Gamma\}},
$$ 
is a compact subset of $U_0\cup U_1$ (see Definitions \ref{smallob} and \ref{main ob} above).  Let $M_0$, $M_1$ be as in Definition \ref{pou ops} with respect to the Rips complex $P_s(\Gamma)$, the compact set $K$, the open sets $U_0$ and $U_1$, and some choice of partition of unity $\{\phi_0,\phi_1\}$.  From Lemma \ref{pou lem} part \eqref{norm mi} we have that $\|M_0\|\leq 1$ and $\|M_1\|\leq 1$, and from part \eqref{pou mi} that $a(t)(M_0+M_1)=a(t)$ for all $t$.  Hence to complete the proof that our algebras satisfy part \eqref{uni ex i} of Definition \ref{uni ex}, it suffices to show that there exists $r\geq 0$ (which is allowed to depend on $r_0$ and $s_0$,  but not on any of the other data) such that for each $\omega$, $t\mapsto a(t)M_0$ is in $I^\omega_r$ and $t\mapsto a(t)M_1$ is in $J^\omega_r$.  We focus on the case of $M_0$ and $I^\omega$; the other case is similar.   We claim that in fact $r=2r_0+s_0$ works.

Write $a=b_0+b_1+c$, where $b_i\in A^{s_0}(G_i^{+r_0})_{r_0}$ and $c\in A^{s_0}(G_0^{+r_0}\cap G_1^{+r_0})_{r_0}$. Part \eqref{supp mi} of Lemma \ref{pou lem} implies that for each $t\in [0,\infty)$
$$
\supp(b_0(t)M_0)\subseteq \supp(b_0(t)) \quad \text{and}\quad \supp(c(t)M_0)\subseteq \supp(c(t)),
$$  
from which it follows straightforwardly that $t\mapsto b_0(t)M_0$ and $t\mapsto c(t)M_0$ are in $A^{s_0}(G_0^{+r_0})_{r_0}$ and $A^{s_0}(G_0^{+r_0}\cap G_1^{+r_0})_{r_0}$ respectively; moreover, these are subspaces of $I^\omega_r$.  To complete the proof, we check that $t\mapsto b_1(t)M_0$ is in $I^\omega_r$.  

Assume that $(gx,gh^{-1},hx)$ is in the support of $b_1(t)M_0$ for some $t$, and write $T=b_1(t)$ for ease of notation.  Then there exist $y,z\in P_{s_0}(\Gamma)$ with $g\in \text{supp}(y)$, $h\in \text{supp}(z)$ and $(TM_0)_{y,z}(x)\neq 0$.  Hence from line \eqref{m coeffs} we must have that $T_{y,z}(x)\neq 0$ and so $y,z$ are actually in $P_{s_0}(\Gamma)$ and $|gh^{-1}|\leq r_0\leq r$.  Moreover,
$$
\sum_{k\in \Gamma} t_k(z)\phi_0(kx)\neq 0,
$$
whence there is $k\in \text{supp}(z)$ such that $\phi_0(kx)\neq 0$, and in particular, $kx$ is in the unit space of $G_0^{+r_0}$.  On the other hand,
$$
(gx,gh^{-1},hx)=(gx,gk^{-1},kx)(kx,kh^{-1},hx).
$$ 
The first factor in the product is in $(G_0^{+r_0})^{+r_0}\subseteq G_0^{+2r_0}$ using that $|gk^{-1}|\leq r_0$, that $kx$ is in the unit space of $G_0^{+r_0}$, and Lemma \ref{twice}.  The second factor is in $(G_0^{+r_0})^{+s_0}\subseteq G_0^{+(r_0+ s_0)}$ using that $|kh^{-1}|\leq s_0$, that $kx$ is in the unit space of $G_0^{+r_0}$, and Lemma \ref{twice}.  Hence $(gx,gh^{-1},hx)$ is in $G_0^{+(2r_0+s_0)}$.  To summarize, we have shown at this point that $b_1(t)M_0$ has support in $G_0^{+(2r_0+s_0)}$ for all $t$, and thus $t\mapsto b_1(t)M_0$ is in $B^{s_0}(G_0^{+r})_{r}$ as claimed.  The $s$-decorated case can be handled precisely analogously.\\

We now look at part \eqref{uni ex ii} of Definition \ref{uni ex}.  Say $a$ is in both $I^\omega$ and $J^\omega$, and that $a$ is within $\delta:=\epsilon/3$ of both $I^\omega_{r_0}$ and $J^\omega_{r_0}$.  Let $a_0$ and $a_1$ be elements of $I^\omega_{r_0}$ and $J^\omega_{r_0}$ respectively which are at most $\delta$ away from $a$.  Define 
$$
K:=\overline{\{x\in X\mid (gx,g,x)\in \supp(a_0(t))\text{ for some } t\in [0,\infty),g\in \Gamma\}},
$$
a compact subset of the unit space $U_0$ of $G_0^{+r_0}$.  Let $M_0$ be as in Definition \ref{pou ops} with respect to the Rips complex $P_{s_0}(\Gamma)$, the compact set $K$, and the open cover $\{U_0\}$ of $K$.  From Lemma \ref{pou lem} parts \eqref{norm mi} and \eqref{pou mi} we have that $\|M_0\|\leq 1$ and that $a_0(t)M_0=a_0(t)$ for all $t$.  Hence for any $t\in [0,\infty)$
\begin{align*}
\|a(t)&-a_1(t)M_0\| \\ & \leq \|a(t)-a_0(t)\|+\|a_0(t)-a(t)M_0\|+\|a(t)M_0-a_1(t)M_0\| \\ &<\epsilon.
\end{align*}
On the other hand, an argument precisely analogous to the one used above to establish part \eqref{uni ex i} of Definition \ref{uni ex} shows that $a_1(t)M_0$ is contained in $I^\omega_r\cap J^\omega_r$, where $r=2r_0+s_0$, and we are done.  The $s$-decorated case can again be handled analogously.
\end{proof}

We need one more preliminary lemma before the proof of Proposition \ref{indstep}.  The proof is similar to (and simpler than) the part of the proof of Lemma \ref{gpd exc} above that establishes part \eqref{uni ex i} of Definition \ref{uni ex}, and is therefore omitted.  

\begin{lemma}\label{gpd contain}
Fix an open subgroupoid $G$ of $\Gamma\ltimes X$ and $s_0\geq 0$.  Let $r_0\geq 0$.  Let $G^{(0)}=U_0\cup U_1$ be an open cover of $G^{(0)}$, and for $i\in \{0,1\}$, let $G_i$ be the subgroupoid of $G$ generated by 
$$
\{(gx,g,x)\in G\mid x\in U_i,~|g|\leq r_0\}.
$$
Then $A^{s_0}(G)_{r_0}\subseteq A^{s_0}(G_0^{+(2r_0+s_0)})_{2r_0+s_0}+A^{s_0}(G_1^{+(2r_0+s_0)})_{2r_0+s_0}$.  \qed
\end{lemma}

Finally, we are ready for the proof of Proposition \ref{indstep}, the last step we need in the proof of Theorem \ref{main}.

\begin{proof}[Proof of Proposition \ref{indstep}] 
Let $\mathcal{B}$ be the class of open subgroupoids $G$ of $\Gamma\ltimes X$ such that the conclusion of Proposition \ref{indstep} holds for all open subgroupoids of $G$, in such a way that the resulting constant $s$ depends only on $G$ and not on the particular open subgroupoid under consideration.  It will suffice to show that $\mathcal{B}$ contains $\mathcal{D}$.  For this it suffices to show that $\mathcal{B}$ contains all relatively compact open subgroupoids of $\Gamma\ltimes X$, and that it is closed under decomposability.  

Let then $G$ be an open subgroupoid of $\Gamma\ltimes X$ with compact closure, and let $r_0$ and $s_0$ be given.   As $G$ has compact closure, the number 
$$
s_1=\max\{|g|\mid (gx,g,x)\in G\},
$$ 
is finite.  Let $s=\max\{r_0,s_0,s_1\}$; we claim this $s$ has the right property.  We have that $A^s(G)_s=A^s(G)$ and so
$$
K_*^{s,1/8}(A^{s}(G))=K_*(A^s(G))
$$
by Proposition \ref{comparelem}.  Moreover, the group on the right hand side is zero by Proposition \ref{basecase}, and so in particular the map
$$
K_*^{r_0,1/8}(A^{s_0}(G))\to K_*^{s,1/8}(A^{s}(G))=K_*(A^s(G))
$$
is certainly zero.  Moreover, the same $s$ clearly works for any open subgroupoid of $G$.  Hence $G$ is in $\mathcal{B}$ as required.

Now let $G$ be an open subgroupoid of $\Gamma\ltimes X$ that decomposes over $\mathcal{B}$, and let $r_0,s_0\geq 0$ be given; we may assume that $s_0\geq 1$.  Let $r_1=r_1(2r_0+s_0,s_0)\geq r_0$ be the constant given by Proposition \ref{con mv} with respect to the uniformly excisive families $(I^\omega,J^\omega;B^\omega)_{\omega\in \Omega}$ and $(I^{\omega,s},J^{\omega,s};B^{\omega,s})_{(\omega,s)\in \Omega\times[s_0,\infty)}$ from Definition \ref{uni algs}; we may assume that $r_1\geq 1$.  Let $r_2=r_2(r_1,s_0)\geq r_1$ be the constant given by Proposition \ref{con mv} for the uniformly excisive collection $(I^{\omega,s},J^{\omega,s};B^{\omega,s})_{(\omega,s)\in \Omega\times[s_0,\infty)}$.  Let $G^{(0)}=U_0\cup U_1$ be an open cover with the property that if $G_i$ is the subgroupoid of $G$ generated by 
$$
\{(gx,g,x)\in G\mid x\in U_i,~|g|\leq r_0\},
$$
then $G_i^{+r_2}$ is in the class $\mathcal{B}$ (and therefore that $G_0^{+r_1}\cap G_1^{+r_1}$ is in $\mathcal{B}$ too, as $\mathcal{B}$ is closed under taking open subgroupoids).  Let the constant $s_1=s_1(r_0,s_0,G_0^{+r_1}\cap G_1^{+r_1})\geq \max\{r_0,s_0\}$ be as in the inductive hypothesis for the groupoid $G_0^{+r_1}\cap G_1^{+r_1}$.  Finally, let $s=s(r_2+s_1,s_1,G_0^{+r_2},G_1^{+r_2})$ be as in the inductive hypothesis for both of the groupoids $G_i^{+r_2}$ simultaneously (this is possible, as if $s$ has the right property for some groupoid, then clearly any $s'\geq s$ works too).  We claim that this $s$ has the right properties.

Let then $x$ be an element of $K_*^{r_0,1/8}(A^{s_0}(G))$.  Using Lemma \ref{gpd contain} we have a subspace inclusion map 
$$
K_*^{r_0,1/8}(A^{s_0}(G))\to K_*^{2r_0+s_0,1/8}(B^\omega(G)),
$$  
where we have used the notation of Definition \ref{uni algs} and written $\omega=(G_0,G_1)\in \Omega$.  Hence we may consider $x$ as an element of $K_*^{2r_0+s_0,1/8}(B^\omega(G))$.  Using Proposition \ref{con mv}, we have a commutative diagram of controlled boundary maps
$$
\xymatrix{ K_*^{2r_0+s_0,1/8}(B^\omega) \ar[d] \ar[r]^-{\partial_c} &  K_*^{r_1}(I^\omega\cap J^\omega) \ar[d] \\
K_*^{2r_0+s_0,1/8}(B^{\omega,s_1}) \ar[r]^-{\partial_c} &  K_*^{r_1}(I^{\omega,s_1}\cap J^{\omega,s_1}) }.
$$
The definition of the algebras involved implies that the right hand vertical map identifies with the forget control map
$$
K_*^{r_1}(A^{s_0}(G^{+r_1}_0\cap G^{+r_1}_1))\to K_*^{s_1r_1}(A^{s_1}(G^{+r_1}_0\cap G^{+r_1}_1)),
$$
which is zero by hypothesis and the fact that $r_1\geq 1$.  Hence the image of $x$ in $K_*^{r_1}(I^{\omega,s_1}\cap J^{\omega,s_1})$ is zero.  

We now apply Proposition \ref{con mv} to get 
$$
y\in K^{r_2,1/8}(I^{\omega,s_1}),\quad z\in K^{r_2,1/8}(J^{\omega,s_1})
$$
such that $x=y+z$ inside $K_*^{r_2,1/8}(B^{\omega,s_1})$.  Consider the commutative diagram
$$
\xymatrix{ & x\in K_*^{2r_0+s_0,1/8}(B^\omega) \ar[d] \\ 
 K^{r_2,1/8}(I^{\omega,s_1}) \oplus K^{r_2,1/8}(J^{\omega,s_1})\ar[r] \ar[d] & K_*^{r_2,1/8}(B^{\omega,s_1}) \ar[d]   \\
K^{r_2+s_1,1/8}(A^{s_1}(G_0^{+r_2}))\oplus K^{r_2+s_1,1/8}(A^{s_1}(G_1^{+r_2})) \ar[d] \ar[r] & K^{r_2+s_1,1/8}(A^{s_1}(G)) \ar[d]  \\
K^{s,1/8}(A^{s}(G_0^{+r_2}))\oplus K^{s,1/8}(A^{s_1}(G_1^{+r_2}))  \ar[r] & K^{s,1/8}(A^{s}(G))  }
$$
where the horizontal maps are defined by taking sums, and the vertical maps by inclusion of the various subspaces involved.  Note that $y$ and $z$ both go to zero under the lower vertical map on the left hand side by inductive hypothesis and the choice of $s$.  Hence $x$ goes to zero in the bottom right group as it is equal to $y+z$ there, and we are done for $G$ itself.  A precisely analogous argument works for any open subgroupoid of $G$, completing the proof.
\end{proof}

\appendix

\section{Finite dynamical complexity for \'{e}tale groupoids}\label{fdc sec}

Our goals in this appendix are: to relate finite dynamical complexity to finite decomposition complexity in the sense of Guentner, Tessera, and Yu \cite{Guentner:2009tg,Guentner:2013aa}; to show that finite dynamical complexity implies topological amenability of the underlying action;  and to collect together several open questions.  This material is not necessary to read the main body of the paper, but provides some useful context, and also shows that many examples of groupoids with finite dynamical complexity exist.  

\subsection*{Finite decomposition complexity}

We will give a convenient definition of finite decomposition complexity in \ref{classical fdc}, adapted slightly from \cite[Definition 2.1.3]{Guentner:2013aa}.  This needs some preliminaries.  We will write `$A=B\bigsqcup C$' to mean that a set $A$ is the disjoint union of subsets $B$ and $C$, and similarly for unions of more than two subsets.  As in \cite[Section 2]{Guentner:2013aa}, if $Z$ and $\{Z_i\}_{i\in I}$ are subspaces of a metric space $X$, then the notation
$$
Z=\bigsqcup_{i,~r\text{-}\text{disjoint}}Z_i
$$
means that $Z$ is the disjoint union of the $Z_i$, and that $d(Z_i,Z_j)>r$ for $i\neq j$.  

\begin{definition}\label{fam subsp}
Let $X$ be a metric spac (with finite-valued metric).  A collection of subsets $\mathcal{Y}$ is a \emph{disjoint family} if no two elements of $\mathcal{Y}$ intersect.  Given a disjoint family $\mathcal{Y}$, we associate a metric space $X_\mathcal{Y}$ by taking the underlying set to be 
$$
X_{\mathcal{Y}}:=\bigsqcup_{Y\mathcal{Y}}Y,
$$
and equipping $X_\mathcal{Y}$ with the (possibly infinite-valued) metric
$$
d_{\mathcal{Y}}(x,y):=\left\{\begin{array}{ll} d_X(x,y) & x,y\in Y \text{ for some } Y\in \mathcal{Y} \\ 0 & \text{otherwise} \end{array}\right.
$$
(in words, the metric agrees with that from $X$ on each `component' subset $Y\in \mathcal{Y}$, and sets the distance between distinct `components' to be infinity).  

Finally, for $r>0$ the \emph{$\mathcal{Y}$-neighbourhood} of a subset $Z$ of $X_{\mathcal{Y}}$ is defined to be 
$$
N_{r,\mathcal{Y}}(Z):=\{y\in X_{\mathcal{Y}}\mid d_{\mathcal{Y}}(y,z)<r \text{ for some } z\in Z\}
$$
\end{definition}

\begin{definition}\label{classical fdc}
Let $X$ be a metric space (with finite-valued metric).  A disjoint family $\mathcal{Y}$ of subspaces of $X$ is \emph{uniformly bounded} if $\text{sup}_{Y\in \mathcal{Y}}\text{diam}(Y)$ is finite. 

Let $\mathcal{C}$ be a collection of disjoint families of subspaces of $X$.  A disjoint family of subspaces $\mathcal{Y}$ is \emph{decomposable} over $\mathcal{C}$ if for all $r\geq 0$ there exist disjoint families $\mathcal{Z}_0,\mathcal{Z}_1\in \mathcal{C}$ such that for all $Y\in \mathcal{Y}$ there exists a decomposition
$$
Y=Y_0\cup Y_1
$$
and further decompositions
$$
Y_i=\bigsqcup_{j\in J_{Y,i},~2r-\text{disjoint}}Y_{ij}
$$
such that for each $j$, $N_{r,\mathcal{Y}}(Y_{ij})$ is in $\mathcal{Z}_i$.  

Define $\mathcal{D}_m$ to be smallest collection of disjoint families of subspaces of $X$ that: contains the uniformly bounded disjoint families; and is closed under decomposability.  The metric space $X$ has \emph{finite decomposition complexity} if the singleton family $\{X\}$ is contained in $\mathcal{D}_m$.
\end{definition}

Using the discussion in \cite[3.1.3]{Guentner:2013aa}, it is not too difficult to see that the above definition is equivalent to \cite[Definition 2.1.3]{Guentner:2013aa}.

Here is the first main goal of this appendix.

\begin{theorem}\label{fdc vs fdc}
Let $\Gamma$ be a countable discrete group, equipped as usual with a metric arising from a proper length function.  Then the following are equivalent:
\begin{enumerate}[(i)]
\item $\Gamma$ has finite decomposition complexity;
\item the canonical action of $\Gamma$ on its Stone-\v{C}ech compactification has finite dynamical complexity in the sense of Definition \ref{gpd fdc}.
\end{enumerate}
\end{theorem}

We will actually prove this in a little more generality, more because this makes the proof more conceptual than because we want the generality for its own sake.  Throughout the rest of this section, then, we will work in the context of \'{e}tale\footnote{We will always assume our groupoids are locally compact and Hausdorff, and do not repeat these assumptions.} groupoids: our conventions here match those of \cite[Section 5.6]{Brown:2008qy}, so in particular we will write $G$ for an \'{e}tale groupoid, $G^{(0)}$ for its unit space, $s,r:G\to G^{(0)}$ for the source and range maps, and for $x\in G^{(0)}$, $G_x$ and $G^x$ denote $s^{-1}(x)$ and $r^{-1}(x)$ respectively.  A pair of elements $(g,h)\in G\times G$ is composable if $s(g)=r(h)$, and their product or composition is then written $gh$.

Here is the definition of finite dynamical complexity for general \'{e}tale groupoids.  

\begin{definition}\label{real gpd fdc}
Let $G$ be an \'{e}tale groupoid, let $H$ be an open subgroupoid of $G$, and let $\mathcal{C}$ be a set of open subgroupoids of $G$.  We say that $H$ is \emph{decomposable} over $\mathcal{C}$ if for any open, relatively compact subset $K$ of $H$ there exists an open cover $H^{(0)}=U_0\cup U_1$ of the unit space of $H$ such that for each $i\in \{0,1\}$ the subgroupoid of $H$ generated by 
$$
\{h\in K\mid s(h)\in U_i\}
$$
is in $\mathcal{C}$. 

An open subgroupoid of $G$ (for example, $G$ itself) has \emph{finite dynamical complexity} if it is contained in the smallest set $\mathcal{D}_g$ of open subgroupoids of $G$ that: contains all relatively compact open subgroupoids; and is closed under decomposability.
\end{definition}

We leave the elementary check that this reduces to Definition \ref{gpd fdc} in the case that $G=\Gamma\ltimes X$ for some action $\Gamma\lefttorightarrow X$ to the reader: compare \cite[Lemma 5.4]{Guentner:2014aa}.  

The following basic lemma will also be left to the reader: compare part (i) of Lemma \ref{fdc lem} above for the part about groupoids and \cite[3.1.3]{Guentner:2013aa} for the part about spaces.

\begin{lemma}\label{subgpd fdc}
\begin{enumerate}[(i)]
\item Let $G$ be an \'{e}tale groupoid, and $H$ an open subgroupoid in the class $\mathcal{D}_g$ of Definition \ref{real gpd fdc}.  Then all open subgroupoids of $H$ are also contained in $\mathcal{D}_g$.
\item Let $X$ be a metric space, and let $\mathcal{Y}$ be a family of subspaces of $X$ in the class $\mathcal{D}_m$ of Definition \ref{classical fdc}.  Let $\mathcal{Z}$ be another family of subspaces of $X$ such that each $Z\in \mathcal{Z}$ is contained in some element of $\mathcal{Y}$.  Then $\mathcal{Z}$ is also in $\mathcal{D}_m$. \qed
\end{enumerate}
\end{lemma}

We will look at a particular class of groupoids arising from discrete metric spaces: we will assume such metric spaces have \emph{bounded geometry} meaning that for all $r\in [0,\infty)$, the cardinality of all $r$-balls in the space is uniformly bounded. Recall that we allow our metrics to be infinite valued.

The following groupoids were introduced by Skandalis, Tu, and Yu \cite{Skandalis:2002ng}; see also \cite[Chapter 10]{Roe:2003rw}.

\begin{definition}\label{coarse gpd}
Let $X$ be a bounded geometry metric space (possibly with infinite-valued metric), and let $\beta X$ be its Stone-\v{C}ech compactification.  For each $r\in [0,\infty)$, let 
$$
E_r=\{(x,y)\in X\times X\mid d(x,y)\leq r\}.
$$
As $X$ is a subspace of $\beta X$ we may identify $E_r$ with a subspace of $\beta X\times \beta X$, and take its closure $\overline{E}_r$.  The \emph{coarse groupoid} of $X$ is the union
$$
G(X):=\bigcup_{r\in [0,\infty)}\overline{E}_r
$$
equipped with the restriction of the pair groupoid operations it inherits as a subset of $\beta X\times \beta X$, and with the weak topology it inherits from the union above\footnote{This means that a subset $U$ of $G(X)$ is defined to be open exactly when $U\cap \overline{E}_r$ is open for each $r\in [0,\infty)$; this is not the same as the subspace topology from $\beta X\times \beta X$.}, when each $\overline{E_r}$ is given the subspace topology from $\beta X\times \beta X$.
\end{definition} 

The groupoids $G(X)$ are always locally compact, Hausdorff, \'{e}tale and $\sigma$-compact; this is not too hard to check directly, or see the references \cite{Skandalis:2002ng} or \cite[Chapter 10]{Roe:2003rw} from above.  Moreover, if $X=\Gamma$ is a discrete group equipped with a metric as in Definition \ref{length fun} above, then $G(X)$ is canonically isomorphic to $\Gamma\ltimes \beta \Gamma$: see \cite[Proposition 3.4]{Skandalis:2002ng}.  Note that if $\mathcal{Y}$ is a disjoint family of subspaces of $X$ and $X_{\mathcal{Y}}$ is the associated metric space as in Definition \ref{fam subsp}, then $G(X_{\mathcal{Y}})$ identifies naturally with a (closed and open) subgroupoid of $G(X)$; we will always make this identification in what follows.  

Here then is the theorem we will actually prove.  From the comments above, it clearly implies Theorem \ref{fdc vs fdc}.

\begin{theorem}\label{gen fdc vs fdc}
Let $X$ be a bounded geometry metric space (with finite-valued metric).  Then the following are equivalent:
\begin{enumerate}[(i)]
\item $X$ has finite decomposition complexity;
\item the coarse groupoid $G(X)$ has finite dynamical complexity.
\end{enumerate}
\end{theorem}

\begin{proof}
To show that (i) implies (ii), it will suffice to show that if $\mathcal{Y}$ is in $\mathcal{D}_m$, then $G(X_\mathcal{Y})$ is in $\mathcal{D}_g$.  For this, it suffices to show that the collection 
\begin{equation}\label{tech fam}
\{\mathcal{Y}\mid G(X_\mathcal{Y})\in \mathcal{D}_g\}
\end{equation}
of disjoint families of subspaces of $X$ contains the uniformly bounded families, and is closed under decomposability of metric families: indeed, this implies that the family in line \eqref{tech fam} contains $\mathcal{D}_m$ by definition of $\mathcal{D}_m$, and thus that it contains $\{X\}$ by assumption (i); hence $G(X)$ is in $\mathcal{D}_g$, which is the required conclusion.

Say first then that $\mathcal{Y}$ is a disjoint family of uniformly bounded subspaces of $X$, say all with diameters at most $s$.  Then  $G(X_\mathcal{Y})$ is contained in the compact set $\overline{E_s}\subseteq G(X)$ from Definition \ref{coarse gpd}, whence $G(X_\mathcal{Y})$ has compact closure and is thus in $\mathcal{D}_g$.  

To complete the proof of (i) implies (ii), it remains to show that the collection in line \eqref{tech fam} is closed under decomposability of disjoint families.  Say then that $\mathcal{Y}$ is a disjoint family of subspaces of $X$ that decomposes over the collection of disjoint families in line \eqref{tech fam}.  We will show that $G(X_\mathcal{Y})$ decomposes over $\mathcal{D}_g$, which will suffice to show that $\mathcal{Y}$ is in the collection of families in line \eqref{tech fam}.  Let then $K$ be an open, relatively compact subset of $G(X_\mathcal{Y})$.  As $G(X)$ is the union of the (compact) open subsets $\{\overline{E_r}\}_{r\geq 0}$, there is $r\geq 0$ with $K$ contained in $\overline{E_r}$.   As in Definition \ref{classical fdc}, there are families $\mathcal{Z}_0,\mathcal{Z}_1$ in the set in line \eqref{tech fam} such that every $Y\in\mathcal{Y}$ admits a decomposition $Y=Y_{0}\cup Y_{1}$ such that each $Y_i$ further decomposes as 
$$
Y_i=\bigsqcup_{j\in J_{Y,i},~2r\text{-}\text{disjoint}}Y_{ij}
$$
with each $r$-neighbourhood $N_{r,\mathcal{Y}}(Y_{ij})$ in $\mathcal{Z}_i$.  Now, let $\mathcal{Y}_i$ be the family of subspaces $\{N_{r,\mathcal{Y}}(Y_{ij})\mid j\in J_{Y,i},~Y\in \mathcal{Y}\}$.  Let $U_i$ be the closure of the set 
$$
\bigcup_{Y\in \mathcal{Y},j\in J_{Y,i}}Y_{ij}
$$
in $G(X_\mathcal{Y})^{(0)}$, which is a (closed and) open set.  Then $\{U_0,U_1\}$ is an open cover of $G(X_\mathcal{Y})^{(0)}$ and it is not too difficult to check that the subgroupoid $H_i$ of $G(X_\mathcal{Y})$ generated by 
$$
\{g\in G(X_\mathcal{Y})\mid  s(g)\in U_i,~g\in K\}
$$
is contained in $G(X_{\mathcal{Y}_i})$.  As $\mathcal{Z}_i$ is in the collection in line \eqref{tech fam}, we have that $G(X_{\mathcal{Z}_i})$ is in $\mathcal{D}_g$; moreover, each $G(X_{\mathcal{Y}_i})$ is contained in $G(X_{\mathcal{Z}_i})$  and so $G(X_{\mathcal{Y}_i})$ is in $\mathcal{D}_g$ by Lemma \ref{subgpd fdc}, and so each $H_i$ is also in $\mathcal{D}_g$ by the same lemma again.  This completes the proof that $G(X_\mathcal{Y})$ decomposes over $\mathcal{D}_g$, and thus the proof of (i) implies (ii).\\

To prove (ii) implies (i), it will be helpful to first introduce some notation.  If $H$ is an open subgroupoid of $G(X)$, let $\sim$ be the equivalence relation on $X\cap H^{(0)}$ defined by $x\sim y$ if $(x,y)$ is an element of $H$, and let $\mathcal{X}_H$ be the disjoint family of equivalence classes for this equivalence relation.  

Now, to prove (ii) implies (i), it will suffice to show that if $G(X_\mathcal{Y})$ is in $\mathcal{D}_g$, then $\mathcal{Y}$ is in $\mathcal{D}_m$.  For this it suffices to show that the collection 
\begin{equation}\label{tech fam 2}
\{H \mid \mathcal{X}_H\in \mathcal{D}_m\}
\end{equation}
of open subgroupoids of $G(X)$ contains the relatively compact open subgroupoids, and is closed under decomposability of groupoids: indeed, given this, the collection in line \eqref{tech fam 2} contains $\mathcal{D}_g$, whence in particular it contains $G(X)$ by assumption (ii); however, $\mathcal{X}_{G(X)}=\{X\}$ so this gives that $\{X\}$ is in $\mathcal{D}_m$, and so we are done at that point. 

Say first then that $H$ is a relatively compact open subgroupoid of $G(X)$.  Then as the collection $\{\overline{E_r}\}_{r\geq 0}$ of (compact) open subsets covers $G(X)$, there must exist $s\geq 0$ with $H\subseteq \overline{E_s}$.  This implies that every $Y\in \mathcal{X}_H$ has diameter at most $s$, and thus $\mathcal{X}_H$ is in $\mathcal{D}_m$ and so $H$ is contained in the collection in line \eqref{tech fam 2}.

It remains to show that the collection in line \eqref{tech fam 2} is closed under decomposability of groupoids.  Let then $H$ be an open subgroupoid of $G(X)$ that decomposes over the collection in line \eqref{tech fam 2}.  We will aim to show that $\mathcal{X}_H$ decomposes over $\mathcal{D}_m$, and thus that $\mathcal{X}_H$ is in $\mathcal{D}_m$, and so $H$ is in the collection in line \eqref{tech fam 2}.  Let then $r\in [0,\infty)$ be given, and let $K$ be the compact subset $\overline{E_{2r}}$ of $G(X)$.  The definition of decomposability of $H$ gives us an open cover $H^{(0)}=U_0\cup U_1$ of the unit space of $H$ such that the subgroupoids $H_i$ of $G$ generated by 
\begin{equation}\label{hj gen set}
\{g\in G\mid s(g)\in U_i,~g\in K\}
\end{equation}
are in the family in line \eqref{tech fam 2}.  Let $\mathcal{X}_{H_i}=\{X_{ij}\}_{j\in J_i}$ be the disjoint family of equivalence classes corresponding to $H_i$.  Let $Y$ be an element of $\mathcal{X}_H$, and for $j\in J_0$ define 
$$
Y_{0j}:=Y\cap U_0\cap X_{0j}
$$
and for $j\in J_1$, define
$$
Y_{1j}:=(Y\cap X_{1j})\setminus U_0.
$$
Then we have 
$$
Y=\underbrace{\Big(\bigsqcup_{j\in J_0}Y_{0j}\Big)}_{=:Y_0}\sqcup \underbrace{\Big(\bigsqcup_{j\in J_1}Y_{1j}\Big)}_{=:Y_1}.
$$
As $H_i$ is generated by the set in line \eqref{hj gen set} with $\overline{E_{2r}}=K$, we have that 
$$
Y_i=\bigsqcup_{j\in J_i,~2r\text{-disjoint}}Y_{ij},
$$
and that each $N_{r,\mathcal{Y}}(Y_{ij})$ is contained in $X_{ij}$, and thus that each $N_{r,\mathcal{Y}}(Y_{ij})$ is contained in an element of the disjoint family $\mathcal{X}_{H_i}$.  Setting $\mathcal{Z}_i=\mathcal{X}_{H_i}$, we now have that $\mathcal{X}_H$ decomposes over $\mathcal{D}_m$, and we are done. 
\end{proof}

\subsection*{Amenability}

Our second goal in this appendix is to discuss the relationship of finite dynamical complexity to amenability.  See \cite{Anantharaman-Delaroche:2000mw} for a comprehensive discussion of amenable groupoids, and \cite[Section 5.6]{Brown:2008qy} for a self-contained discussion of the \'{e}tale case (which is all we will need).  In particular, the next definition is a slight variant of \cite[Definition 5.6.13]{Brown:2008qy} and \cite[Proposition 2.2.13]{Anantharaman-Delaroche:2000mw}.  The only difference between our definition and that of \cite[Definition 5.6.13]{Brown:2008qy} is that assumption (i) is not present in the latter reference.  It follows, however, from the argument that `condition (a) is irrelevant' in the proof of  \cite[Proposition 2.2.13]{Anantharaman-Delaroche:2000mw} that this leads to an equivalent definition.

\begin{definition}\label{gpd amen}
A locally compact, Hausdorff, \'{e}tale groupoid $G$ is \emph{amenable} if for all compact $K\subseteq G$ and all $\epsilon>0$ there exists a continuous, compactly supported function $\mu:G\to [0,1]$ such that:
\begin{enumerate}[(i)]
\item for all $x\in G^{(0)}$, we have $\sum_{g\in G_x}\mu(g)\leq 1$;
\item for all $k\in K$, we have $|1-\sum_{g\in G_{r(k)}}\mu(g)|<\epsilon$;
\item for all $k\in K$, we have $\sum_{g\in G_{r(k)}}|\mu(g)-\mu(gk)|<\epsilon.$
\end{enumerate}
\end{definition}

Our next goal is to prove the following theorem.

\begin{theorem}\label{fdc vs amen}
Let $G$ be a locally compact, Hausdorff, \'{e}tale groupoid with finite dynamical complexity.  Then $G$ is amenable.
\end{theorem}

This result is inspired by \cite[Theorem 4.6]{Guentner:2013aa}, part of which states that finite decomposition complexity for a bounded geometry metric space implies property A \cite[Definition 2.1]{Yu:200ve}.  As finite decomposition complexity for a bounded geometry metric space is equivalent to finite dynamical complexity for the corresponding coarse groupoid (Theorem \ref{gen fdc vs fdc} above), and as property A for a bounded geometry metric space is equivalent to amenability of the corresponding coarse groupoid (\cite[Theorem 5.3]{Skandalis:2002ng}), Theorem \ref{fdc vs amen} above is a generalization of the result of \cite[Theorem 4.6]{Guentner:2013aa}.

We now turn to the proof of Theorem \ref{fdc vs amen}.  It suffices, as usual, to show that the class $\mathcal{A}$ of amenable open subgroupoids of any \'{e}tale groupoid $G$ contains the relatively compact open subgroupoids, and is closed under decomposability.   The following lemma (which is presumably well-known to experts) starts this off.

\begin{lemma}\label{rel com amen}
Let $G$ be an \'{e}tale groupoid, and let $H$ be an open subgroupoid of $G$ with compact closure.  Then $H$ is amenable.
\end{lemma}

\begin{remark}\label{sig com rem}
In the case $H$ as in Lemma \ref{rel com amen} is $\sigma$-compact, there is a direct proof of the lemma using \cite[Theorem 2.14]{Renault:2013lp}.  Indeed, in this theorem Renault shows that for a $\sigma$-compact, locally compact, Hausdorff, \'{e}tale groupoid $G$ (and indeed more generally), amenability is equivalent to the existence of a sequence $(\mu_n:G\to [0,\infty))$ of Borel functions such that 
\begin{enumerate}[(i)]
\item for all $x\in G^{(0)}$, we have $\sum_{g\in G_x}\mu_n(g)\leq 1$;
\item for all $x\in G^{(0)}$, we have $\sum_{g\in G_x}\mu_n(g)\to 1$;
\item for all $k\in G$, we have $\sum_{g\in G_{r(k)}}|\mu_n(g)-\mu_n(gk)|\to 0.$
\end{enumerate}
Now, let $H$ be as in Lemma \ref{rel com amen} and also be $\sigma$-compact.  Define 
$$
\mu:H^{(0)}\to [0,1], \quad x\mapsto |H_x|^{-1}.
$$
It is not difficult to check that the (constant) sequence $(\mu_n=\mu)$ satisfies the properties above exactly, so we are done.  

Below we give a general proof of Lemma \ref{rel com amen} as some examples that are important to us (specifically, the coarse groupoids of Definition \ref{coarse gpd}) have open, relatively compact subgroupoids that are not $\sigma$-compact.
\end{remark}

\begin{proof}[Proof of Lemma \ref{sig com rem}]
Let $N=\sup\{|r^{-1}(x)\cap H| \mid x\in H^{(0)}\}$.  Compactness of $\overline{H}$ implies that this is finite.  We will proceed by induction on $N$.  In the base case $N=1$, $H$ is just a space and is thus clearly amenable.  Assume now that we have proven all cases up to $N-1$, and assume that $H$ has some range fibers with cardinality $N$, but none higher.  Let $U=\{x\in H^{(0)}\mid |r^{-1}(x)|=N\}$, which is open as $H$ is \'{e}tale, and clearly it is invariant for the $H$ acton.  Let $F=H^{(0)}\setminus U$, which is closed, and let $H_U$ and $H_F$ be the respective restrictions of $H$ to $U$ and $F$.   Note that $H_F$ is amenable by inductive hypothesis.  We first claim that $H_U$ is amenable.

Indeed, to see this, let $K\subseteq H_U$ be a compact set, and $\epsilon>0$.   Let $\phi:H_U^{(0)}\to [0,1]$ be any compactly supported function that is equal to $1$ on $r(K)\cup s(K)$, and define $\mu(h)=\frac{1}{N}\phi(s(h))$ for all $h\in H_U$; we claim that this has the right properties.  We first claim that $\mu$ is compactly supported.  To see this, note that $\mu$ is supported in $s^{-1}(\text{supp}(\phi))$, whence it suffices to show that $s^{-1}(E)$ is compact for any compact subset $E$ of $H^{(0)}_U$.  For each $x\in H^{(0)}$, choose an open set $V_x$ with compact closure, and such that $s^{-1}(V_x)$ can be written as a disjoint union 
$$
s^{-1}(V_x)=\bigsqcup_{j=1}^N V_x^{(j)},
$$
where $s$ restricts to a homeomorphism $s:V_x^{(j)}\to V_x$, and each $V_x^{(j)}$ has compact closure; using local compactness, the fact that $s$ is a local homeomorphism and the fact that each source fiber contains exactly $N$ elements, it is not too difficult to see that such sets exist.  Now, let $x_1,...,x_n$ be a finite collection of points of $H_U^{(0)}$ such that $E\subseteq \bigcup_{i=1}^nV_{x_i}$.  Then 
$$
s^{-1}(E)\subseteq \bigcup_{i=1}^n\bigcup_{j=1}^N V_{x_i}^{(j)}
$$
the set on the right is a finite union of sets with compact closure, so has compact closure, and the set $s^{-1}(E)$ is closed.  It is thus compact, and we have completed the proof that $\mu$ is compactly supported.  

To complete the proof that $\mu$ has the properties needed to show amenability, for each $x\in H_U^{(0)}$, note now that 
$$
\sum_{h\in (H_U)_x}\mu(h)=\frac{1}{N}\sum_{h\in H_x}\phi(x)=\phi(x),
$$
which is at most one for a general $x$, and exactly one if $x=r(k)$ for some $k\in K$.  On the other hand, we have that for each $k\in K$,
$$
\sum_{h\in (H_U)_{r(k)}}|\mu(h)-\mu(hk)|=\frac{1}{N}\sum_{h\in (H_U)_{r(k)}}|\phi(r(k))-\phi(s(k))|,
$$
which is exactly zero as $\phi$ is identically one on $r(K)\cup s(K)$.  This completes the proof that $H_U$ is amenable.

Now, consider the commutative diagram of $C^*$-algebras
$$
\xymatrix{ 0\ar[r] & C^*_{\max}(H_U)\ar[r] \ar[d] & C^*_{\max}(H) \ar[r] \ar[d] & C^*_{\max}(H_F) \ar[r] \ar[d] & 0  \\
 0\ar[r] & C^*_{r}(H_U)\ar[r] & C^*_{r}(H) \ar[r] & C^*_{r}(H_F) \ar[r] & 0~.}
$$
The top row is exact as this always holds for the maximal groupoid $C^*$-algebra (see \cite[Lemma 6.3.2]{Anantharaman-Delaroche:2000mw} -- the second countability assumption used there is unnecessary in the \'{e}tale case).  The bottom row might not be exact, although all that can go wrong is that the kernel of the map out of $C^*_r(H)$ might not equal the image of the map going in.  However, as $H_U$ and $H_F$ are amenable, the left and right hand vertical arrows are the identity map (\cite[Corollary 5.6.17]{Brown:2008qy}); it follows from this and a diagram chase that the bottom row is in fact exact in this case.  Now, as $H_U$ and $H_F$ are amenable, their reduced $C^*$-algebras are nuclear \cite[Theorem 5.6.18]{Brown:2008qy}.  Finally, an extension of nuclear $C^*$-algebras is nuclear \cite[Proposition 10.1.3]{Brown:2008qy}, so this implies that $C^*_r(H)$ is nuclear and thus that $H$ is amenable by \cite[Theorem 5.6.18]{Brown:2008qy} again.
\end{proof}

We will need the following lemma about the existence of almost invariant partitions of unity, which can be proved in the same way as \cite[Proposition 7.1]{Guentner:2014aa}.

\begin{lemma}\label{decom pou}
Let $G$ be an \'{e}tale groupoid, $\mathcal{C}$ a collection of open subgroupoids of $G$, and let $H$ be an open subgroupoid of $G$ that decomposes over $\mathcal{C}$.  Then for any open, relatively compact subset $K$ of $H$ and any $\epsilon>0$ there exists an open cover $\{U_0,U_1\}$ of $H^{(0)}$ and continuous compactly supported functions $\phi_i:H^{(0)}\to [0,1]$ with the following properties.
\begin{enumerate}[(i)]
\item For each $i\in \{0,1\}$, the set 
$$
\{k\in K\mid s(k)\in U_i\}
$$
generates an open subgroupoid $H_i$ of $H$ (whence also of $G$) in the class $\mathcal{C}$.
\item Each $\phi_i$ is supported in $U_i$.
\item For all $x\in H^{(0)}$, $\phi_0(x)+\phi_1(x)\leq 1$, and for all $k\in K$, $\phi_0(r(k))+\phi_1(r(k))=1$.
\item For any $k\in K$, and $i\in \{0,1\}$, $|\phi_i(s(k))-\phi_i(r(k))|<\epsilon$. \qed
\end{enumerate}
\end{lemma}

We are now ready to complete the proof of the Theorem \ref{fdc vs amen} by showing that the collection $\mathcal{A}$ of amenable open subgroupoids of $G$ is closed under decomposability.

\begin{proof}[Proof of Theorem \ref{fdc vs amen}]
Let $H$ be an open subgroupoid of $G$ that decomposes over $\mathcal{A}$, and let $K\subseteq H$ be compact, and $\epsilon>0$.  Using local compactness, expanding $K$ slightly we may assume that $K$ is in fact open and relatively compact.  Let $U_0,U_1$, $H_0,H_1$, and $\phi_0,\phi_1$ be as in Lemma \ref{decom pou} for the relatively compact set $K$ and error estimate $\epsilon/3$.  An elementary argument shows that $K$ can be written as $K_0\cup K_1$, where each $K_i$ is open, and has compact closure inside $H_i$.  For each $i$, let $\mu_i:H_i\to [0,1]$ be a function as in the definition of amenability, with respect to compact subset the closure $\overline{K_i}$ of $K_i$, and error estimate $\epsilon/3$.  Extending by zero outside (the open set) $H_i$, we may assume that $\mu_i$ is defined on all of $H$.  Define 
$$
\mu:H\to [0,1],\quad h\mapsto \phi_0(s(h))\mu_0(h)+\phi_1(s(h))\mu_1(h),
$$
which we claim has the right properties.

Indeed, note first for any $x\in H^{(0)}$,
\begin{align*}
\sum_{h\in H_{x}} \mu(h) & =\sum_{h\in H_{x}}\phi_0(s(h))\mu_0(h)+\phi_1(s(h))\mu_1(g) \\ 
& = \phi_0(x)\sum_{h\in (H_0)_x}\mu_0(h)+\phi_1(x)\sum_{h\in (H_1)_x}\mu_1(h) \\
& \leq \phi_0(x)+\phi_1(x)\leq 1.
\end{align*}
On the other hand, for any $k\in K$,
\begin{align*}
& \Big|1-\sum_{h\in H_{r(k)}} \mu(h)\Big|  =\Big|1-\sum_{h\in H_{r(k)}}\phi_0(s(h))\mu_0(h)+\phi_1(s(h))\mu_1(h)\Big| \\
& = \phi_0(r(k))\Big|1-\sum_{h\in (H_0)_{r(k)}}\mu_0(h)\Big|+\phi_1(r(k))\Big|1-\sum_{h\in (H_1)_{r(k)}}\mu_1(h)\Big| \\ 
& <\phi_0(r(k))\epsilon+\phi_1(r(k))\epsilon=\epsilon.
\end{align*}
Finally, we note that for any $k\in K$,
\begin{align*}
& \sum_{h\in H_{r(k)}}|\mu(h)-\mu(hk)|  \\& \leq \sum_{i=0}^1\sum_{h\in H_{r(k)}}\phi_i(s(h))|\mu_i(h)-\mu_i(hk)|+\mu_i(hk)|\phi_i(s(h))-\phi_i(s(hk))| \\
& =\sum_{i=0}^1\sum_{h\in (H_i)_{r(k)}}\phi_i(r(k))|\mu_i(h)-\mu_i(hk)|+\mu_i(hk)|\phi_i(r(k))-\phi_i(s(k))| \\
&=\sum_{i=0}^1\phi_i(r(k))\sum_{h\in (H_i)_{r(k)}}|\mu_i(h)-\mu_i(hk)| \\ & \quad \quad \quad \quad \quad \quad+\sum_{i=0}^1|\phi_i(r(k))-\phi_i(s(k))|\sum_{h\in (H_i)_{r(k)}}\mu_i(hk) \\
& < \Big(\sum_{i=0}^1 \phi_i(r(k))\Big)\frac{\epsilon}{3}+\frac{2\epsilon}{3}\cdot 1= \epsilon.
\end{align*}
This completes the proof.
\end{proof}

\subsection*{Open questions}

To state the following lemma, we recall that if $G$ is an \'{e}tale groupoid and $x\in G^{(0)}$, then the \emph{isotropy group} of $G$ at $x$ is 
$$
\{g\in G\mid r(g)=s(g)=x\}.
$$
We then have the following, which provides an easy obstruction to finite dynamical complexity.

\begin{lemma}\label{isotropy lemma}
Let $G$ be an \'{e}tale groupoid with finite dynamical complexity.  Then all isotropy groups of $G$ are locally finite.
\end{lemma}

\begin{proof}
Let $\mathcal{LF}$ be the collection of all open subgroupoids of $G$ whose isotropy groups are locally finite.  It suffices to show that $\mathcal{LF}$ contains the relatively compact open subgroupoids, and is closed under decomposability. This is reasonably straightforward, and we leave the details to the reader.
\end{proof}

At this point, we know two obstructions to a groupoid having finite dynamical complexity: having infinite isotropy, and being non-amenable.  The following question seems particularly interesting.  It is closely related to \cite[Question 5.1.3]{Guentner:2013aa}, and is a more general version of the well-known question as to whether finite decomposition complexity and Yu's property A are equivalent.

\begin{question}\label{fdc vs amen q}
Say $G$ is a principal\footnote{This means that all isotropy groups are trivial; one could also ask what happens when the isotropy groups are just assumed locally finite.}, amenable, \'{e}tale groupoid.  Must $G$ have finite dynamical complexity?
\end{question}

The connection to finite decomposition complexity of groups is not completely clear.  The following questions are thus natural.

\begin{question}\label{fdc vs fdc q}
If $\Gamma\ltimes X$ is a transformation groupoid with finite dynamical complexity (and $X$ compact), must $\Gamma$ have finite decomposition complexity?
\end{question}

By analogy with the case of finite dynamic asymptotic dimension \cite[Section 6]{Guentner:2014aa}, we suspect the answer to the above question is `yes', but did not seriously pursue this.

\begin{question}\label{fdc vs fdc q2}
If $\Gamma$ has finite decomposition complexity, must it admit an action on the Cantor set with finite dynamical complexity?  
\end{question}

Much more ambitiously, we do not see any obvious obstructions to a positive answer to the following question.  While we would be surprised if it has a positive answer in general, it is also interesting to ask about special classes of groups $\Gamma$ such as nilpotent groups (compare \cite{Szabo:2014aa}), free groups, general word hyperbolic groups, or even linear groups (compare \cite[Section 3]{Guentner:2009tg}).

\begin{question}\label{fdc vs fdc q3}
If $\Gamma$ has finite decomposition complexity, must any free amenable action of $\Gamma$ have finite dynamical complexity?  
\end{question}

Another interesting question, related to our earlier work \cite{Guentner:2014aa} is as follows.

\begin{question}\label{fdc vs dad q}
Say $G$ is an \'{e}tale groupoid with finite dynamic asymptotic dimension.  Must $G$ have finite dynamical complexity?
\end{question}

We suspect the answer is 	`yes', but it is currently not clear.  Note that the answer is clearly yes if the dynamic asymptotic dimension of $G$ is zero or one.  

\begin{question}\label{fdc c* q}
Say $G$ is an \'{e}tale groupoid with finite dynamical complexity.  What structural properties must the reduced $C^*$-algebra $C^*_r(G)$ have?
\end{question}

Certainly $C^*_r(G)$ must be nuclear by Theorem \ref{fdc vs amen} and \cite[Theorem 5.6.18]{Brown:2008qy}.  However, we do not know much beyond this.  For example, if $C^*_r(G)$ is also assumed simple, one might ask about properties of interest in the classification program such as comparison and $\mathcal{Z}$-stability (although to avoid examples like those in \cite{Giol:2010aa} and thus have some hope of positive results, one should assume that $G^{(0)}$ is `reasonable', say for example finite-dimensional, or just a Cantor set).

\section{Comparison to the Baum-Connes assembly map}\label{bc sec}

The purpose of this appendix is to identify the standard picture of the Baum-Connes assembly map for $\Gamma$ with coefficients in $C(X)$ as discussed in say \cite{Baum:1994pr}, with our picture defined using localization algebras and the evaluation-at-zero map.   As it is no more complicated and may be useful for other work, we do this in more generality than necessary for this paper in that we allow $C(X)$ to be replaced by an arbitrary (separable) $\Gamma$-$C^*$-algebra $A$.

Of necessity, we assume more of the reader than in the rest of the paper: specifically, some working knowledge of Hilbert modules (see \cite{Lance:1995ys} for background and conventions) and of equivariant $KK$-theory (see \cite[Section 2]{Kasparov:1988dw} for background and conventions).  On the other hand, it is certainly not necessary to read this appendix to understand the rest of the paper.

\begin{definition}\label{gen loc}
Let $\Gamma$ be a countable discrete group, and $A$ a (separable) $\Gamma$-$C^*$-algebra.  Let $Y$ be a locally compact metric space, equipped with a proper, co-compact, and isometric $\Gamma$-action, and fix a compact subset $K\subseteq Y$ such that $\Gamma\cdot K=Y$.  Let $H_Y$ be a non-degenerate, covariant representation of $C_0(Y)$ with the property that no non-zero element of $Y$ acts as a compact operator.  Let $H$ be a separable infinite-dimensional Hilbert space equipped with the trivial $\Gamma$ action.

Define the Hilbert $A$-module $\e$ to be the tensor product
$$
\e:= H_Y\otimes A \otimes H\otimes \ell^2(\Gamma)
$$
(here the tensor products are completed external tensor product of Hilbert modules: see \cite[Chapter 4]{Lance:1995ys}).  We write elementary tensors in $\e$ as 
$$
\xi\otimes a\otimes \eta\otimes \zeta,\quad \xi\in H_Y,~ a\in A,~ \eta\in H,~\zeta\in \ell^2(\Gamma).
$$
The actions of $\Gamma$ on $C_0(Y)$ and $A$ are denoted $\gamma$ and $\alpha$ respectively, and the unitaries implementing the action of $g\in \Gamma$ on $\ell^2(\Gamma)$ and $H_Y$ are denoted by $\lambda_g$ and $u_g$ respectively.  We define an action $\epsilon$ of $\Gamma$ on $\e$ by
$$
\epsilon_g(\xi\otimes a\otimes \eta\otimes \zeta):=u_g\xi\otimes \alpha_g(a)\otimes \eta\otimes \lambda_g\zeta.
$$
We write $\widetilde{\epsilon}$ for the $\Gamma$-action on the $C^*$-algebra $\mathcal{L}(\e)$ of adjointable operators on $\e$ defined for $e\in E$ by
$$
(\widetilde{\epsilon}_g(T))(e):=\epsilon_g(T(\epsilon_{g^{-1}}(e)));
$$
note that even though the linear isometries $\epsilon_g:\e\to \e$ are not adjointable operators, we nonetheless have that if $T$ is adjointable, then $\widetilde{\epsilon}_g(T)$ is too.

The $A$-valued inner product on $\e$ is given on elementary tensors by
$$
\langle \xi_1\otimes a_1\otimes \eta_1\otimes \zeta_1~,~ \xi_2\otimes a_2\otimes \eta_2\otimes \zeta_2\rangle:=\langle \xi_1,\xi_2\rangle\langle \eta_1,\eta_2\rangle \langle \zeta_1,\zeta_2\rangle a_1^*a_2,
$$
the right action of $a_1\in A$ by
$$
(\xi\otimes a\otimes \eta\otimes \zeta)\cdot a_1:=\xi\otimes aa_1\otimes \eta\otimes \zeta,
$$
and the left action of $f\in C_0(Y)$ by
\begin{equation}\label{y rep}
f\cdot (\xi\otimes a\otimes \eta\otimes \zeta)=f\xi\otimes a\otimes \eta\otimes \zeta;
\end{equation}
if we need notation for this representation, we will denote it by $\pi:C_0(Y)\to \mathcal{L}(\e)$, but we will generally omit the $\pi$ when no confusion is likely to arise. 

Denote by $\mathcal{K}(\e)$ the compact operators on $\e$ in the sense of Hilbert module theory, so in this case $\mathcal{K}(\e)$ is naturally isomorphic to $\mathcal{K}(H_Y\otimes H\otimes \ell^2(\Gamma))\otimes A$ (see \cite[pages 9-10]{Lance:1995ys}).  We will need the following properties of an adjointable operator $T$ on $\e$.

\begin{enumerate}[(i)]
\item $T$ is \emph{locally compact} if for any $f\in C_0(Y)$, $fT$ and $Tf$ are in the $C^*$-algebra $\mathcal{K}(\e)$.
\item The \emph{support} of $T$, denoted $\supp(T)$, is the complement of the set 
$$
\left\{ \begin{array}{l|l} (y,z)\in Y\times Y &  \text{ there are } f_1,f_2\in C_0(Y) \text{ with } f_1(y)\neq 0,f_2(y)\neq 0 \\ &  \text{ and } f_1Tf_2=0\end{array}\right\}.
$$
The \emph{metric propagation} of $T$ is the extended real number 
$$
\sup\{d(y,z)\mid (y,z)\in \supp(T)\}.
$$
\item The \emph{$\Gamma$-propagation} of $T$ is the extended real number
$$
\sup\{|g|\mid \text{supp}(T)\cap K\times gK \neq\varnothing\}
$$
(where we recall that $K\subseteq Y$ is a fixed compact set satisfying $\Gamma\cdot K=Y$).
\item $T$ is \emph{$\Gamma$-invariant} if $\widetilde{\epsilon}_g(T)=T$ for all $g\in \Gamma$.
\end{enumerate}
The \emph{Roe algebra}, denoted $C^*(Y;A)$, associated to $\e$ is the $C^*$-algebra closure of the $*$-algebra of all finite $\Gamma$-propagation, locally compact, $\Gamma$-invariant adjointable operators on $\e$ for the norm inherited from $\mathcal{L}(\e)$.  The \emph{localization algebra}, denoted $C_L^*(Y;A)$, associated to $\e$ is the $C^*$-algebra completion of the $*$-algebra all bounded, uniformly continuous functions 
$$
a:[0,\infty)\to C^*(Y;A)
$$
such that the $\Gamma$-propagation of $a(t)$ is bounded independently of $t$, such that the metric propagation tends to zero as $t$ tends to infinity, and where the norm is given by $\sup_t \|a(t)\|_{C^*(Y;A)}$.
\end{definition}

\begin{remark}\label{iso rem}
\begin{enumerate}[(i)]
\item The exact numerical value of the $\Gamma$-propagation as defined above depends on the choice of compact set $K\subseteq Y$ with $\Gamma\cdot K=Y$. However, whether or not the $\Gamma$-propagation of a family of operators is bounded does not depend on the choice of $K$, whence the Roe algebras and localization algebras do not depend on this choice.
\item Say $Y=P_s(\Gamma)$ is a Rips complex of $\Gamma$ and set $H_Y=\ell^2(Z_s)$ as in Definition \ref{rips}.  Say $A=C(X)$.  Considering $\ell^2(X)$ as a left $A$-module, we may form the internal Hilbert module tensor product of $\e$ and $\ell^2(X)$ over $A$ (see \cite[Chapter 4]{Lance:1995ys}), and thus get an isomorphism of Hilbert spaces
$$
\e\otimes_A \ell^2(X)\cong H_Y\otimes \ell^2(X)\otimes H\otimes \ell^2(\Gamma).
$$
The map
$$
\mathcal{L}(\e)\to \mathcal{B}(\ell^2(\Gamma)\otimes H_Y\otimes H\otimes \ell^2(X)),\quad T\mapsto T\otimes_A 1_{\ell^2(X)}
$$
from the adjointable operators on $\e$ to the bounded operators on $\ell^2(\Gamma)\otimes H_Y\otimes H\otimes \ell^2(X)$ is then an isometric $*$-homomorphism (see the discussion on \cite[page 42]{Lance:1995ys}), and it is not difficult to check that it takes the Roe algebra $C^*(Y;A)$ onto the Roe algebra $C^*(\Gamma\lefttorightarrow X;s)$ as in Definition \ref{roe}.  Thus the two notions agree in this special case.
\item Note that if $Y=P_s(\Gamma)$, then we may use $H_Y=\ell^2(Z_s)$ (to avoid silly degeneracies, we should assume here that $s$ is large enough that $P_s(\Gamma)$ is not zero-dimensional).  It is clear then that if $s\leq t$, there are isometric inclusions $C^*(P_s(\Gamma);A)\to C^*(P_t(\Gamma);A)$, and similarly for the localization algebras.
\end{enumerate}
\end{remark}

Now, there is an evaluation-at-zero map
$$
\epsilon_0:K_*(C_L^*(Y;A))\to K_*(C^*(Y;A))
$$
induced by the obvious underlying $*$-homomorphism.  Our goal here is to relate this to the Baum-Connes assembly map for $\Gamma$ with coefficients in $A$ as in \cite[Section 9]{Baum:1994pr}.

In order to make this precise, let us fix some terminology here.  A \emph{cut-off} function for $Y$ is some non-negative valued $c\in C_c(Y)$ such that 
$$
\sum_{g\in \Gamma} c(gy)=1
$$
for all $y\in Y$; using properness and cocompactness, it is not difficult to see that such a $c$ exists and we fix one from now on.  If as usual $\gamma$ denotes the action of $\Gamma$ on $C_0(Y)$, then the \emph{basic projection} associated to $c$ is the element
$$
p_Y\in C_c(G,C_0(Y))\subseteq C_0(Y)\rtimes_r\Gamma
$$
defined by 
\begin{equation}\label{basic}
p_Y(g) :=\gamma_g(c)c. 
\end{equation}
The associated class
$$
[p_Y]\in K_0(C_0(Y)\rtimes_r\Gamma)=KK_0(\C,C_0(Y)\rtimes_r\Gamma)
$$
does not depend on the choice of $c$.  The \emph{assembly map} 
$$
\mu:KK_*^\Gamma(C_0(Y),A)\to K_*(A\rtimes_r\Gamma)
$$
is defined as the composition 
$$
\xymatrix{ KK_*^\Gamma(C_0(Y),A) \ar[r]^-{j^\Gamma_r} & KK_*(C_0(Y)\rtimes_r\Gamma,A\rtimes_r\Gamma) \ar[r]^-{[p_Y]\otimes \cdot} & KK(\C,A\rtimes_r\Gamma) }
$$
where the first map is Kasparov's descent morphism of \cite[Theorem 3.11]{Kasparov:1988dw}, and the second is Kasparov product with $[p_Y]\in KK(\C,C_0(Y)\rtimes_r\Gamma)$.  

The assembly maps as defined above are functorial under proper, equivariant, continuous maps of the space $Y$ appearing on the left hand side.  Let $\underline{E}\Gamma$ be a universal $\Gamma$-space for proper actions as in \cite[Section 1]{Baum:1994pr} and define 
$$
KK^\Gamma_*(\underline{E}\Gamma,A):=\lim_{Y\subseteq \underline{E}\Gamma}KK^\Gamma_*(C_0(Y),A)
$$
where the limit is over all $\Gamma$-invariant cocompact subspaces of $\underline{E}\Gamma$.  Finally, the \emph{Baum-Connes assembly map} is the map
$$
\mu:KK^\Gamma_*(\underline{E}\Gamma,A)\to K_*(A\rtimes_r\Gamma)
$$
defined as the direct limit of the individual assembly maps defined above.  

We will want to use the following concrete model for $\underline{E}\Gamma$ .  Let $X_\Gamma:=\bigcup_{s\geq 0}P_s(\Gamma)$ equipped with $\ell^1$-metric; as discussed in \cite[Section 2]{Baum:1994pr}, this is a model for the classifying space $\underline{E}\Gamma$.  Moreover, the individual Rips complexes form a `homotopy-cofinal' system inside the collection of $\Gamma$-cocompact equivariant subsets (ordered by inclusion) of $X_\Gamma$: precisely, we mean that for any cocompact $Y\subseteq X_\Gamma$, the inclusion map is (equivariantly, properly) homotopic to a map with image in some $P_s(\Gamma)$.  Hence the Baum-Connes assembly map is equivalent to the direct limit of the assembly maps for the individual Rips complexes, i.e.\ the Baum-Connes assembly map can be thought of as a map
$$
\mu:\lim_{s\to\infty} KK^\Gamma_*(C_0(P_s(\Gamma)),A)\to K_*(A\rtimes_r\Gamma).
$$
We are now ready to state the main result of this section.

\begin{theorem}\label{main com}
Let $\Gamma$ be a countable discrete group, and $A$ a $\Gamma$-$C^*$-algebra.  Let $Y$ be a locally compact metric space, equipped with a proper, cocompact, and isometric $\Gamma$-action.  Let 
$$
\mu:KK^\Gamma_*(C_0(Y),A)\to K_*(A\rtimes_r\Gamma)
$$
be the assembly map associated to this data.  Then there is a commutative diagram
$$
\xymatrix { KK^\Gamma_*(C_0(Y),A) \ar[d] \ar[r]^-\mu &  K_*(A\rtimes_r\Gamma) \ar[d] \\
K_*(C_L^*(Y;A)) \ar[r]^-{\epsilon_0} & K_*(C^*(Y;A))  }
$$
where the vertical maps are isomorphisms.

Moreover, let $s\leq t$ be non-negative real numbers, and say $P_s(\Gamma)$, $P_t(\Gamma)$ are the associated Rips complexes.  Then with notation as above, there is a commutative diagram
{\tiny
$$
\xymatrix { KK^\Gamma_*(C_0(P_s(\Gamma)),A) \ar[dd] \ar[dr] \ar[rr]^-\mu && K_*(A\rtimes_r\Gamma) \ar[dr]^-= \ar'[d][dd] &  \\
& KK^\Gamma_*(C_0(P_t(\Gamma)),A) \ar[dd] \ar[rr]^{\mu \quad\quad\quad} && K_*(A\rtimes_r\Gamma) \ar[dd] \\
K_*(C_L^*(P_s(\Gamma);A)) \ar'[r][rr]^-{\epsilon_0} \ar[dr] && K_*(C^*(P_s(\Gamma);A)) \ar[dr] & \\
& K_*(C_L^*(P_t(\Gamma);A)) \ar[rr]^-{\epsilon_0}  && K_*(C^*(P_t(\Gamma);A))~.   }
$$}
Here the diagonal maps are induced by the inclusion $P_s(\Gamma)\to P_t(\Gamma)$, together with Remark \ref{iso rem} for the Roe algebras and localization algebras.  Hence taking the direct limit as $s\to\infty$ identifies the Baum-Connes assembly map 
$$
\mu:\lim_{s\to\infty}KK^\Gamma_*(P_s(\Gamma),A)\to K_*(A\rtimes_r\Gamma)
$$
with the evaluation-at-zero map
$$
\epsilon_0:\lim_{s\to\infty}K_*(C_L^*(P_s(\Gamma);A))\to \lim_{s\to\infty} K_*(C^*(P_s(\Gamma);A))
$$
\end{theorem}

In order to explain the proof of Theorem \ref{main com}, we will need to define some auxiliary $C^*$-algebras.  The statement in the second part of Theorem \ref{main com} on compatibility with increasing the Rips parameter is straightforward from the proof of the first part, so we only give the proof of the first part.  Say then that $Y$, $H_Y$, and $A$ as above are all fixed.  

\begin{definition}\label{pas}
Let $\e$ be as above, and let $C^*$ and $C^*_L$ be shorthand for the associated Roe algebra and localization algebra.  An adjointable operator $T$ on $\e$ is \emph{pseudolocal} if for any $f\in C_0(Y)$, the commutator $[f,T]$ is in $\mathcal{K}(\e)$.  Let $D^*$ denote the $C^*$-algebra closure of the collection of all finite $\Gamma$-propagation, pseudolocal, $\Gamma$-invariant adjointable operators on $\e$ inside $\mathcal{L}(\e)$.  Let $D_L^*$ denote the $C^*$-algebra completion of all bounded, uniformly continuous functions 
$$
a:[0,\infty)\to D^*
$$
such that the $\Gamma$-propagation of $a(t)$ is uniformly bounded for all $t$, such that the metric propagation tends to zero as $t$ tends to infinity, and where the norm is given by $\sup_t \|a(t)\|_{D^*}$.
\end{definition}
\noindent Note that $C^*$ and $C_L^*$ are ideals in $D^*$ and $D_L^*$ respectively.  

We will prove the first part of Theorem \ref{main com} by showing that there is a commutative diagram
\begin{equation}\label{pas diag}
\xymatrix{ KK_i^\Gamma(C_0(Y),A) \ar[r]^-\mu \ar[d]_{(i)} & K_i(A\rtimes_r\Gamma) \ar[d]^-{(ii)} \\
K_{i+1}(D^*/C^*)  \ar[r]^-\partial &  K_{i}(C^*)  \\
K_{i+1}(D_L^*/C_L^*) \ar[u]^-{(iii)} \ar[r]^-{(iv)} & K_i(C_L^*) \ar[u]_{\epsilon_0} }
\end{equation}
such that the arrows labelled by roman numerals are all isomorphisms.

\begin{remark}\label{pas ass}
The arrow labelled `$\partial$' is the standard boundary map in the $K$-theory six-term exact sequence associated to the short exact sequence
$$
\xymatrix{ 0\ar[r] & C^* \ar[r] & D^* \ar[r] & D^*/C^* \ar[r] & 0 }.
$$
Note that we get for free from this proof that $\partial$ gives another model for the assembly map: this is a version with coefficients of the `Paschke duality' model for the (coarse) Baum-Connes assembly map that is discussed for example in \cite{Roe:2002nx} and \cite{Higson:1995fv}.
\end{remark}

We now explain the main steps of the proof, starting with the top square.  The arrow labelled (i) is a form of Paschke duality, and is shown to be an isomorphism by building on arguments in \cite[Chapter 8]{Higson:2000bs}; the key technical points needed in addition are Fell's trick, Kasparov's stabilization theorem, and Kasparov's Hilbert module version of Voiculescu's theorem \cite{Kasparov:1980sp}.  The arrow labelled (ii) is induced by a Morita equivalence, which is canonical given the choice of cut-off function $c$.  The argument that the top square commutes involves a significant amount of computation, and is based on \cite{Roe:2002nx}.  

For the bottom square, the arrow labelled (iii) is induced by the evaluation-at-zero map, and the arrow labelled (iv) is the boundary map from the $K$-theory six-term exact sequence associated to the short exact sequence
$$
\xymatrix{ 0\ar[r] & C_L^* \ar[r] & D_L^* \ar[r] & D_L^*/C_L^* \ar[r] & 0 }.
$$
In particular, commutativity of the bottom square is immediate from naturality of the six-term exact sequence.  Our proofs that (iii) and (iv) are isomorphisms are closely based on arguments from \cite{Qiao:2010fk} (which were in turn inspired by work of the third author \cite{Yu:1997kb}, although that paper uses quite a different argument); in both cases, the proofs boil down to clever uses of Eilenberg swindles.

In the next two subsections, we look at the top (`Paschke duality') and bottom (`Localization algebra') squares separately.

\subsection*{Paschke duality square}

Let us set up some conventions for equivariant $KK$-theory.  We will work entirely in the odd $KK$ and $K$ groups: the even case can be 
deduced from the odd case by replacing $A$ with $A\otimes C_0(\R)$, where $C_0(\R)$ has the trivial $\Gamma$-action (alternatively, the even case can be handled directly by arguments analogous to those used below for the odd case, but is notationally more complicated due to the necessity of dragging gradings through all the proofs).

Let $B$, $C$ be (trivially graded) $\Gamma$-$C^*$-algebras.  We will write cycles for $KK_1^\Gamma(B,C)$ as quadruples $(\mathcal{E},F,\beta,\phi)$ where $\mathcal{E}$ is a Hilbert $C$-module, $F$ is an adjointable operator on $\mathcal{L}(\mathcal{E})$, $\beta$ is a $\Gamma$-action on $\mathcal{E}$ by bounded linear isometries (not necessarily by adjointable operators, however), and $\phi:B\to\mathcal{L}(\mathcal{E})$ is an equivariant $*$-homomorphism.  Cycles for $KK_1(B,C)$ will analogously be written $(\mathcal{E},F,\phi)$.  There are then required to be various compatibilities between this data: see \cite[Section 2]{Kasparov:1988dw} for precise definitions.

We will need the following Hilbert module version of Voiculescu's theorem, due to Kasparov \cite[Theorem 5]{Kasparov:1980sp}; as the statement is a little technical, we repeat the special case we need for the reader's convenience.

\begin{theorem}\label{kv the}
Let $B$ be a unital, nuclear, separable $C^*$-algebra, and $C$ a $\sigma$-unital $C^*$-algebra.  Assume that $B$ is equipped with a unital $*$-representation $B\to \mathcal{B}(H)$ on some separable infinite dimensional Hilbert space $H$ whose image contains no compact operators.  Let moreover $H\otimes C$ be the standard Hilbert $C$-module, and note that there is a unital inclusion 
$$
\pi:B\to \mathcal{B}(H)\to \mathcal{L}(H\otimes C),
$$
where the map from the bounded operators on $H$ to the adjointable operators on $H\otimes C$ is defined by amplification (see \cite[page 35]{Lance:1995ys}).

Let $\phi:B\to\mathcal{L}(H\otimes C)$ be a unital $*$-homomorphism, and consider the sum $\phi\oplus \pi:B\to \mathcal{L}((H\otimes C)\oplus (H\otimes C))$.  Then there is an adjointable isometry $V:H\otimes C \to (H\otimes C)\oplus (H\otimes C)$ such that the difference
$$
V^*\pi(b)V -\phi(b)
$$ 
is in $\mathcal{K}(H\otimes C)$ for all $b\in B$. \qed
\end{theorem}
While Kasparov's theorem also applies in the presence of a compact group action, there is unfortunately no general version for non-compact groups.  To get around this issue in the case of proper actions that is relevant for us, we need a version of Fell's trick for Hilbert modules that we now discuss.

Let $\mathcal{E}$ be an equivariant Hilbert $A$-module, with $\Gamma$-action $\beta$, and recall that the $\Gamma$-action on $A$ is denoted by $\alpha$.  We will denote by $\ell^2(\Gamma)\otimes\mathcal{E}$ the usual external tensor product of Hilbert modules (see \cite[Chapter 4]{Lance:1995ys}), equipped with the $\Gamma$ action $\lambda\otimes \beta$ defined as the tensor product of the left regular representation and $\beta$, the right action of $A$ defined by $(\delta_g\otimes e)\cdot a:=\delta_g\otimes ea$, and the inner product defined by 
$$
\langle \delta_{g_1}\odot e_1~,~\delta_{g_2}\odot e_2\rangle:=\langle \delta_{g_1},\delta_{g_2}\rangle_{\ell^2(\Gamma)}\langle e_1,e_2\rangle_{\mathcal{E}}.
$$
With this structure, $\ell^2(\Gamma)\otimes \mathcal{E}$ is again an equivariant Hilbert $C$-module.

On the other hand, let $\ell^2_0(\Gamma)$ denote the elements of $\ell^2(\Gamma)$ with finite support, and let $\ell^2_0(\Gamma)\odot \mathcal{E}$ denote the algebraic tensor product (over $\C$).  Let $\Gamma$ act on $\ell^2_0(\Gamma)\odot \mathcal{E}$ by the tensor product $\lambda\odot 1$ of the left regular representation and the trivial representation.  Define a right action of $A$ on $\ell^2_0(\Gamma)\odot \mathcal{E}$ by the formula $(\delta_g\odot e)\cdot a:=\delta_g \odot e\alpha_{g^{-1}}(a)$, and define an $A$-valued inner product by the formula
$$
\langle \delta_{g_1}\odot e_1~,~\delta_{g_2}\odot e_2\rangle:=\langle \delta_{g_1},\delta_{g_2}\rangle_{\ell^2(\Gamma)}\langle \beta_{g_1^{-1}}(e_1),\beta_{g_2^{-1}}(e_2)\rangle_{\mathcal{E}}
$$
on elementary tensors, and extending.  One checks that this is an $A$-valued inner product, so completion gives a Hilbert $A$-module, which we denote by $\ell^2(\Gamma,\mathcal{E})$.  The action of $\Gamma$ moreover extends to an action on $\ell^2(\Gamma,\mathcal{E})$, which we still denote by $\lambda\otimes 1$, and the result is an equivariant Hilbert $A$-module.

For $\delta_g\otimes e\in \ell^2(\Gamma)\otimes \mathcal{E}$, define $U(\delta_g\otimes e):=\delta_g\otimes \alpha_{g}(a)$.  It is straightforward to check that $U$ extends to an equivariant unitary isomorphism
$$
U:\ell^2(\Gamma)\otimes \mathcal{E}\to\ell^2(\Gamma,\mathcal{E}),
$$
of Hilbert $A$-modules.  Using such a $U$ to switch `on / off' the second component of a $\Gamma$-action of this form is called \emph{Fell's trick}.

\begin{lemma}\label{kk rep}
Let $(\mathcal{E},F,\beta,\phi)$ be a cycle representing some class $x\in KK_1^\Gamma(C_0(Y),A)$.  Then there is a (non-canonical) way of associated a new cycle representing $x$ to $(\mathcal{E},F,\beta,\phi)$ that has the following additional properties.
\begin{enumerate}[(i)]
\item The new cycle has the form $(\e,F,\epsilon,\pi)$, where $F$ is a self-adjoint element of $D^*$ and $\pi:C_0(Y)\to \mathcal{L}(\e)$ is as defined in line \eqref{y rep} above.
\item The process takes: degenerate cycles to compact perturbations of degenerate cycles; unitary equivalences of cycles to compact perturbations of unitary equivalences of cycles; operator homotopies to operator homotopies; and direct sums of cycles to orthogonal sums of operators.
\end{enumerate}
\end{lemma}

\begin{proof}
We just give the proof of part (i) above; part (ii) follows from the proof we give and some straightforward direct checks.

Recall that $\gamma$ denotes the action of $\Gamma$ on $C_0(Y)$, and that $c$ is a fixed choice of cut-off function for the action of $\Gamma$ on $Y$.  Let $(\mathcal{E},F,\beta,\phi)$ be a cycle for $KK^\Gamma_1(C_0(Y),A)$.  Cutting down to $C_0(Y)\cdot \mathcal{E}$, we may assume that the action of $C_0(Y)$ on $\mathcal{E}$ is nondegenerate (compare \cite[Lemma 2.8]{Kasparov:1988dw}).  Let $\mathcal{E}_c$ denote the equivariant subspace $C_c(Y)\cdot \mathcal{E}$ of $\mathcal{E}$, which is dense by non-degeneracy.   Define $V:\mathcal{E}_c\to \ell^2(\Gamma)\otimes \mathcal{E}$ by the formula
$$
V:e\mapsto \sum_{g\in \Gamma} \delta_g\otimes \gamma_g(c)e;
$$
the sum makes sense as properness of the $\Gamma$ action combined with the compact support conditions on $c$ and $e$ imply that only finitely many terms are non-zero.  Computing gives
$$
\langle Ve,Ve\rangle=\sum_{g,h\in \Gamma}\langle \delta_g,\delta_h \rangle \langle \gamma_g(c)e,\gamma_h(c)e\rangle=\Big\langle \sum_{g\in \Gamma}\gamma_g(c^2)e,e\Big\rangle=\langle e,e\rangle, 
$$
from which it follows that $V$ extends to an isometric linear map $V:\mathcal{E}\to \ell^2(\Gamma)\otimes \mathcal{E}$.  It is straightforward to check moreover that $V$ is equivariant, and has an adjoint defined by
$$
V^*:\delta_g\otimes e\mapsto \gamma_g(c)e.
$$
In particular, there is an equivariant submodule $\mathcal{E}'$ of $\ell^2(\Gamma)\otimes \mathcal{E}$ such that $V(\mathcal{E})\oplus \mathcal{E}'\cong \ell^2(\Gamma)\otimes \mathcal{E}$.  Summing our cycle $(\mathcal{E},F,\beta,\phi)$ with the degenerate cycle $(\mathcal{E}',1,\beta\otimes\lambda|_{\mathcal{E}'},0)$ and applying a unitary isomorphism, we may replace our original cycle by one of the form $(\ell^2(\Gamma)\otimes \mathcal{E},F, \lambda\otimes \beta,\phi)$ (this $F$ and $\phi$ are not the same as the original ones, but what exactly they are does not matter at this point; we abuse notation as the price to pay for not multiplying primes or subscripts).

Conjugating by the unitary appearing in Fell's trick, we may replace our cycle by one of the form $(\ell^2(\Gamma,\mathcal{E}),F,\lambda\otimes 1,\phi)$.  Now, ignoring the $\Gamma$ actions, Kasparov's stabilization theorem (see for example \cite[Chapter 6]{Lance:1995ys}) embeds $\mathcal{E}$ as a complemented submodule of $H_Y\otimes H\otimes A$.  Hence we may embed $\ell^2(\Gamma,\mathcal{E})$ \emph{equivariantly}  as a complemented submodule of $\ell^2(\Gamma,H_Y\otimes H\otimes A)$.  Adding a degenerate cycle equipped with the zero action of $C_0(Y)$, we may thus assume that our class is represented by a cycle of the form 
$$
(\ell^2(\Gamma, H_Y\otimes H\otimes A),F,\lambda\otimes 1_{H_Y\otimes H\otimes A},\phi).
$$
Applying Fell's trick again, this time `in reverse' then shows that there is a cycle of the form
\begin{equation}\label{a cycle}
(\e,F,\epsilon,\phi)
\end{equation}
representing the same class.  Note from our construction so far that while the action of $C_0(Y)$ on $\e$ need not be non-degenerate, we do at least have that the submodule $\phi(C_0(Y))\cdot \e$ is complemented.  

Now, let $\widetilde{C_0(Y)}$ denote the unitization of $C_0(Y)$; abusing notation slightly, write 
$$
\pi:\widetilde{C_0(Y)}\to \mathcal{B}(H_Y)\to \mathcal{L}(\e),
$$
for the unital $*$-homomorphism extending our fixed $\pi$; here the first arrow is the unital extension of our fixed representation, and the second is amplification.   The $C^*$-algebra $\widetilde{C_0(Y)}$ is nuclear, and we assumed that no non-zero element acts as a compact operator on $H_Y$.  Hence (replacing $\phi$ with its unitization) Theorem \ref{kv the} gives an adjointable isometry
$$
V:\e\to \e
$$
with the property that 
$$
V^*\pi(f)V-\phi(f)\in \mathcal{K}(\e)
$$
for all $f\in C_0(Y)$.  

Unfortunately, $V$ does not need to respect the action of $\Gamma$, but we can rectify this as follows.  Choose a family of equivariant isometries 
$$
\big(v_g:\ell^2(\Gamma)\otimes H\to \ell^2(\Gamma)\otimes H\big)_{g\in \Gamma}
$$
with the properties that $\sum_g v_gv_g^*=1$ (convergence in the strong operator topology) and $v_h^*v_g=0$ for $h\neq g$ (such exist by the classical version of Fell's trick), and abusing notation, also write $v_g$ for the isometries on $\e$ induced by these.  Consider the submodule $\mathcal{E}_c:=\phi(C_c(Y))\cdot \e$ of $\e$, which is dense in the complemented submodule $\mathcal{E}:=\phi(C_0(Y))\cdot \e$ of $\e$.  Define a map
$$
W:\mathcal{E}_c\to \e,\quad e\mapsto  \sum_{g\in \Gamma}v_g\widetilde{\epsilon}_g(V)\gamma_g(c)e,
$$ 
which makes sense as the compact support conditions on $c$ and $\mathcal{E}_c$ guarantee that the sum on the right is finite.  Computing, for any $e_1,e_2\in \mathcal{E}_c$, 
\begin{align*}
\langle We_1,We_2\rangle & = \sum_{g,h\in \Gamma} \langle v_g\widetilde{\epsilon}_g(V)\gamma_g(c)e_1~,~ v_h\widetilde{\epsilon}_h(V)\gamma_h(c)e_2\rangle \\ & = \sum_{g,h\in \Gamma} \langle v_h^*v_g\widetilde{\epsilon}_g(V)\gamma_g(c)e_1~,~ \widetilde{\epsilon}_h(V)\gamma_h(c)e_2\rangle  \\
&= \sum_{g\in \Gamma}\langle \widetilde{\epsilon}_g(V)\gamma_g(c)e_1~,~ \widetilde{\epsilon}_h(V)\gamma_h(c)e_2\rangle \\
& = \sum_{g\in \Gamma} \langle \epsilon_g(V^*V)\gamma_g(c)e_1,\gamma_g(c)e_2\rangle \\
& = \sum_{g\in \Gamma} \langle \gamma_g(c^2)e_1,e_2\rangle=\langle e_1,e_2\rangle.
\end{align*}
Hence $W$ extends to an isometry $\mathcal{E}\to \e$, which is clearly equivariant.  Extending $W$ by zero on the complement of $\mathcal{E}$, we may consider $W$ as an equivariant partial isometry $W:\e\to\e$, and it is straightforward to check (using that $WW^*$ is the projection onto $\mathcal{E}:=\phi(C_0(Y))\cdot \e$) that 
$$
W^*\pi(f)W-\phi(f)\in \mathcal{K}(\e)
$$
for all $f\in C_0(Y)$.  This gives us that 
$$
(\e,W^*FW,\epsilon,\pi)
$$
(where $F$ is in the cycle in line \eqref{a cycle}) is a cycle for $KK_1^\Gamma(C_0(Y),A)$ that is equivalent to our original cycle.  

We now have a cycle $(\e,F,\epsilon,\pi)$ on the correct equivariant $C_0(Y)$-$A$ module $\e$.  Replacing the operator $F$ by
$$
\sum_{g} \gamma_g(c)\widetilde{\epsilon}_g(F)\gamma_g(c)
$$
(the sum converges strictly, as is straightforward to see using elements of $C_c(Y)\cdot \e$ as we have a couple of times already), we get a new cycle which is just a compact perturbation of the old one, and for which $F$ is equivariant and of finite $\Gamma$-propagation.  Together with the other conditions defining a Kasparov cycle, this gives that $F$ is an element of $D^*$.  Finally, replacing $F$ by $\frac{1}{2}(F+F^*)$, we may assume that $F$ is self-adjoint and are done.
\end{proof}

The above lemma now allows us to define the map labelled (i) in Diagram \eqref{pas diag}, and show it to be an isomorphism.

\begin{definition}\label{pas dual}
Define a homomorphism 
$$
\delta:KK_1^\Gamma(C_0(Y),A)\to K_1(D^*/C^*)
$$
by first representing a class $x$ in $KK^\Gamma_1(C_0(Y),A)$ as a cycle of the form in Lemma \ref{kk rep}; then note the conditions on a Kasparov cycle imply that the image of $\frac{1}{2}(1+F)$ in $D^*/C^*$ is a projection $p$, and define $\delta(x):=[p]$.
\end{definition}

The following proof is based on \cite[Theorem 8.4.3]{Higson:2000bs}: see the discussion there for more details.

\begin{corollary}\label{pas dual cor}
The map $\delta$ from Definition \ref{pas dual} above is a well-defined isomorphism.
\end{corollary}

\begin{proof}
We note first that the equivalence relation on $KK_1^\Gamma(C_0(Y),A)$ may be taken to be that generated by operator homotopies, addition of degenerate cycles, and unitary equivalences.  Indeed, as already noted there is a canonical process
$$
F \quad \leadsto \quad \sum_{g\in \Gamma}\gamma_{g}(c)\widetilde{\epsilon}_g(F)\gamma_g(c)
$$
for replacing operators by $\Gamma$-invariant ones.  Using this, it is not too difficult to see that the proof of the corresponding fact in the \emph{non}-equivariant case \cite{Skandalis:1984aa} extends to our setting.  The fact that $\delta$ is well-defined follows from this: operator homotopies give rise to homotopic projections, unitary equivalences to equivalent projections, and degenerate cycles to projections vulnerable to an Eilenberg swindle.  Moreover, $\delta$ is a homomorphism as one can add orthogonal projections.  

To see that $\delta$ is an isomorphism, note that it is surjective as every element of $K_*(D^*/C^*)$ can be represented by a projection in $D^*/C^*$ (as opposed to a matrix algebra over it), and lifting to $D^*$ gives rise to a cycle for $KK^\Gamma(C_0(Y),A)$.  It is injective as the equivalence relations on projections and unitaries defining $K_*(D^*/C^*)$ lift to equivalences of Kasparov cycles.
\end{proof}

We now recall some more details about the $KK$-theoretic assembly map
$$
\mu:KK^\Gamma_1(C_0(Y),A)\to K_1(A\rtimes_r\Gamma).
$$
We start with a class on the left hand side represented by some cycle $(\mathcal{E},F,\beta,\phi)$.  Kasparov \cite[3.7 -- 3.11]{Kasparov:1988dw} defines a \emph{descent homomorphism}
$$
j^\Gamma_r:KK^\Gamma_1(C_0(Y),A)\to KK_1(C_0(Y)\rtimes_r\Gamma,A\rtimes_r\Gamma)
$$
from equivariant $KK$ groups to the non-equivariant $KK$ groups of crossed products as follows.  Define a scalar product on $C_c(\Gamma,\mathcal{E})$ with values in $C_c(\Gamma,A)$ by for each pair for $e_1,e_2\in C_c(\Gamma,\mathcal{E})$ defining the function 
$$
\langle e_1,e_2\rangle:\Gamma\to A 
$$
by the formula
$$
\langle e_1,e_2\rangle(g):=\sum_{h\in \Gamma}\alpha_{h^-1}\big(\langle e_1(h),e_2(hg)\rangle_{\mathcal{E}}\big).
$$
Define a right action of $C_c(\Gamma,A)$ on $C_c(\Gamma,\mathcal{E})$ by the formula
$$
(e\cdot a)(g):=\sum_{h\in G} e(h)\alpha_h(a(h^{-1}g)).
$$
Kasparov shows that the inner product is positive, and thus it makes sense to define $\mathcal{E}\rtimes\Gamma$ as the Hilbert $A\rtimes_r\Gamma$-module defined by simultaneous completion (see \cite[pages 4-5]{Lance:1995ys}).  The module $\mathcal{E}\rtimes\Gamma$ is equipped with a left action of $C_0(Y)\rtimes_r\Gamma$ defined as the integrated form of the covariant representation defined by setting
$$
(\widetilde{\phi}(f)\cdot e)(g):=\phi(f)\cdot e(g),\quad f\in C_0(Y),~e\in C_c(\Gamma,\mathcal{E}),~g\in \Gamma
$$
and the unitary representation of $\Gamma$ defined by 
$$
(u_g\cdot e)(h):=\epsilon_g(e(g^{-1}h)), \quad g,h\in \Gamma,~e\in C_c(\Gamma,\mathcal{E}).
$$
Kasparov shows that this integrates to a representation of $C_0(Y)\rtimes_r\Gamma$, also denoted $\widetilde{\phi}$, so $\mathcal{E}\rtimes \Gamma$ is a $C_0(Y)\rtimes_r\Gamma$-$A\rtimes_r\Gamma$ bimodule.  An operator $\widetilde{F}$ is defined on $\mathcal{E}$ by the formula
$$
(\widetilde{F}\cdot e)(g):=F\cdot e(g).
$$
The map $j^\Gamma_r$ is then defined by $j[\mathcal{E},F,\beta,\phi]=[\mathcal{E}\rtimes \Gamma,\widetilde{F},\widetilde{\phi}]$.

The second step in defining the assembly map is to choose a cut-off function $c$ for $Y$, and use it to construct a basic projection as in line \eqref{basic} above and thus a class $[p_Y]$ in $K_0(C_0(Y)\rtimes_r\Gamma)\cong KK_0(\C,C_0(Y)\rtimes_r\Gamma)$.  The assembly map is now defined by
$$
\mu[\mathcal{E},F,\beta,\phi]:=[p_Y]\bigotimes_{C_0(Y)\rtimes_r\Gamma}j^\Gamma_r[\mathcal{E},F,\beta,\phi]\in KK_*(\C,A\rtimes_r\Gamma).
$$
where `$\otimes_{C_0(Y)\rtimes_r\Gamma}$' denotes Kasparov product over the $C^*$-algebra $C_0(Y)\rtimes_r\Gamma$.  More explicitly, it is straightforward to check that this class is represented by the Kasparov cycle 
$$
(\widetilde{p_Y}(\mathcal{E}\rtimes \Gamma),\widetilde{p_Y}\widetilde{F}\widetilde{p_Y},\iota)
$$ 
for $KK_*(\C,A\rtimes_r\Gamma)$, where $\iota:\C\to \mathcal{L}(\mathcal{E}\rtimes\Gamma)$ is the unital representation given by $z\mapsto zp_Y$ (in what follows will usually omit the `$\widetilde{\phi}$' where this is unlikely to cause confusion).  Hence 
$$
\mu[\mathcal{E},F,\beta,\phi]=[p_Y(\mathcal{E}\rtimes \Gamma),p_Y\widetilde{F}p_Y,\iota].
$$  

In order to analyze this cycle, it will be extremely convenient to introduce a new Hilbert $A\rtimes_r\Gamma$ modules as follows; the following discussion is inspired by, but a little different from \cite[Lemmas 2.1, 2.2, 2.3 and 3.4]{Roe:2002nx}.  Let $\ell^2_0(\Gamma)$ denote the subspace of finitely supported functions in $\ell^2(\Gamma)$.  Write `$\odot$' for the uncompleted tensor product over $\C$, and define 
$$
\mathcal{E}_0:=H_Y\odot A\odot H\odot \ell^2_0(\Gamma),
$$
which is a dense subspace of $\e$.  Equip $\mathcal{E}_0$ with the restriction of the $\Gamma$-action $\epsilon$ on $\e$; symbolically, this is given by 
$$
\epsilon_g(\xi\odot a\odot \eta\odot\delta_h):=u_g\xi\odot \alpha_g(a)\odot \eta\odot \delta_{gh}.
$$
Provisionally define a new inner product on $\mathcal{E}_0$ with values in $C_c(\Gamma,A)\subseteq A\rtimes_r\Gamma$ by the formula
$$
\langle e_1,e_2\rangle_{\eg}(g):=\langle e_1,\epsilon_g(e_2)\rangle_{\e},
$$
and a right action of $C_c(\Gamma,A)$ by 
$$
e\cdot b:=\sum_{g\in \Gamma}\epsilon_{g^{-1}}(e\cdot b(g)),
$$
where the product `$e\cdot b(g)$' on the right refers to the $A$-module structure of $\e$.  Define finally a linear map
\begin{equation}\label{u map}
U:\mathcal{E}_0\to \e\rtimes\Gamma,\quad (Ue)(g):=c\cdot \epsilon_g(e).
\end{equation}
A straightforward computation that we leave to the reader shows that for any $e_1,e_2\in \mathcal{E}_0$, we have 
\begin{equation}\label{u isom}
\langle Ue_1,Ue_2\rangle_{\e\rtimes\Gamma}=\langle e_1,e_2\rangle_{\eg}
\end{equation}
and moreover for all $e\in\mathcal{E}_0$ and $b\in C_c(\Gamma,A)$ we have $(Ue)\cdot b=U(e\cdot b)$.

It follows from this that the form $\langle~,~\rangle_{\eg}$ is positive semi-definite, and thus simultaneous completion as discussed in \cite[pages 4-5]{Lance:1995ys} gives rise to a Hilbert $A\rtimes_r\Gamma$-module $\eg$.  Moreover, the map in line \eqref{u map} above extends to an isometric inclusion
$$
U:\eg\to \e\rtimes \Gamma.
$$

\begin{lemma}\label{u lem 2}
The map $U:\eg\to \e\rtimes\Gamma$ is an adjointable isometry, with image exactly equal to $p_Y\cdot (\e\rtimes \Gamma)$.
\end{lemma}

\begin{proof}
An elementary computation shows that the adjoint of $U$ is given for $e\in \e\rtimes \Gamma$ by the formula
$$
U^*e=\sum_{g\in \Gamma}\epsilon_{g^{-1}}(c\cdot e(g));
$$
combined with the formula in line \eqref{u isom}, we now have that $U$ is an adjointable isometry.  Computing, for any $e\in \e\rtimes\Gamma$ and $g\in \Gamma$
\begin{align*}
(UU^*e)(g)=c\cdot \epsilon_g(U^*e) & =c\cdot \epsilon_g\Big(\sum_{h\in \Gamma}\epsilon_{h^{-1}}(c\cdot e(h))\Big) \\ 
& = \sum_{h\in \Gamma}c\gamma_{gh^{-1}}(c)\cdot e(h).
\end{align*}
Making the change of variables $k=gh^{-1}$, this becomes
$$
UU^*e(g)=\sum_{k\in \Gamma}c\gamma_k(c)e(k^{-1}g)=(p_Y\cdot e)(g).
$$
In other words, the range projection of $U$ is $p_Y$, which completes the proof.
\end{proof}

Now, let $E:A\rtimes_r\Gamma\to A$ denote the faithful conditional expectation defined by $b\mapsto b(e)$, where $e$ here denotes the identity element of $\Gamma$.  Following the discussion in \cite[pages 57-58]{Lance:1995ys}, this conditional expectation gives rise to a `localization' Hilbert $A$-module $\mathcal{E}_{A\rtimes_r\Gamma,E}$ defined as the separated completion of $\eg$ for the $A$-valued inner product defined by
$$
\langle e_1,e_2\rangle_{\mathcal{E}_{A\rtimes_r\Gamma,E}}:=E(\langle e_1,e_2\rangle_{\eg}).
$$
Moreover, there is a $*$-representation 
$$
\pi_E:\mathcal{L}(\eg)\to \mathcal{L}(\mathcal{E}_{A\rtimes_r\Gamma,E}).
$$
defined on the dense subspace of $\mathcal{E}_{A\rtimes_r\Gamma,E}$ defined as the image of $\eg$ by the formula $\pi_E(T)\cdot e=T\cdot e$; as in our case $E$ is faithful, $\pi_E$ is isometric.

\begin{lemma}\label{local lem}
The localization $\mathcal{E}_{A\rtimes_r\Gamma,E}$ identifies naturally with $\e$: more precisely, on the dense subspace of these $A$-modules defined by $\mathcal{E}_0$, the inner products agree. 

Moreover, having made this identification, the $*$-representation $\pi_E$ takes $\mathcal{K}(\eg)$ onto $C^*$.  
\end{lemma}

\begin{proof}
The first part is clear from the formulas involved.  For the second part, recall first that $\mathcal{K}(\eg)$ is generated by operators of the form
$$
\theta_{e_1,e_2}:e\mapsto e_1\langle e_2,e\rangle_{\eg},
$$
where $e_1,e_2,e$ are in $\eg$ (or just in the dense subspace $\mathcal{E}_0$).  Computing,
$$
\theta_{e_1,e_2}(e)=\sum_{g\in \Gamma}\epsilon_{g^{-1}}(e_1\langle e_2,e_2\rangle_{\eg}(g))=\sum_{g\in \Gamma}\epsilon_{g^{-1}}(e_1\langle e_2,\epsilon_g(e)\rangle_{\e}).
$$
Now, specialize to the case where $e_i=\xi_i\odot a_i\odot \eta_i\odot \delta_{h_i}$ for $i\in \{1,2\}$, and $e=\xi\odot a\odot \eta\odot \delta_{h}$ are all given by elementary tensors, which we may regard equally as elements of $\e$.  The first part combined with the above computation then says that 
\begin{align*}
\pi_E(\theta_{e_1,e_2})e & =\sum_{g\in \Gamma} \epsilon_{g^{-1}}(e_1\langle e_2,\epsilon_g(e)\rangle_{\e}) \\ & =\sum_{g\in \Gamma}\langle \eta_2,\eta\rangle\langle \xi_1,u_g\xi_2\rangle \langle \delta_{h_1},\delta_{gh}\rangle u_{g^{-1}}\xi_1\otimes \alpha_{g^{-1}}(a_1a_2^*)a\otimes \eta\otimes \delta_{g^{-1}h_1}.
\end{align*}
It is straightforward to check that the operator $\pi_E(\theta_{e_1,e_2})$ is in $C^*$, and moreover that linear combinations of such operators are dense in $C^*$, completing the proof.
\end{proof}

At this point, we have that $\eg$ is an $A\rtimes_r\Gamma$-module, and that the compact operators on it identify naturally with $C^*$.  To complete the proof that $A\rtimes_r\Gamma$ is Morita equivalent with $C^*$, it will suffice to show that $\eg$ is full.  For the sake of completeness, as well as to ease the subsequent analysis, the next lemma gives a more precise statement.

To state it, define an $C_c(\Gamma,A)$-valued inner product on $\mathcal{E}_0$ by the formula 
$$
\langle \xi_1\odot a_1\odot \eta_1\odot\delta_{h_1}~,~\xi_2\odot a_2\odot \eta_2\odot \delta_{h_2}\rangle(g):=\langle \xi_1,\xi_2\rangle \langle \eta_1,\eta_2\rangle \langle \delta_g,\delta_{h_1^{-1}h_2} \rangle \alpha_{h_1^{-1}}(a_1^*a_2).
$$
In other words, identifying $\ell^2_0(\Gamma)\odot A$ with $C_c(\Gamma,A)$ in the natural way, this is $H_Y\odot H\odot C_c(\Gamma,A)$ with its natural $C_c(\Gamma,A)$-valued inner product.  Thus it is positive definite, and completion gives rise to the standard $A\rtimes_r\Gamma$-module $H_Y\otimes H\otimes A\rtimes_{r}\Gamma$.

\begin{lemma}\label{v isom}
The map
$$
V:\mathcal{E}_0\to \mathcal{E}_0,\quad \xi\odot a\odot\eta\odot \delta_h\mapsto u_{h^{-1}}\xi\odot \alpha_{h^{-1}}(a)\odot \eta\odot \delta_{h^{-1}}
$$
extends to an isometric isomorphism
$$
V:\eg\to H_Y\otimes H\otimes A\rtimes_r\Gamma.
$$
In particular, conjugation by $V$ induces an isomorphism
$$
\mathcal{K}(\eg)\cong \mathcal{K}(H_Y\otimes H)\otimes A\rtimes_r\Gamma,
$$
and combining with Lemma \ref{local lem}, $C^*\cong \mathcal{K}\otimes A\rtimes_r\Gamma$.
\end{lemma}

\begin{proof}
Computing, 
\begin{align*}
\langle V(\xi_1&\odot a_1\odot\eta_1\odot \delta_{h_1}) ~,~V(\xi_2\odot a_2\odot\eta_2\odot \delta_{h_2})\rangle_{H_Y\otimes H\otimes A\rtimes_r\Gamma}(g) \\ 
& =\langle u_{h^{-1}}\xi,u_{h_2}^{-1}\xi_2\rangle \langle \eta_1,\eta_2\rangle \langle \delta_g,\delta_{h_1h_2^{-1}}\rangle \alpha_{h_1}(\alpha_{h_1^{-1}}(a_1^*)\alpha_{h_2}(a_2)) \\
& = \langle \xi,u_{h_1h_2}^{-1}\xi_2\rangle \langle \eta_1,\eta_2\rangle \langle \delta_g,\delta_{h_1h_2^{-1}}\rangle a_1^*\alpha_{h_1h_2^{-1}}(a_2) \\
& =\langle \xi,u_{g}^{-1}\xi_2\rangle \langle \eta_1,\eta_2\rangle \langle \delta_{h_1},\delta_{gh_2}\rangle a_1^*\alpha_{g}(a_2) \\ 
& =
 \langle \xi_1\odot a_1\odot\eta_1\odot \delta_{h_1}~,~\xi_2\odot a_2\odot\eta_2\odot \delta_{h_2}\rangle_{\eg}(g).
\end{align*}
Hence $V$ extends to an isometry from $\eg$ into $H_Y\otimes H\otimes A\rtimes_r\Gamma$.  A standard computation shows that $V$ is adjointable, with adjoint given on $\mathcal{E}_0$ by the same formula as for $V$, i.e.\
$$
V^*(\xi\odot a\odot \eta\odot \delta_h)=u_{h^{-1}}\xi\odot \alpha_{h^{-1}}(a)\odot \eta\odot \delta_{h^{-1}}.
$$
Clearly from these formulas $V$ has dense image, and thus extends to a unitary isomorphism as claimed.

To complete the statement about the compact operators, note that we now have 
\begin{align*}
\mathcal{K}(\eg) &\cong \mathcal{K}(H_Y\otimes H\otimes A\rtimes_r\Gamma)\cong \mathcal{K}(H_Y\otimes H)\otimes \mathcal{K}(A\rtimes_r\Gamma)\\ &\cong \mathcal{K}(H_Y\otimes H)\otimes A\rtimes_r\Gamma,
\end{align*}
where the first isomorphism is conjugation by $V$, the second is a standard general isomorphism for external tensor products of Hilbert modules discussed in \cite[page 37]{Lance:1995ys}, and the third is the standard identification $\mathcal{K}(B)\cong B$ for any $C^*$-algebra considered as a Hilbert module over itself as discussed in \cite[page 10]{Lance:1995ys}.
\end{proof}

We now go back to commutativity of the top square of diagram \eqref{pas diag}.  Filling in some more details, the top square in Diagram \eqref{pas diag} looks as follows.
$$
\xymatrix{ KK_1^\Gamma(C_0(Y),A) \ar[r] \ar[ddd]_-\delta^-\cong & KK_1(\C,A\rtimes_r\Gamma) \ar[r]_-\cong^-\kappa & K_1(A\rtimes_r\Gamma) \ar[d] \\
& & K_1((A\rtimes_r\Gamma)\otimes \mathcal{K}) \ar[d] \ar[u]^-\cong_-{s} \\
& & K_1(\mathcal{K}(\eg)) \ar[u]^-\cong_-{\text{ad}_V} \ar[d]_-\cong^-{\pi_E}  \\
K_0(D^*/C^*) \ar[rr]_-\partial & & K_1(C^*) ~~,}
$$
where in the above:
\begin{enumerate}[(i)]
\item the map labelled $\delta$ is the Paschke duality isomorphism of Corollary \ref{pas dual cor};
\item the map labelled $\partial$ is the standard boundary map in $K$-theory;
\item the composition of the top two horizontal arrows is the the Baum-Connes assembly map $\mu$ (we have explicitly included the isomorphism $\kappa$);
\item the map labelled $\pi_E$ is the map on $K$-theory induced by the isomorphism of Lemma \ref{local lem};
\item the map labelled $\text{ad}_V$ is the map on $K$-theory induced by conjugation by the unitary isomorphism of Lemma \ref{v isom};
\item the map labelled $s$ is the stabilization isomorphism in $K$-theory.
\end{enumerate}

Consider now what happens to a class in $KK_1^\Gamma(C_0(Y),A)$ as it goes around this diagram.  Using Lemma \ref{kk rep}, we may assume our class is of the form $[\e,F,\epsilon,\pi]$, where $F$ is in $D^*$.  As discussed above, the assembly map $\mu$ along the top row of diagram \eqref{pas diag} takes this class to 
$$
[p_Y\cdot (\e\rtimes \Gamma),p_Y\widetilde{F}p_Y,\iota]\in KK_1(\C,A\rtimes_r\Gamma)
$$
where $\iota$ is the unit representation of $\C$.  Using Lemma \ref{u lem 2}, this class is the same as $[\eg,U^*\widetilde{F}U,\iota]$.  Lemma \ref{v isom} implies that $\eg$ is actually a standard module over $A\rtimes_r\Gamma$, and thus one of the standard formulations of the isomorphism between $KK_*(\C,B)$ and $K_*(B)$ (see \cite[17.5.4 -- 17.5.6]{Blackadar:1998yq}) says that this class corresponds to the image of the projection $\frac{1}{2}(1+U^*\widetilde{F}U)$ under the composition 
$$
\partial:K_0(\mathcal{L}(\eg)/\mathcal{K}(\eg))\to K_1(\mathcal{K}(\eg)) \to K_1(A\rtimes_r\Gamma),
$$
of the $K$-theory boundary map and the combination of the isomorphism $\text{ad}_V$ and the stabilization isomorphism.  On the other hand, going around the square to the bottom right corner in the other direction, our class $[\e,F,\epsilon,\pi]$ goes to the image of the projection $\frac{1}{2}(1+F)$ in $D^*/C^*$ under the boundary map
$$
\partial:K_0(D^*/C^*)\to K_1(C^*).
$$
Consider then the commutative diagram of boundary maps
$$
\xymatrix{ K_0(\mathcal{L}(\eg)/\mathcal{K}(\eg)) \ar[r]^-\partial & K_1(\mathcal{K}(\eg)) \\ 
K_0(D^*/C^*)\ar[r]_-\partial \ar[u]^-{\pi_E^{-1}} &  K_1(C^*) \ar[u]_-{\pi_E^{-1}}  },
$$
where the vertical maps are induced by the inverse of $\pi_E$ restricted to its image.  To complete the proof, the discussion above implies that it will be enough to show that the projections
$$
\frac{1}{2}(1+U^*\widetilde{F}U) \quad \text{and} \quad \pi_E^{-1}(\frac{1}{2}(1+F))
$$
in $\mathcal{L}(\eg)/\mathcal{K}(\eg)$ are the same.  For this the following lemma suffices, so it completes our analysis of the top sqaure.

\begin{lemma}\label{pert lem}
For any $F\in D^*$, 
$$
\pi_E(U^*\widetilde{F}U)-F
$$
is in $C^*$.
\end{lemma}

\begin{proof}
We compute what the operator $\pi_E(U^*\widetilde{F}U)$ does on an element $e$ of $\mathcal{E}_0\subseteq \e$.  
\begin{align*}
\pi_E(U^*\widetilde{F}U)e & =U^*\widetilde{F}Ue=\sum_{g\in \Gamma}\epsilon_{g^{-1}}(c(\widetilde{F}Ue)(g))=\sum_{g\in \Gamma}\epsilon_{g^{-1}}(cF(Ue)(g)) \\ & =\sum_{g\in \Gamma}\epsilon_{g^{-1}}(cFc\epsilon_g(e))=\sum_{g\in \Gamma}\gamma_{g^{-1}}(c)F\gamma_{g^{-1}}(c)e,
\end{align*}
where the last inequality used $\Gamma$-invariance of $F$.  Hence 
$$
\pi_E(U^*\widetilde{F}U)-F=\sum_{g\in \Gamma}\gamma_{g^{-1}}(c)F\gamma_{g^{-1}}(c)-F.
$$
To see that this operator is in $C^*$, we must show that it is $\Gamma$-invariant, has finite $\Gamma$-propagation, and is $A$-locally compact.  The first two of these are clear, as they hold for each of the two terms individually.  To see that the operator is $A$ locally compact, let $f$ be an element of $C_c(Y)$.  Let $S:=\{g\in \Gamma\mid f\cdot \gamma_{g^{-1}}(c)\neq0\}$, which is finite by properness of the action, and compact support of $f$ and $c$.  Then we have
\begin{align*}
f\cdot \Big(\sum_{g\in \Gamma}\gamma_{g^{-1}}(c)F\gamma_{g^{-1}}(c)-F\Big) &= f\cdot \Big(\sum_{g\in S}\gamma_{g^{-1}}(c)F\gamma_{g^{-1}}(c)-\sum_{g\in S} \gamma_{g^{-1}}(c^2)F\Big) \\
& =f\cdot \Big(\sum_{g\in S}\gamma_{g^{-1}}(c)[F,\gamma_{g^{-1}}(c)]\Big),
\end{align*}
where the first equality uses that $\sum\gamma_{g^{-1}}(c^2)=1$.  The sum in parentheses is a finite sum of operators in $\mathcal{K}(\e)$, so we are done.
\end{proof}

\subsection*{Localization algebra square}

In studying the bottom square, it will help to introduce some auxiliary $C^*$-algebras.   For a $C^*$-algebra $B$, let $TB$ denote the $C^*$-algebra of all bounded, uniformly continuous functions from $[0,\infty)$ to $B$.  We then have a commutative diagram of short exact sequences of $C^*$-algebras.  
\begin{equation}\label{3by3}
\xymatrix{ 0 \ar[r] & C^* \ar[r] & D^* \ar[r] & D^* / C^* \ar[r] & 0 \\
0 \ar[r] & TC^* \ar[r] \ar[u] & TD^* \ar[r] \ar[u] & TD^* / TC^* \ar[r] \ar[u] & 0 \\
0 \ar[r] & C_L^* \ar[r]  \ar[u] & D_L^* \ar[r] \ar[u]  & D_L^* / C_L^* \ar[r]  \ar[u] & 0 }
\end{equation}
Here the upper three vertical arrows are all induced by evaluation-at-zero maps, while the lower three vertical arrows are all induced by simply forgetting the condition on metric propagation in the definition of $D_L^*$ and $C^*_L$.  As already mentioned, the bottom square in Diagram \eqref{pas diag} is induced by the boundary maps from the top and bottom sequences in Diagram \eqref{3by3}, and thus automatically commutes, so it remains to show that the maps labelled (iii) and (iv) in Diagram \eqref{pas diag} are isomorphisms.  Indeed, that (iii) is an isomorphism follows from Lemmas \ref{top} and \ref{c*-iso} below, while isomorphism of (iv) is Lemma \ref{dl die} below, so these lemmas complete our analysis of the bottom square.

\begin{lemma}\label{top}
(Compare \cite[Proposition 3.6]{Qiao:2010fk}).  The upper three vertical maps in Diagram \eqref{3by3} induce isomorphisms on $K$-theory.
\end{lemma}

\begin{proof}
Using the six term exact sequence and the five lemma, it suffices to show the two maps on the left induce isomorphisms on $K$-theory.  For this it suffices to show the following: if $B$ is a $C^*$-algebra which has a stability structure in the sense of Definition \ref{stab str}, then the evaluation-at-zero map $TB\to B$ induces an isomorphism on $K$-theory.  Using the six term exact sequence again, it suffices to show that if $B$ is any $C^*$-algebra with a stability structure, then 
$$
T_0B:=\{f\in TB\mid f(0)=0\}
$$
has trivial $K$-theory.  This is what we will now do.

Let $(u_n)$ be the unitaries in the definition of a stability structure.  For an element $b\in T_0(B)$, extend $b$ to a function $b:\R\to B$ by setting $b(t)=0$ for all $t<0$.   For each $n$, define an inclusion 
$$
\mu_n:T_0B\to T_0B,\quad (\mu_nb)(t)=b(t-n).
$$
Then each $\mu_n$ is a $*$-homomorphism.  Moreover, the map
$$
\mu:T_0B\to T_0B, \quad b\mapsto \sum_{n=0}^\infty u_n \mu_n(b)u_n^*
$$
is a $*$-homomorphism, as for any fixed $t$, all but finitely many of the functions $\mu_n(b)$ take the value zero at $t$.  Conjugating by the isometry 
$$
v=\sum_{n=0}^\infty u_{n+1}u_n^*
$$
shows that $\mu$ induces the same map on $K$-theory as the map $\mu':T_0B\to T_0B$ defined by 
$$
\mu'(b)=\sum_{n=1}^\infty u_n\mu_{n-1}(b)u_n^*,
$$
and applying a shift homotopy at each `level' indexed by $n$ (plus using uniform continuity of $b$) shows that $\mu'$ induces the same map on $K$-theory as $\mu^{+1}:T_0B\to T_0B$ defined by
$$
\mu^{+1}(b):=\sum_{n=1}^\infty u_n\mu_n u_n^*.
$$
Then we clearly have that 
$$
\mu=\text{ad}_{u_0}\circ \mu_0+\mu^{+1} 
$$
as $*$-homomorphisms (the right hand side is a $*$-homomorphism as $\mu_0$ and $\mu^+1$ have orthogonal images: compare Lemma \ref{orth} above).  Note that $\text{ad}_{u_0}$ is just conjugation by an isometry in the multiplier algebra of $T_0B$, and thus defines the identity on $K$-theory (see Lemma \ref{isom con} above).  Hence passing to induced maps on $K$-theory gives
$$
\mu_*=(\text{ad}_{u_0})_*\circ (\mu_0)_*+\mu^{+1}_*=\text{id}+\mu_*,
$$
and cancelling $\mu_*$ gives that the identity map is zero, which gives $K_*(T_0(B))=0$ as claimed.
\end{proof}

\begin{lemma}\label{c*-iso}
(Compare \cite[Proposition 2.3]{Qiao:2010fk}).  With notation as in Diagram \eqref{3by3} above, the map $D_L^*/C_L^*\to TD^*/TC^*$ is an isomorphism of $C^*$-algebras.
\end{lemma}

\begin{proof}
We define an inverse map.  As the action of $\Gamma$ on $Y$ is proper, it is not too difficult to see that for each $n$, there exists a partition of unity $\{\phi_{i,n}:Y\to [0,1]\}_{i\in I_n}$, which is $\Gamma$-invariant, such that each $\phi_{i,n}$ has compact support of diameter at most $1/n$, and such that $\sum_{i\in I_n}\phi_{i,n}^2(y)=1$ for all $y\in Y$.  Define a map $\Phi:TD^*\to D_L^*$ by stipulating that when $t\in [n,n+1]$, 
$$
\Phi(a)(t):=(n+1-t)\sum_{i\in I_{n+1}}\phi_{i,n+1}a(t)\phi_{i,n+1} +(t-n)\sum_{i\in I_{n+2}}\phi_{i,n+2}a(t)\phi_{i,n+2}.
$$
Then it is not too difficult to see that $\Phi$ is a well-defined complete contraction, that $\Phi$ descends to a well-defined $*$-homomorphism on the quotients, and that $\Phi(a(t))-a(t)\in TC^*$ for all $a$ and all $t$ (compare \cite[Lemma 2.2]{Qiao:2010fk}).  The result follows from this.
\end{proof}

\begin{lemma}\label{dl die}
(Compare \cite[Proposition 3.5]{Qiao:2010fk}).  The $C^*$-algebra $D_L^*$ has trivial $K$-theory.  
\end{lemma}

\begin{proof}
Again, we use the stability structure $(u_n)_{n=0}^\infty$ on $D^*_L$ coming from a decomposition of the `auxiliary Hilbert space' $H$ into countably many infinite dimensional summands.  For each $n$, define a $*$-homomorphism $\mu_n:D_L^*\to D_L^*$ by the formula
$$
(\mu_n(a))(t)=a(t+n).
$$
Then we may define $\mu:D^*_L\to M(D^*_L)$ by the formula 
$$
\mu(a):=\sum_{n=0}^\infty u_n\mu_n(a)u_n^*.
$$
Note however, that the image actually lands in $D^*_L$, not its multiplier algebra: the point is that $[\mu_n(a)(t),f]\to 0$ in norm as $n\to\infty$ for any $f\in C_0(Y)$ and $t\in [0,\infty]$, using the propagation condition (compare the proof of \cite[Proposition 5.18]{Roe:1993lq}).  A combination of conjugation by an isometry and a homotopy quite analogous to the argument of Lemma \ref{top} shows that $\mu$ induces the same map on $K$-theory as $\mu_{+1}$, where the latter is defined by the same formula, except that the sum starts at $n=1$.  Finally, we have that as maps on $K$-theory
$$
\mu_*=(\text{ad}_{u_0})_*\circ (\mu_0)_*+\mu^{+1}_*=\text{id}+\mu_*,
$$
whence the identity induces the zero map on $K$-theory, and we are done.
\end{proof}

\bibliographystyle{abbrv}

\bibliography{Generalbib}

\end{document}